\documentclass[12pt,reqno]{amsart}
\usepackage{amsmath,a4wide}
\usepackage{xfrac}
\usepackage{stmaryrd,mathrsfs,bm,amsthm,mathtools,yfonts,amssymb,color}
\usepackage{xcolor}

\begingroup
\newtheorem{theorem}{Theorem}[section]
\newtheorem{lemma}[theorem]{Lemma}
\newtheorem{proposition}[theorem]{Proposition}

\newtheorem{ipotesi}[theorem]{Assumption}
\endgroup

\theoremstyle{definition}
\newtheorem{definition}[theorem]{Definition}
\newtheorem{remark}[theorem]{Remark}

\numberwithin{equation}{section}

\newcommand\supp{{\rm spt}}
\newcommand{\sing}{{\rm Sing}}

\newcommand\res{\mathop{\hbox{\vrule height 7pt width .3pt depth 0pt
\vrule height .3pt width 5pt depth 0pt}}\nolimits}

\newcommand{\bG}{\mathbf{G}}
\newcommand{\cH}{{\mathcal{H}}}

\newcommand{\p}{{\mathbf{p}}}

\newcommand{\cC}{{\mathcal{C}}}

\newcommand{\bI}{{\bm{I}}}
\newcommand{\bE}{{\mathbf{E}}}

\newcommand{\bB}{{\mathbf{B}}}
\newcommand{\bC}{{\mathbf{C}}}

\newcommand{\be}{{\mathbf{e}}}
\newcommand{\bd}{{\mathbf{d}}}

\newcommand\N{{\mathbb N}}

\newcommand\s{{\mathbb S}}
\newcommand\C{{\mathbb C}}
\newcommand\R{{\mathbb R}}

\newcommand{\eps}{{\varepsilon}}
\newcommand{\bA}{\mathbf{A}}
\newcommand{\bmax}{\mathbf{m}}

\def\Xint#1{\mathchoice
{\XXint\displaystyle\textstyle{#1}}%
{\XXint\textstyle\scriptstyle{#1}}%
{\XXint\scriptstyle\scriptscriptstyle{#1}}%
{\XXint\scriptscriptstyle\scriptscriptstyle{#1}}%
\!\int}
\def\XXint#1#2#3{{\setbox0=\hbox{$#1{#2#3}{\int}$ }
\vcenter{\hbox{$#2#3$ }}\kern-.6\wd0}}
\def\mint{\Xint-}

\newcommand{\Lip}{{\rm {Lip}}}

\newcommand{\dist}{{\rm {dist}}}
\newcommand{\dv}{{\text {div}}}

\newcommand\weak{{\rightharpoonup}\,}

\newcommand\Id{{\rm Id}\,}

\newcommand{\cA}{{\mathcal{A}}}

\newcommand{\cG}{{\mathcal{G}}}
\newcommand{\cE}{{\mathcal{E}}}
\newcommand{\cF}{{\mathcal{F}}}

\renewcommand{\cC}{{\mathcal{C}}}

\newcommand{\cQ}{{\mathcal{Q}}}

\newcommand\Pe{{\mathscr P}}

\newcommand{\mass}{{\mathbf{M}}}
\renewcommand\d{\mathbf{d}}
\newcommand\e{\mathbf{e}}
\newcommand{\Om}{\Omega}
\def\I#1{{\mathcal{A}}_{#1}}

\newcommand{\Iqs}{{\mathcal{A}}_Q(\R^{n})}
\newcommand{\Iq}{{\mathcal{A}}_Q}
\def\a#1{\left\llbracket{#1}\right\rrbracket}

\newcommand{\norm}[2]{\left\|#1\right\|_{#2}}
\newcommand{\ra}{\right\rangle}
\newcommand{\la}{\left\langle}
\newcommand{\D}{\textup{Dir}}
\newcommand{\de}{\partial}
\newcommand{\xii}{{\bm{\xi}}}
\newcommand{\ro}{{\bm{\rho}}}
\newcommand{\g}{{g'}}
\newcommand{\f}{{f'}}

\newcommand{\etaa}{{\bm{\eta}}}
\newcommand{\ph}{\varphi}

\newcommand{\lin}{{\text{lin}}}

\newcommand{\bq}{{\mathbf{q}}}
\newcommand\B{{\mathbf{B}}}
\newcommand{\bh}{\mathbf{h}}

\title[Gradient $L^p$ estimates]{Regularity of area minimizing currents I:\\ gradient $L^p$ estimates}
\author{Camillo De Lellis}
\author{Emanuele Spadaro}

\begin{document}

\begin{abstract}
In a series of papers, including the present one, we give a new, shorter proof of
Almgren's partial regularity theorem for area mi\-ni\-mi\-zing currents
in a Riemannian manifold, with a slight improvement on the regularity assumption for the latter.
This note establishes a new a priori estimate
on the excess measure of an area minimizing current, together with several statements
concerning approximations with Lipschitz multiple valued graphs.
Our new a priori estimate is an higher integrability type result, which has a counterpart
in the theory of Dir-minimizing multiple valued functions and plays a key role in estimating
the accuracy of the Lipschitz approximations. 
\end{abstract}

\maketitle

\section{Foreword: a new proof of Almgren's partial regularity}
In the present work we continue the investigations started in \cite{DS1, DS2}, which
together with the forthcoming papers \cite{DS4, DS5} lead to a proof of the following
theorem.

\begin{theorem}\label{t:Almgren_migliorato}
Let $\Sigma\subset \R^{m+n}$ be a $C^{3, \varepsilon_0}$ submanifold
for some $\varepsilon_0 >0$
and $T$ an $m$-dimensional area minimizing integral
current in $\Sigma$. Then, there is a closed set $\sing (T)$ of 
Hausdorff dimension at most $m-2$ such that $T$ is a $C^{3,\varepsilon_0}$
embedded submanifold in $\Sigma \setminus
(\supp (\partial T) \cup \sing (T))$.
\end{theorem}

Theorem \ref{t:Almgren_migliorato} was first proved by Almgren in his monumental
work \cite{Alm}, assuming slightly better regularity on $\Sigma$, namely
$\Sigma \in C^5$.
The improvement itself is therefore not so significant, but
our proof, besides being much shorter, introduces new ideas and
establishes several new results, which we hope will provide useful tools 
for further investigations in the area.
Indeed, although we still follow Almgren's program
and use many of his groundbreaking discoveries,
the main steps are achieved in a more
efficient way thanks to new estimates and techniques.
A striking example is the construction of the so-called
center manifold, which is by far
the most intricate part of Almgren's work and the least explored,
in spite of its importance: in this respect,
our construction in \cite{DS4} is considerably simpler and shorter than
\cite[Chapter~4]{Alm}, and establishes better results.

Some of our improvements are more transparent, although not
substantially simpler, when $\Sigma = \R^{m+n}$ and
in a book in preparation \cite{DS-book} we will provide a complete and
self-contained account of Theorem~\ref{t:Almgren_migliorato} under such assumption. 
Moreover, building on our understanding of the
various issues involved to the analysis of higher codimension singularities, 
we plan to tackle Chang's improvement \cite{Chang},
which shows that $\sing (T)$ consists
of isolated points when $m=2$.
His arguments rely on a center manifold construction
which does not match exactly the statements of \cite{Alm}
and it is not fully justified, but only briefly sketched in the appendix of \cite{Chang}. 
In \cite{DSS}, instead, we give a detailed, simple construction for such
center manifold and a complete proof of this refined regularity result.

An alternative route to Chang's result for $J$-holomorphic
currents in symplectic manifolds has been given recently in
\cite{RT1,RT2}.
The interest in the regularity theory for this class of area minimizing
2-dimensional currents
has been generated by the seminal paper of Taubes \cite{Taubes}
on the equivalence between Gromov and Seiberg-Witten invariants,
where it plays an important role.
Moreover, the papers \cite{RT1,RT2} have stimulated a lot of activity in the area,
cf., for example, \cite{BeRi,PuRi,Ri04bis,Ri04}.
In \cite{BeRi} Bellettini and Rivi\`ere proved that, when $T$ is a
special Lagrangian cone in $\R^6$, $\sing (T)$ consists of finitely many
half-lines meeting at the origin. This is, to our knowledge,
the only result of its type not covered by the Almgren-Chang works.
We believe that the Bellettini-Rivi\`ere regularity theorem can be extended to general
$3$-dimensional area minimizing cones in any space dimension, combining the techniques 
developed in \cite{DS1}--\cite{DSS}. 
Most of the proofs in
\cite{BeRi,PuRi, Ri04bis,  Ri04, RT1, RT2, Taubes}
take advantage of two specific assumptions,
the underlying almost complex structure and the $2$-dimensionality
of the objects of study. Nonetheless these works have had
a profound influence on our research.

\subsection{A blow-up proof: a very brief overview}
In the rest of this foreword we will give a rough outline of the proof of Theorem \ref{t:Almgren_migliorato},
highlighting the contents of this note and the way it merges with
its companion papers \cite{DS4, DS5}, while comparing them to \cite{Alm}.
Our discussion will be based on a well-known class of examples
for which the statement of Theorem~\ref{t:Almgren_migliorato} is optimal, namely 
singular holomorphic curve of $\C^2$.
As it was first observed by Federer (cf.~\cite[5.4.19]{Fed}),
the integral currents induced 
by holomorphic subvarieties of $\C^n$ (with their natural orientation)
are area minimizing.

We denote by $D_Q (T)$ the set of points in $\supp (T)\setminus \supp (\partial T)$
where the density of  a current $T$ equals the natural number $Q\geq 1$.
One first pioneering contribution by Almgren is an elementary, but very clever,
generalization of Federer's reduction argument, which has been widely used
in several contexts (see \cite[Theorem 35.3]{Sim} and \cite{White97}).
This argument implies that, if $T$ is area minimizing, then $\supp (T) \setminus (\cup_{Q} D_Q (T) \cup \supp (\partial T))$ has Hausdorff
dimension at most $m-3$.
Thus, to prove Theorem \ref{t:Almgren_migliorato} it suffices to show that the Hausdorff dimension of $\sing_Q (T) := \sing (T)\cap D_Q (T)$ is at most $m-2$. 
Since the ``classical'' regularity theory ensures that $T$ is a $C^{1,\alpha}$ submanifold in the neighborhood of any point $x\in D_1 (T)$, it is natural
to argue by induction on $Q$. 

Let us therefore consider the case $Q=2$ and a point $x\in D_2 (T)$. By the monotonicity formula, in some neighborhood $U$ of $x$,
$\|T\|$-almost all points have density $1$ or $2$.
If the points of density $1$ are a set of $\|T\|$-measure zero,
by the classical regularity theory
$x$ is a regular point for $T$. So any $x\in \sing_2 (T)$ must
be surrounded by many points of density $1$,
as it is, for instance, for the complex curve $\{ z^2 = w^3\} \subset \C^2$ at $x=0$.
On the other hand, in such an example $0$ is an isolated singularity,
whereas, if $T$ were to contradict Theorem~\ref{t:Almgren_migliorato},
by standard measure theoretic arguments
there would be a point $x\in \sing_2 (T)$ surrounded by many points of density $2$. From now on
we argue by contradiction and assume that this happens for some area minimizing $T$ at the point $0\in D_2 (T)$. 
Moreover, by known facts in geometric measure theory, we can reduce
the contradiction to the case that, 
for a suitable sequence of radii $r_k\downarrow 0$,
the homothetic rescalings of the current $T$ by a factor $1/r_k$ (from now on denoted by $T_k$) converge to a double copy of an 
$m$-dimensional plane, while at the same time $D_2 (T_k)$ remains rather large.

It was first recognized by De Giorgi that  the convergence of $T_k$ to a
{\em single copy} of a flat plane  implies that
$\supp (T_k)$ can be well approximated by the graph of Lipschitz functions
which are ``almost harmonic''.
However, the example $\{z^2 = w^3\}\subset \C^2$ shows that this is not
always the case if the limiting plane has higher multiplicity.
Motivated by this fact, Almgren undertook in \cite{Alm} the 
strikingly ambitious program of giving a rather complete existence
and regularity theory for 
{\em multiple valued} functions minimizing a suitable generalization
of the Dirichlet energy, called $\D$-minimizers.
The crowning achievement of this theory is that,
except for a closed set of codimension at most $2$,
$\D$-minimizers can be locally decomposed in classical (i.e.~single-valued) non-intersecting
harmonic sheets (possibly counted with multiplicity).
Such ``linear theory'' is developed in \cite[Chapter 2]{Alm} and revisited in our paper \cite{DS1}.
Moreover, it is complemented by several technical statements linking
the multiple valued graphs to the integral currents, a task
which is accomplished in \cite[Chapter 1]{Alm} by Almgren and in \cite{DS2} by us
(we refer to the introduction to our previous two papers \cite{DS1,DS2} for more
details). 

The guiding idea in the contradiction argument is to approximate
the currents $T_k$ with Lipschitz {\em $2$-valued} functions and, after a suitable
renormalization of their Dirichlet energy, show that they converge to
a $\D$-minimizer. If the limit inherits a large singular set from the currents
$T_k$, then it contradicts the linear regularity theory.
Obviously, this strategy requires suitable approximations of area minimizing
currents with multiple valued graphs, accomplished by Almgren in
\cite[Chapter~3]{Alm} and by us in the present paper.
If one follows our approach, the convergence of these approximations
to a $\D$-minimizer can be concluded in a rather direct way. However,
we cannot expect that such limit inherits the
singular set of the current.
For example, given the complex curve $\{(z,w) : (z-w^2)^2 = w^5\}\subset \C^2$,
any reasonable approximations of
homothetic rescalings of this algebraic variety in a neighborhood
of the origin converge to a double copy of the classical holomorphic
graph $\{(w, w^2) : w\in \C\}$,
which has lost the singularity at the origin.

In order to perform the blow-up argument, we then need to ``modulate
lower order regularities out''.
This is accomplished by the construction of a
center manifold (see \cite[Chapter~4]{Alm} and \cite{DS4}):
such an object is a regular $C^{3,\alpha}$
submanifold which is very close to the average of the sheets of the
current at any scale where the latter is ``very collapsed''.
The final blow-up argument is then carried over to a new sequence of
$2$-valued approximations of $T_k$, performed on the normal
bundles of the center manifolds (see \cite[Chapter 5]{Alm} and \cite{DS5}).
By a delicate unique continuation principle, based on a new monotonicity
formula discovered by Almgren, 
a suitable normalization of the latter approximations does
converge to a $\D$-minimizer which would be forced to
have a large singular set, reaching the desired contradiction.
This final step builds upon very delicate computations, which thus
require a lot of accuracy in the construction
of the center manifold, that in turn needs very good estimates on
the approximation results of this note.
Thus, unlike the two works \cite{DS1, DS2}, which can be considered separately,
the present paper and \cite{DS4,DS5} are intimately interconnected.

\subsection{Our contribution; or, what is new}
In their overall structure,
our five papers match bijectively the five chapters of \cite{Alm}.
Moreover, it is clear that the ultimate
reason for the success of the program is the very same prodigious and
celebrated discovery of Almgren: the monotonicity of the frequency
function and its astonishing robustness, which enters twice in the
plan: at the very beginning, in the linear regularity theory,
and at the end, in the convergence of the final approximations (cf.~\cite{DS1, DS5}).

So, what is new in our proof? Aside from finer details,
which are explained in the introductions
to each of our papers, there are some new contributions which
come at a higher level. 
Our investigations started with the idea that
the machinery developed in metric analysis and metric geometry
in the last 30 years could reduce the complexity of several arguments in Almgren's
program.
This is, indeed, the case at many levels in the two papers \cite{DS1,DS2}
and in this note. Approaching vast parts of Almgren's theory with these tools,
we not only get shorter and more transparent proofs, but often
also achieve stronger analytic estimates, which give a better starting point
for the PDE parts of the program. 
Moreover, as it often happens when ``abstract nonsense'' simplifies
preexisting mathematical theories, such machinery provides also a better insight
to the material of \cite{Alm}, as it highlights the important points in the proofs therein. 

However, this alone would not explain the shortness of our papers
compared to \cite[Chapters~3,4,5]{Alm}.
The other important reason is that
we also derive some fundamental, new ``hard'' estimates.
A primary example is the present paper, where the main a priori estimate is a new 
higher integrability result, which comes from
a Gehring-type argument and is inspired by a simple remark in the linear theory
(the higher integrability of gradients of $\D$-minimizers)
which to our knowledge is not observed in Almgren's monograph. Similar instances
are present in the papers \cite{DS4,DS5}, where some new quantities
and guiding principles are introduced (for instance, the ``modified frequency''
function in \cite{DS5} and the ``splitting-before-tilting''
principle in \cite{DS4}, inspired by \cite{Ri04}),
which probably lead to the improvement on the regularity assumptions
of the ambient manifold $\Sigma$.
In all these cases we provide more efficient tools compared to \cite{Alm} and
invoke more PDE theory at several levels, drawing connections with 
fairly classical concepts from other areas of analysis (such as maximal functions, Lipschitz truncations, elliptic systems, Sobolev capacity).  
Unfortunately we do not understand Almgren's arguments at a sufficiently
deep level to draw a fine parallel between our papers \cite{DS4,DS5}
and the last two chapters of his book, where the intricacy of
the arguments in \cite{Alm} is almost prohibitive. 
It remains the fact that our papers are much more accessible,
and we hope that in the near future our work will be used to penetrate 
further in the richness and beauty of Almgren's monograph and to go beyond
Theorem~\ref{t:Almgren_migliorato}.

\medskip

{\bf Acknowledgments.} 
The first author is deeply indebted to Tristan Rivi\`ere who
``infected him'' with the beauty of the problem. Both authors are also warmly
thankful to Giovanni Alberti, Bill Allard, Luigi Ambrosio and Bernd Kirchheim
not only for several enlightening conversations,
but also for their constant enthusiastic encouragement. 

Several other colleagues and friends have contributed
with important scientific conversations at some specific stage,
for which they will be acknowledged specifically in the 
various papers.
In particular, for this first one we are grateful 
to Stefano Bianchini, Sergio Conti, Matteo Focardi, Jonas Hirsch and Luca Spolaor
for very useful discussions and comments.

This work was carried over several years and the authors
wish to thank many institutions where they spent very productive visits,
namely: the University of Rome La Sapienza,
the Scuola Normale Superiore and the University of Pisa,
the Max Planck Institute for Mathematics in the Sciences and the University of Leipzig, 
the University of Z\"urich,
the SISSA in Trieste, the University of Warwick and, most
of all, Princeton University, which hosted the first author
during his last sabbatical.
We also acknowledge the support of the ERC grant RAM, ERC 306247.

We finally thank the anonymous referee
for his/her careful reading and very valuable suggestions, which contributed to a substantial improvement of the initial manuscript.

\section{Introduction}

\subsection{A priori gradient $L^p$ estimate}
In order to state the main results, we start specifying some assumptions,
which will hold throughout the paper. For the notation concerning submanifolds
$\Sigma\subset \R^{m+n}$ we refer to \cite[Section 1]{DS2}. With $\bB_r (p)$ and $B_r (x)$ we denote, respectively, the open ball with radius $r$ and center $p$ in $\mathbb R^{m+n}$ and the open ball with radius $r$ and center $x$ in $\mathbb R^m$.  $\bC_r (x)$ will always denote the cylinder $B_r (x)\times \mathbb R^n$ and the point $x$ will be omitted when it is the origin. In fact, by a slight abuse of notation, we will often treat the center $x$ as a point in $\mathbb R^{m+n}$, avoiding the correct, but more cumbersome, $(x,0)$.
Let $e_i$ be the unit vectors in the standard basis,
$\pi_0$ the (oriented) plane $\R^m\times \{0\}$ and $\vec \pi_0$
the $m$-vector $e_1\wedge \ldots \wedge e_m$ orienting it.
We denote by $\p$ and $\p^\perp$ the
orthogonal projections onto, respectively, $\pi_0$ and its orthogonal complement $\pi_0^\perp$. In some cases we need orthogonal projections onto other planes $\pi$
and their orthogonal complements $\pi^\perp$,
for which we use the notation $\p_\pi$ and $\p^\perp_\pi$. 
For what concerns integral currents we use the definitions and the notation of \cite{Sim}.

\begin{ipotesi}\label{ipotesi_base}
$\Sigma\subset\R^{m+n}$ is a $C^2$
submanifold of dimension $m + \bar n = m + n - l$, which is the graph of an entire
function $\Psi: \R^{m+\bar n}\to \R^l$ and satisfies the bounds
\begin{equation}\label{e:Sigma}
\|D \Psi\|_0 \leq c_0 \quad \mbox{and}\quad \bA := \|A_\Sigma\|_0
\leq c_0,
\end{equation}
where $c_0$ is a positive (small) dimensional constant.
$T$ is an integral current of dimension $m$ with bounded support contained in $\Sigma$ and which, for some open cylinder $\bC_{4r} (x)$ 
(with $r\leq 1$)
and some positive integer $Q$, satisfies
\begin{equation}\label{e:(H)}
\p_\sharp T = Q\a{B_{4r} (x)}\quad\mbox{and}\quad
\de T \res \bC_{4r} (x) =0\, .
\end{equation}

If we say that $T$ is {\em area minimizing} we then mean that it is area-minimizing in $\Sigma\cap
\bC_{4r} (x)$, namely that $\mass (T) \leq \mass (T + \partial S)$ for any integral $S$ with $\supp (S)\subset
\Sigma \cap \bC_{4r} (x)$.
\end{ipotesi}

\begin{definition}[Excess measure]\label{d:excess}
For a current $T$ as in Assumption \ref{ipotesi_base} we define the \textit{cylindrical excess} $\bE(T,\bC_{4r} (x))$,
the \textit{excess measure} $\e_T$ and its {\em density} $\bd_T$:
\begin{gather*}
\bE(T,\bC_r (x)):= \frac{\|T\| (\bC_r (x))}{\omega_m r^m} - Q ,\\
\e_T (A) := \|T\| (A\times\R^{n})  - Q\,|A| \qquad \text{for every Borel }\;A\subset B_r (x),\\
\bd_T(y) := \limsup_{s\to 0} \frac{\e_T (B_s (y))}{\omega_m\,s^m}= \limsup_{s\to 0} \bE (T,\bC_s (y)),
\end{gather*}
where $\omega_m$ is the measure of the $m$-dimensional unit ball
(the subscripts $_T$ will be omitted if clear from the context).
\end{definition}

Since $T$ has finite mass, the function $\bd$ is naturally an $L^1$ function. However,
we can show the following higher integrability estimate when $T$ is, in addition, area minimizing.
We call it a gradient $L^p$ estimate because we will show that $\bd$ coincides with the gradient of
an appropriate Lipschitz function on a large region.

\begin{theorem}[Gradient $L^p$ estimate]\label{t:higher1}
There exist constants $p_1 >1$ and $C, \eps_{10}>0$ (depending on $m,n,\bar n, Q$)
with the following property. Let $T$ be as in Assumption \ref{ipotesi_base} in the cylinder $\bC_4$.
If $T$ is area minimizing and $E=\bE (T,\bC_4)< \eps_{10}$, then
\begin{equation}\label{e:higher1}
\int_{\{\bd\leq1\}\cap B_2} \bd^{p_1} \leq C\, E^{p_1-1} \left(E + \bA^{2}\right).
\end{equation}
\end{theorem}

In the case $Q=1$ or $\bar{n}=1$, it follows from the classical regularity theory (essentially due to De Giorgi, cf.~\cite{DG}) that
$T$ is a $C^{1,\alpha}$ submanifold in $\bC_2$.
However, when $\min \{Q, \bar{n}\}\geq 2$, $T$ is not necessarily regular and
Theorem~\ref{t:higher1} gives in fact an \textit{a priori} regularity estimate: in this case
\eqref{e:higher1} cannot be improved
(except for optimizing the constants $p_1$, $C$
and $\eps_{10}$). Indeed, for $Q=m=2$, $\Sigma = \R^4$ and $p_1=2$, \eqref{e:higher1}
is false no matter how large $\eps_{10}^{-1}$ and $C$ are chosen (cf.~\cite[Section 6.2]{ICM}).

In order to prove Theorem \ref{t:higher1} we develop the following tools:
\begin{itemize}
\item[(a)] a general scheme to approximate integer
rectifiable currents with multiple valued functions, relying
heavily on the ``metric analysis'' of \cite{DS1} and on
a modified ``Jerrard--Soner'' BV estimate for the slicing of currents
(cf.~Proposition~\ref{p:max});
\item[(b)] a simple and robust harmonic approximation of area minimizing
currents with multiple valued functions (cf.~Theorem~\ref{t:o(E)});
\item[(c)] the higher integrability of the gradient of Dir-minimizing
multiple valued functions (cf.~Theorem~\ref{t:hig fct} --
see also \cite{Sp10} for a different proof and related results).
\end{itemize}

In turn, Theorem \ref{t:higher1} will be combined with (a) to achieve a very accurate  
approximation result for area minimizing current, stated in Theorem \ref{t:main}. This theorem and some corollaries
of our analysis play a fundamental role in the papers \cite{DS4,DS5} and,
as explained in
the Foreword, have a counterpart in \cite[Chapter 3]{Alm}. However, our derivation of Theorem
\ref{t:main} differs substantially from Almgren's and when we use some of his ideas, as it is
for the existence of the almost projection $\ro^\star$ of
Section~\ref{s:ro*}, we give independent arguments for the main steps of the proof. 

\subsection{Strong approximation of area minimizing currents}
Concerning multiple valued functions we will follow the notation and terminology of \cite{DS1,DS2}. In particular,
a $Q$-valued function is a map $f$ (usually defined over a measurable subset $\Omega$ of $\mathbb R^m$) taking values in the space $\Iq (\R^n)$ of unordered $Q$-tuples of points in $\mathbb R^n$, denoted by $\sum_i \a{P_i}$. $\Iqs$ can be equipped with a natural metric $\cG$ (cf. \cite[Definition 0.2]{DS1}) and for $f$ measurable there exist
measurable functions $f_i : \Omega \to \R^n$ such that $f (x) = \sum_i \a{f_i (x)}$ $\forall x\in \Omega$ (cf. \cite[Proposition 0.4]{DS1}).
The functions $f_i$ are not uniquely determined, but in using
this notation we assume to have fixed some suitable $f_i$'s. 
Moreover, if $f$ is Lipschitz, resp. $f \in W^{1,2}(\Omega,\Iq(\R^n))$ (cf. \cite[Definition~0.5]{DS1}) and 
$\Omega$ is open, then there exist measurable
functions $Df_i \in L^\infty$, resp. $L^2$, such that $\sum_i \a{Df_i (x)}$ is the approximate differential of $f$ (cf. \cite[Definition~2.6]{DS1}) at a.e. $x$. In fact in this case the $f_i$'s and $Df_i$'s can be chosen so that the first are approximately differentiable a.e. and the second are their approximate differentials in the classical sense
(cf. \cite[Lemma 1.1]{DS2}). The Dirichlet energy of $f$ is then
${\rm Dir}(f,\Omega):=\int_\Omega |Df|^2$, where $|Df|^2 := \sum_i |Df_i|^2$.
Following \cite[Definition 1.10]{DS2}, we denote
by $\bG_f$ the integer rectifiable current, in $\mathbb R^{m+n}$, naturally associated to the graph of
a Lipschitz $Q$-valued map $f:\R^m \supset A\to \Iqs$.
Moreover, we will use                               
the notation ${\rm osc}\, (f)$ for the quantity $\inf_p \sup_x \cG (f(x), Q \a{p})$.

\begin{theorem}[Almgren's strong approximation]\label{t:main}
There exist constants $C, \gamma_1,\eps_1>0$ (depending on $m,n,\bar n,Q$)
with the following property. Assume that $T$ is area minimizing, satisfies Assumption \ref{ipotesi_base} in
the cylinder $\bC_{4r} (x)$ and 
$E =\bE(T,\bC_{4\,r} (x)) < \eps_1$.
Then, there is a map $f: B_r (x) \to \Iqs$, with $\supp (f(x))\subset \Sigma$ for every $x$,
and a closed set $K\subset B_r (x)$ such that
\begin{gather}\label{e:main(i)}
\Lip (f) \leq C E^{\gamma_1}, \\
\label{e:main(ii)}
\bG_f\res (K\times \R^n)=T\res (K\times\R^{n})\quad\mbox{and}\quad
|B_r (x)\setminus K| \leq
 C \, E^{\gamma_1} \left(E+ r^2\,\bA^2\right)\, r^m,\\
\label{e:main(iii)}
\left| \|T\| (\bC_{\sigma\,r} (x)) - Q \,\omega_m\,(\sigma\,r)^m -
{\textstyle{\frac{1}{2}}} \int_{B_{\sigma\,r} (x)} |Df|^2\right| \leq
 C \, E^{\gamma_1} \left(E+ r^2\,\bA^2\right)\, r^m \quad \forall\,0<\sigma \leq 1.
\end{gather}
If in addition $\bh (T, \bC_{4r} (x), \pi_0) := \sup \{|\p^\perp (x) - \p^\perp (y)|: x,y\in \supp (T) \cap \bC_{4r} (x)\} \leq r$, then
\begin{equation}\label{e:L-infty_est}
{\rm osc}\, (f) \leq C \bh (T, \bC_{4r} (x), \pi_0) + C (E^{\sfrac{1}{2}} + r\,\bA)\, r\, .
\end{equation}
\end{theorem}

The gain of a small power $E^{\gamma_1}$ in the three estimates
\eqref{e:main(i)}-\eqref{e:main(iii)} plays a crucial role in the papers \cite{DS4,DS5}.
When $Q=1$ and $\Sigma = \R^{m+1}$, this approximation theorem was first proved by De Giorgi in \cite{DG}.
In the generality above it appears in the big regularity paper for the first time (cf.~\cite[Sections 3.28-3.30]{Alm}).
Its proof is an elementary consequence of Theorem~\ref{t:higher} and the Lipschitz approximation algorithm mentioned
above. In turn Theorem~\ref{t:higher}
will be derived from Theorem~\ref{t:higher1} using a suitable competitor argument. 
In the case $Q=1$, the competitor is the convolution of (a first) Lipschitz approximation
with a smooth kernel, a classical argument which in fact appears already
in De Giorgi's seminal paper \cite{DG}, although in a slightly different form (cf.~\cite[Appendix]{DS-cm}).

Here we need a similar approach in the framework of multiple
valued functions. However, since $\Iqs$ is highly nonlinear,
it is not possible to regularize directly by convolution.
We exploit at this point a key idea
of Almgren, embedding $\Iqs$ in an Euclidean space and using some suitable
``almost projections'' $\ro^\star_\delta$. Our proof of the existence of these almost projections is however different
from the one given by Almgren in \cite[Theorem 1.3]{Alm} and, indeed, gives better
bounds in terms of the relevant parameters (see Proposition~\ref{p:ro*}).

\subsection{Harmonic approximation} A second ingredient which in \cite{DS4,DS5} will play a key role is the harmonic approximation of Theorem
\ref{t:harmonic_final} below (already mentioned in (b) above). 
In order to state it we need to set some notation about the ambient manifold
$\Sigma$.

\begin{remark}[Estimates on $\Psi$ in good Cartesian coordinates]\label{r:Psi}
Assume that $T$ is as in Assumption \ref{ipotesi_base} in the cylinder $\bC_{4r} (x)$. If 
$E := \bE( T, \bC_{4r} (x))$ is smaller than a geometric constant,
we can assume, without loss of generality, that the function
$\Psi: \R^{m+\bar n} \to \R^l$ parameterizing $\Sigma$ satisfies 
$\Psi(x) = 0$, $\|D \Psi\|_0 \leq C\, E^{\sfrac{1}{2}} + C \bA r$ and $\|D^2 \Psi\|_0 \leq C\bA$.
Indeed observe that
\[
E = \bE(T, \bC_{4r} (x)) =
\frac{1}{2\,\omega_m\,(4r)^m} \int_{\bC_{4r}(x)} |\vec T(y) - \vec \pi_0|^2 \, d\|T\|(y)\, .
\]
Thus, we can fix a point $p \in \supp(T)\cap \bC_{4r} (x)$ such that 
$|\vec{T}(p) - \vec \pi_0|\leq C\,E^{\sfrac{1}{2}}$.
Then, we can find an associated rotation $O \in O(m+n)$ such that
$O_\sharp\vec{T}(p) = \vec\pi_0$ and $|O|\leq C\, E^{\sfrac{1}{2}}$. 
It follows that $\pi := O (T_p\Sigma)$ is a $(m + \bar n)$-dimensional 
plane such that $\pi_0 \subset\pi$ and 
$\|\pi - T_p \Sigma\| \leq C E^{\sfrac12}$. We choose new coordinates
so that $\pi_0$ remains equal to $\R^m\times \{0\}$ but $\R^{m+\bar{n}}\times \{0\}$
equals $\pi$. 
Since the excess $E$ is assumed to be sufficiently small, 
we can write $\Sigma$ as the graph of a function $\Psi: \pi\to \pi^\perp$. If $(z, \Psi (z))=p$,
then $|D \Psi (z)|\leq C \|T_p \Sigma - \R^{m+\bar n}\times\{0\}\| \leq C E^{\sfrac12}$.
However, $\|D^2 \Psi\|_{0}\leq C\bA$ and so
$\|D \Psi\|_{0} \leq C E^{\sfrac{1}{2}} + C\bA r$. Moreover,
$\Psi (x) =0$ is achieved translating the system of reference by a vector orthogonal to
$\R^{m+\bar n}\times \{0\}$ and, hence, belonging to $\{0\}\times \R^n$.
\end{remark}

From now on, we will often consider $Q$-valued maps $y \mapsto w (y) \in \Iq (\R^n) = \Iq (\R^{\bar{n}}\times \R^l)$ 
which take the form $w (y) = \sum_i \a{(u_i (y), \Psi (y, u_i (y))}$, where $u = \sum_i \a{u_i}$ is evidently a map
taking values in $\Iq (\R^{\bar{n}})$. For $w$ we will then use the short-hand notation $w = (u, \Psi (y, u))$.
We also recall the notation for the average map $\eta:\Iq(\R^n) \to \R^n$ defined by
\[
\Iq(\R^n) \ni T = \sum_{i=1}^Q \a{P_i} \mapsto \eta(T):=\frac{1}{Q} \sum_{i=1}^Q P_i \in \R^n.
\]

\begin{theorem}[Harmonic approximation]\label{t:harmonic_final}
Let $\gamma_1$ be the constant of Theorem \ref{t:main}.
Then, for every $\bar{\eta}, \bar{\delta}>0$, there is a positive constant $\bar{\eps}_1$ with the following property. 
Assume that $T$ is as in Theorem \ref{t:main},
$E := \bE(T,\bC_{4\,r}(x)) < \bar{\eps}_1$
and $r \bA \leq E^{\sfrac{1}{4}+\bar\delta}$. If $f$ is the map in Theorem~\ref{t:main} and we fix Cartesian coordinates
as in Remark \ref{r:Psi}, then there exists a $\D$-minimizing function $u: B_r (x) \to \Iq (\R^{\bar{n}})$ such that $w := (u, \Psi (y, u))$ satisfies 
\begin{equation}\label{e:harmonic_final}
r^{-2} \int_{B_r (x)} \cG (f, w)^2 + \int_{B_r(x)} \left(|Df|-|Dw|\right)^2 + \int_{B_r(x)} |D(\etaa\circ f)-
D(\etaa \circ w)|^2\leq \bar{\eta}\, E \, r^m\, .
\end{equation}
\end{theorem}

This theorem is the multi-valued analog of De Giorgi's harmonic
approximation (cf.~\cite{DG}).
We prove it via a compactness argument which, although very close in spirit
to De Giorgi's original one,
is to our knowledge new (even when $n=\bar{n}=1$).
Indeed, it uses neither the monotonicity formula nor a regularization
by convolution of the Lipschitz approximation, and we expect it to be useful
in different contexts.

\subsection{Persistence of $Q$-points}
A major ingredient in \cite{DS5} is the persistence of points of maximal multiplicity in the
approximation of Theorem \ref{t:main}, when interpreted in a suitable ``limiting sense''.
If the current $T$ has a point of density
$Q$, $f$ must satisfy the following integral bound (even though $f$ might have no values of
multiplicity $Q$). 

\begin{theorem}[Persistence of $Q$-points]\label{t:persistence}
For every $\hat{\delta}, C^\star>0$, there is $\bar{s}\in ]0, \frac{1}{2}[$ such that, for every $s<\bar{s}$, there exists $\hat{\eps} (s,C^*,\hat{\delta})>0$ with the following property. If $T$ is as in Theorem \ref{t:main}, $E := \bE(T,\bC_{4\,r} (x)) < \hat{\eps}$,
$r^2 \bA^2 \leq C^\star E$ and $\Theta (T, (p,q)) = Q$ at some $(p,q)\in \bC_{r/2} (x)$, then the approximation
$f$ of Theorem~\ref{t:main} satisfies
\begin{equation}\label{e:persistence}
\int_{B_{sr} (p)} \cG (f, Q \a{\etaa\circ f})^2 \leq \hat{\delta} s^m r^{2+m} E\, .
\end{equation}
\end{theorem}

\subsection{A remark on notation}\label{ss:xii}
Finally we remark that we follow closely the notation of \cite{DS1,DS2}, except for a subtle point.
We denote by $\xii$ the map in \cite[Corollary 2.2]{DS1}, which
there was denoted by $\xii_{BW}$, 
since the symbol $\xii$ was in fact used for the ``precursor map''of  \cite[Theorem 2.1]{DS1}.
So, here $\xii:\Iqs\to\R^{N(Q,n)}$ is an injective function satisfying
the following three properties:
\begin{itemize}
\item[(i)] $\Lip(\xii)\leq1$;
\item[(ii)] $\Lip(\xii^{-1}|_\cQ)\leq C(n,Q)$, where $\cQ= \xii(\Iq)$;
\item[(iii)] $|Df| = |D (\xii\circ f)|$ almost everywhere
for every $f\in W^{1,2} (\Omega, \Iq)$.
\end{itemize}
This ``improved'' $\xii$ was suggested by Brian White and appears for the first time in \cite{Chang}.
The conclusion (iii) above is actually not explicitly stated in \cite{DS1}, but it
follows easily: 
indeed \cite[Corollary 2.2]{DS1} implies the identity
$|Df|=|D (\xii\circ f)|$ at every point of differentiability of a Lipschitz map
and, hence, almost everywhere. The case of a general $f\in W^{1,2} (\Omega, \Iq)$ can then be concluded
from \cite[Proposition 2.5]{DS1}.

We will use the notation $C$ and $c$ for generic positive dimensional
constants, which may possibly change from line to line: we will always understand that these constants depends only on the dimensional parameters 
$m, \bar n, n, Q, c_0$ of Assumption \ref{ipotesi_base}.

\section{Lipschitz approximation}\label{s:approx}
To begin with, we develop a robust algorithm to approximate currents $T$
as in Assumption \ref{ipotesi_base} with graphs of multiple valued functions.
Following the work of Ambrosio and Kirchheim \cite{AK},
we view the slice map $x\mapsto\langle T, \p, x\rangle$ as a function taking values
in the space $\bI_0 (\R^n)$ of $0$-dimensional integral currents.
A key estimate of Jerrard and Soner (cf.~\cite{AK,JS2}) implies that this map has
bounded variation in the metric sense introduced by Ambrosio in \cite{Amb}.
On the other hand, following \cite{DS1}, $Q$-valued functions can be viewed as Sobolev
maps taking values into (a subset of) $\bI_0 (\R^n)$.
Thus, finding Lipschitz multiple
valued approximations of $T$ can be seen as a particular 
case of the more general task of
finding Lipschitz approximations of BV maps with a fairly general target space.

\begin{definition}[Maximal function of the excess measure]\label{d:maximal}
Given a current $T$ as in Assumption \ref{ipotesi_base} we introduce the ``non-centered'' maximal function of $\e_T$:
\begin{equation*}
\bmax\be_T (y) := \sup_{y \in B_s (w)\subset B_{4r} (x)} \frac{\be_T (B_s (w))}{\omega_m\, s^m} = \sup_{y\in B_s (w)\subset B_{4r} (x)}
\bE (T,\bC_s (w)) .
\end{equation*}
\end{definition}

We can now state the main result of the section, which provides the first Lipschitz
approximation for rectifiable currents.

\begin{proposition}[Lipschitz approximation]\label{p:max}
There exists a constant $C>0$ with the following property.
Let $T$ and $\Psi$ be as in Assumption \ref{ipotesi_base} in the cylinder $\bC_{4s} (x)$.
Set $E=\bE(T,\bC_{4s}(x))$, let $0<\delta_{11}<1$ be such that $16^m E < \delta_{11}$,
and define 
\[
K := \big\{\bmax\be_T<\delta_{11}\big\}\cap B_{3s}(x)\, .
\]
Then, there is $u\in \Lip (B_{3s}(x), \Iqs)$ such that
$\supp(u(y)) \subset \Sigma$ for every $y \in B_{3s}(x)$ and
\[
\Lip (u)\leq C\,\big(\delta_{11}^{\sfrac{1}{2}} + \|D\Psi\|_0\big),
\qquad {\rm osc}\, (u) \leq C \bh (T, \bC_{4s} (x), \pi_0) + C s \|D \Psi\|_0\, ,
\]
\[
\bG_u \res (K\times \R^{n})= T\res (K\times \R^{n}),
\]
\begin{equation}\label{e:max1}
|B_{r}(x)\setminus K|\leq \frac{10^m}{\delta_{11}}\,\e_T \Big(\{\bmax\be_T > 2^{-m} \delta_{11} \}\cap B_{r+r_0s}(x)\Big) \quad\forall\;
r\leq 3\,s ,
\end{equation}
where $r_0 = 16 \sqrt[m] {E/\delta_{11}} < 1$.
\end{proposition}

The proof of the proposition is based on a BV estimate which
differs from the ones of \cite{AK,JS2}.
Note that we do not assume that $T$ is area minimizing.
Indeed, even the assumption \eqref{e:(H)} could be relaxed, but we do not pursue this issue here.

\subsection{The modified Jerrard--Soner estimate}
Recall that each element $S\in\bI_0 (\R^{m+n})$ is simply a finite sum of Dirac deltas,
$S=\sum_{i=1}^h w_i \,\delta_{z_i}$,
where $h\in\N$, $w_i\in\{-1,1\}$ and the $z_i$'s are (not necessarily distinct) points in $\R^{m+n}$.
Let $T$ be a current as in Assumption \ref{ipotesi_base} in the cylinder $\bC_4 $.
The slicing map $x\mapsto \la T,\p,x\ra$ takes values in $\bI_0(\R^{m+n})$ and is characterized by (cf.~\cite[Section 28]{Sim}):
\begin{equation}\label{e:char}
\int_{B_4}\la T,\p,x\ra (\varphi) dx = T (\varphi\,dx )
\quad\text{for every }\,\varphi\in C_c^\infty(\bC_4).
\end{equation}
Moreover
$\supp(\la T,\p,x\ra)\subseteq \p^{-1}(\{x\})$ and therefore
$\la T,\p,x\ra = \sum_i w_i\, \delta_{(x, y_i)}$.
The assumption \eqref{e:(H)} guarantees that $\sum_i w_i = Q$ for almost every $x$.
In order to state our BV estimate, we consider the push-forwards of $\la
T,\p,x\ra$ into the vertical directions:
\begin{equation}\label{e:ident}
T_x:=\p^\perp_\sharp \big(\la T,\p,x\ra\big)\in\bI_0(\R^{n})\, .
\end{equation}
It follows from \eqref{e:char} that the currents $T_x$ are characterized through the identity:
\begin{equation}\label{e:carat_T_x}
\int_{B_4} T_x (\psi) \ph(x)\, dx = T (\ph(x)\, \psi(y)\,dx )
\quad\text{for every }\,\ph\in C_c^\infty(B_4),\;\psi\in C_c^\infty(\R^{n}).
\end{equation}

\begin{proposition}[BV estimate]\label{p:JS}
Assume $T$ satisfies Assumption \ref{ipotesi_base} in $\bC_4$ (i.e. $r=1$ and $x=0$ in Assumption 
\ref{ipotesi_base}). For every $\psi\in
C^\infty_c(\R^{n})$, set $\Phi_\psi(x):= T_x (\psi)$.
If $\norm{D\psi}{\infty}\leq 1$, then $\Phi_\psi\in BV (B_4)$ and satisfies
\begin{equation}\label{e:BV}
\big(|D\Phi_\psi|(A)\big)^2\leq 2m^2 \,\e_T(A)\, \|T\| (A\times\R^{n}) \quad\text{for every Borel set }\,
A\subseteq B_4.
\end{equation}
\end{proposition}

Note that in the usual Jerrard-Soner estimate the RHS of \eqref{e:BV} would be
$(\|T\|(A\times\R^{n}))^2$.
\begin{proof}
It is enough to prove \eqref{e:BV} for every open set $A\subseteq B_4$. To this aim, recall that:
\begin{equation}\label{e:def_TV}
|D\Phi_\psi|(A)=\sup\left\{\int_{A}\Phi_\psi(x)\,\dv\, \ph(x)\,dx\,:\,\ph\in C^\infty_c(A,\R^{m}),\;\norm{\ph}{\infty}\leq 1\right\}.
\end{equation}
For any smooth vector field $\ph$ we have $(\dv\, \ph(x))\,dx=d\,\Xi$, where
\[
\Xi=\sum_j \ph_j\,d\hat x^j\quad\text{and }\quad
d\hat x^j=(-1)^{j-1}dx^1\wedge\cdots\wedge dx^{j-1}\wedge dx^{j+1}\wedge\cdots\wedge dx^{m}.
\]
From \eqref{e:carat_T_x} and the assumption $\partial T\res \bC_4 = 0$ in \eqref{e:(H)}, we conclude that
\begin{align}\label{e:var}
\int_{A}\Phi_\psi(x)\,\dv\,  \ph(x)\,dx &= \int_{B_4}  T_x (\psi) \dv\, \ph(x)\,dx=
T (\psi\,\dv\,\ph\,dx)\notag\\
&=T (\psi\,d\,\Xi) =T ( d(\psi\,\Xi))- T (d\psi\wedge\Xi) =-T
(d\psi\wedge\Xi)\, .
\end{align}
Observe that the $m$-form $d\psi\wedge\Xi$ has no $dx$ component, since
\begin{equation}\label{e:forma_esplicita}
d\psi\wedge\Xi=\sum_{j=1}^m\sum_{i=1}^n
\frac{\de \psi}{dy^i}(y)\,\ph_j(x)\,dy^i\wedge d\hat x^j.
\end{equation}
Write $\vec{T}=\langle \vec T, \vec{\pi}_0\rangle\,\vec{\pi}_0+\vec{S}$. Then,
\begin{align*}
(T (d\psi\wedge\Xi))^2 &=\Big(\int \langle \vec{S} , d\psi\wedge\Xi\rangle\, d\|T\|\Big)^2
\leq \| |d\psi\wedge \Xi| \|_\infty^2 \|T\| (A\times \R^n)
\int_{A\times\R^{n}}|\vec{S}|^2\,d\norm{T}{},
\end{align*}
($|\cdot |$ denotes the norms  on $\Lambda_m$ and $\Lambda^m$
induced by the natural inner products $\langle, \rangle$). 
Since $|\vec{S}|^2 = 1-\langle \vec T, \vec \pi_0\rangle^2 \leq 2-2 \langle \vec{T}, \vec{\pi}_0\rangle$,
we have
\[
\int_{A\times\R^{n}}|\vec{S}|^2\,d\norm{T}{}
\leq 2 \int_{A\times\R^{n}}\left(1-\langle\vec T, \vec \pi_0\rangle\right)\,d\norm{T}{}=2\,\e_T(A).
\]
Moreover, by \eqref{e:forma_esplicita}, $\||d\psi\wedge\Xi| \|_\infty \leq m \norm{D\psi}{\infty}\,\norm{\ph}{\infty}\leq m$. Summarizing, we get
\begin{align*}\label{e:quasi_JS}
& \int_{A}\Phi_\psi(x)\,\dv\, \ph(x)\, dx 
\leq \left( 2m^2\, \e_T(A)\,\|T\| (A\times\R^n)\right)^{\sfrac{1}{2}}\qquad
 \forall \varphi\in C^\infty_c (A, \R^m),  \|\varphi\|_\infty\leq 1\, .
\end{align*}
Taking the supremum  over $\ph$'s
we conclude \eqref{e:BV} through \eqref{e:def_TV}.
\end{proof}

\subsection{Proof of Proposition \ref{p:max}}
Since the statement is invariant under translations and dilations,
without loss of generality we assume $x=0$ and $s=1$.
Consider the slices $T_x : = \p^\perp_\sharp \langle T, \p, x\rangle \in\bI_0(\R^{n})$ and recall that $\|T\| (A\times \R^{n}) \geq \int_A \mass (T_x)\, dx$ for every open set $A$
(cf.~\cite[Lemma 28.5]{Sim}).
Therefore,
\[
\mass(T_x)\leq\lim_{r\to 0}\frac{\|T\| (\bC_r(x))}{\omega_m\,r^m}\leq \bmax\be_T(x)+Q \quad\text{for almost every }x.
\]
Since $\delta_{11}<1$, we infer $\mass(T_x)<Q+1$ for a.e. $x\in K$.
There are, then, $Q$ functions $g_i: K \to \R^n$ such that $T_x=\sum_{i=1}^Q\delta_{g_i(x)}$ for a.e. $x\in K$.
Define $g:K\mapsto\Iqs$ as $g := \sum_i \a{g_i}$ and fix $\psi\in C^\infty_c(\R^{n})$. Proposition~\ref{p:JS} gives
\[
\left(|D\Phi_\psi| (B_r (y))\right)^2 \leq 2m^2\, \be_T (B_r (y))\, \|T\| (\bC_r (y)) = 2m^2\,
\be_T (B_r (y)) \big(Q |B_r (y)| + \be_T (B_r (y))\big)\, .
\]
Hence, if we define the maximal function 
\[
\bmax |D\Phi_\psi| (x) := \sup_{x\in B_r (y)\subset B_{4r}} \frac{|D\Phi_\psi|(B_r (y))} {|B_r (y)|}\, ,
\]
we conclude that
\begin{align*}
(\bmax |D\Phi_\psi| (x))^2
&\leq 2 m\, \bmax\be_T(x)^2+2m \,Q\,\bmax\be_T(x)\leq C \delta_{11} \quad \text{for every $x\in K$}.
\end{align*}
Therefore, the theory of $BV$ functions gives a dimensional constant $C$ such that
\begin{equation}\label{e:supremize}
|\Phi_\psi(x)-\Phi_\psi(y)|
\leq C\,\delta_{11}^{\sfrac{1}{2}}\,|x-y| \qquad \mbox{$\forall x, y\in K$ Lebesgue points of $\Phi_\psi$,}
\end{equation}
(see for instance \cite[Section 6.6.2]{EG}: although in that reference the authors use the {\em centered}
maximal function, the proof works obviously also in our context).
Consider next the Wasserstein distance of exponent $1$ on $0$-dimensional integral currents $S_1$, $S_2$:
\begin{equation}
W_1(S_1,S_2):=\sup \left\{\la S_1-S_2, \psi\ra\,:\, \psi\in C^1(\R^{n}),\;
\norm{D\psi}{\infty}\leq 1\right\}.\label{e:W1sup}
\end{equation}
Obviously, when $S_1=\sum_i\a{S_{1i}}, S_2=\sum_i\a{S_{2i}}\in \Iqs$, the supremum in \eqref{e:W1sup} can be
taken over a suitable countable subset of $\psi\in C_c^\infty(\R^{n})$, chosen independently of the $S_i$'s.
Moreover, we have that
\begin{equation}\label{e:wellknown}
W_1(S_1,S_2)=\min_{\sigma\in \Pe_Q}\sum_i|S_{1i}-S_{2\sigma(i)}|\geq
\min_{\sigma\in \Pe_Q}\Big(\sum_i|S_{1i}-S_{2\sigma(i)}|^2\Big)^{\sfrac{1}{2}} =
\cG(S_1,S_2).
\end{equation}
So $\cG(g(x) , g(y)) \leq C\,\delta_{11}^{\sfrac{1}{2}}\,|x-y|$ for a.e. $x,y\in K$.
(The first equality in \eqref{e:wellknown} is well-known, but not easy to find in the literature. It can be derived by suitably modifying the
arguments of \cite[4.1.12]{Fed}. Another quick derivation is the following. 
Consider the set $\Pi$
of probability measures $\pi$ on $\R^n\times \R^n$ of the form $\sum_{i,j} c_{ij} \delta_{(S_{1i}, S_{2j})}$, where the matrix of coefficients $c_{ij}$ 
consists of nonnegative entries with $\sum_k c_{kj} = 1$ and $\sum_k c_{ik}=1$ for every $i$ and $j$, i.e. it is a doubly stochastic matrix. It then follows from the Kantorovich duality, see for instance \cite[Theorem 1.14]{Vil}, that $W_1 (S_1, S_2) = \min_{\pi\in \Pi} \int |x-y|\, d\pi (x,y)$. Observe however that $\int |x-y|\, d\pi (x,y)$ is a linear function of the coefficients $c_{ij}$: the space of such matrices, also called Birkhoff polytope, is a compact convex set and so the minimum is attained on the subset of extremal points. By the classical
Birkhoff-von Neumann Theorem this set consists of the permutations matrices (see \cite{Birk}) and so
$\min_{\pi\in \Pi} \int |x-y|\, d\pi = \min_{\sigma\in \Pe_Q} \sum_i |S_{1i}- S_{2\sigma (i)}|$.)

Next, write $g (x) = \sum_i \a{(h_i (x), \Psi (x, h_i (x)))}$. Obviously $x\mapsto h (x) := \sum_i \a{h_i (x)} \in \Iq(\R^{\bar n})$ is a Lipschitz map on $K$ with Lipschitz constant $\leq C\, \delta_{11}^{\sfrac{1}{2}}$.
Recalling \cite[Theorem 1.7]{DS1}, we can extend it to a map $\bar{u}\in \Lip (B_3, \Iq (\R^{\bar n}))$ satisfying $\Lip(\bar u)\leq C\,\delta_{11}^{\sfrac{1}{2}}$ and
${\rm osc}\, (\bar{u})\leq C {\rm osc}\, (h)$.
Set finally $u(x) = \sum_i \a{(\bar{u}_i (x), \Psi(x, \bar{u}_i(x)))}$. 
We start showing the Lipschitz bound. Fix $x_1, x_2 \in B_3$ and assume, without loss of generality, that $\cG(\bar u(x_1), \bar u(x_2))^2 = \sum_i |\bar u_i(x_1) - \bar u_i(x_2)|^2$. Then
\begin{align*}
\cG(u(x_1), u(x_2))^2 & \leq \sum_i \big\vert (\bar{u}_i (x_1), \Psi(x_1, \bar{u}_i(x_1))) - (\bar{u}_i (x_2), \Psi(x_2, \bar{u}_i(x_2))) \big\vert^2\\
&\leq 2\sum_i \Big((1+\|D_y\Psi\|_0^2)|\bar u_i(x_1) - \bar u_i(x_2)|^2 + \|D_x\Psi\|_0^2 |x_1-x_2|^2 \Big)\\
& \leq 2 (1+\|D\Psi\|_0^2) \cG(\bar u(x_1), \bar u(x_2))^2 + 2 \|D\Psi\|_0^2 |x_1-x_2|^2\\
&\leq C (\delta_{11}+\|D\Psi\|_0^2) |x_1 -x_2|^2\, .
\end{align*}
As for the $L^\infty$ bound, let $\eta>0$ be arbitrary and $p \in \R^{\bar n}$
be such that ${\rm osc}(\bar u) \leq \sup_{x \in B_3} \cG(\bar u(x), Q\a{p})+\eta$. Proceeding as above
\begin{align*}
{\rm osc}(u)^2 &\leq \sup_{x \in B_3} \cG(u(x), Q\a{(p, \Psi(0,p))})^2 \\
&\leq 2 \sup_{x \in B_3} \Big((1+\|D \Psi\|_0^2)\cG(\bar u(x), Q\a{p})^2  +  \|D \Psi\|_0^2 |x|^2\Big)\\
& \leq 4 (1+\|D\Psi\|_0^2) \big({\rm osc}(\bar u)^2 +\eta^2\big) + 18\,\|D\Psi\|_0^2.
\end{align*}
Since ${\rm osc}(h) \leq \bh(T, \bC_4, \pi_0)$, the estimate on ${\rm osc} (u)$ follows letting $\eta\downarrow 0$.

The identity $\mathbf{G}_u \res (K\times \R^{n})= T\res (K\times\R^{n})$ is a consequence
of $u(x) = T_x$ for a.e. $x\in K$. Indeed, recall that both $T$ and $\bG_u$ are rectifiable and observe that $\langle \vec{T}, 
\vec\pi_0\rangle \neq 0$ $\|T\|$-a.e. on $K\times \R^n$, because $\bmax\be_T < \infty$ on $K$. Similarly,
$\langle \vec{\bG}_u, \vec{\pi}_0\rangle \neq 0$ $\|\bG_u\|$-a.e. on $K\times \R^n$, by \cite[Proposition 1.4]{DS2}.
Thus, $(\bG_u - T)\res K\times \R^n =0$ if and only if $(\bG_u - T) \res dx {\mathbf 1}_{K\times \R^n}=0$.
The latter identity follows from the slicing formula and the property $\langle T, \p, x \rangle =
\langle \bG_u, \p, x\rangle = \sum_i \delta_{(x, u_i (x))}$, valid for a.e. $x\in K$.

Finally, for each $x\in B_r\setminus K$ choose a ball $x\in B^x = B_{r(x)} (y(x)) \subset B_4$ such that $\e_T (B^x) \geq 2^{-m} \delta_{11} \omega_m r (x)^m$. By the $5r$-Covering theorem, we choose balls $\hat{B}^i =
B_{5r (x_i)} (y(x_i))$ which cover $B_r\setminus K$ and such that the balls $B^{x_i}$ are pairwise disjoint. We
then conclude
\begin{equation}\label{e:5r_covering}
|B_r \setminus K|\leq 10^m \delta_{11}^{-1} \e_T \left(\bigcup_i B^{x_i}\right)\, .
\end{equation}
Fix $y\in B^{x_i}$. Since $B^{x_i} \subset B_4$, we have $2^{-m} \delta_{11} \omega_m r (x_i)^m \leq \e_T (B^{x_i}) \leq \e_T (B_4) = 4^m \omega_m E$, which implies $2r (x_i) \leq r_0 <1$. Thus, $y\in B_{r+r_0} \subset B_4$.
By definition of $\bmax\be_T$ we obviously have $\bmax\be_T (y) \geq 2^{-m} \delta_{11}$. So
$\cup_i B^{x_i} \subset B_{r+r_0} \cap \{\bmax\be_T > 2^{-m} \delta_{11}\}$ and
\eqref{e:5r_covering} implies \eqref{e:max1}.

\section{Patching multiple valued graphs}
In this section we prove some complementary results to the theory of
multiple valued functions as exposed in \cite{DS1, DS2}.
In particular, we show here a concentration compactness principle for
$Q$-valued functions, and give an algorithm to construct suitable competitors for
the Dirichlet energy, which will be also used in \cite{DS5}.
We first introduce some terminology.

\begin{definition}[Translating sheets]\label{d:pacchetti}
Let $\Omega\subset\R^m$ be a bounded open set. A sequence of maps $\{h_k\}_{i\in \N}\subset W^{1,2}(\Omega,\Iqs)$ is called a sequence of {\em translating sheets} if there are:
\begin{itemize}
\item[(a)] integers $J\geq1$ and $Q_1, \ldots, Q_J\geq1$
satisfying $\sum_{j=1}^J Q_j = Q$,
\item[(b)] vectors $y^j_k\in \R^{n}$ (for  $j\in \{1, \ldots, J\}$ and $k\in \N$) with
\begin{equation}\label{e:separazione}
\lim_k |y^j_k - y^i_k|= + \infty\qquad \forall i\neq j,
\end{equation}
\item[(c)] and maps $\zeta^j\in W^{1,2}(\Omega, \I{Q_j})$ for $j\in \{1, \ldots, J\}$,
\end{itemize}
such that $h_k=\sum_{j=1}^J\llbracket\tau_{y^j_k}\circ \zeta^j\rrbracket$, 
where for any generic $y \in \R^n$ we denote by $\tau_{y} : \Iq(\R^n) \to \Iq(\R^n)$ the translation map
(cp.~\cite[Section 3.3.3]{DS1})
\[
\Iq(\R^n) \ni T = \sum_i\a{P_i} \mapsto \tau_y(T) := \sum_i \a{P_i - y} \in \Iq(\R^n).
\]
\end{definition}

\begin{remark}\label{r:separating} 
Assume that $h_k, Q_j$, $y^j_k$ and $\zeta^k$ satisfy all the requirements
of Definition \ref{d:pacchetti} except for \eqref{e:separazione}. 
Up to subsequences and relabellings,
assume that $y^1_k - y^2_k$ converges to a vector $2 \bar{y}$.
We can replace
\begin{itemize}
\item the integers $Q_1$ and $Q_2$ with $Q'= Q_1+Q_2$;
\item the vectors $y^1_k$ and $y_2^k$ with $y'_k = (y^1_k + y^2_k)/2$;
\item the maps $\zeta^1$ and
$\zeta^2$ with $\zeta' := \a{\tau_{\bar{y}} \circ \zeta^1}
+\a{\tau_{-\bar{y}} \circ \zeta^2}$.
\end{itemize}
The new collections $Q', Q_3, \ldots, Q_{J}$, $y'_k, y^3_k, \ldots, y^J_k$
and $\zeta', \zeta^3, \ldots, \zeta^{J}$, and the function
$h'_k := \a{\zeta'} + \sum_{j=3}^J\a{\zeta^j}$,
satisfy again all the requirements of Definition \ref{d:pacchetti} except,
possibly, for \eqref{e:separazione}. Moreover, $\|\cG(h_k',h_k)\|_{L^2}\to0$ and $|Dh'_k| = |Dh_k|$.
Obviously, we can iterate this procedure only a finite number of times,
obtaining a subsequence of translating sheets $\hat{h}_k$ asymptotic to $h_k$ in the $L^2$ distance with
$|D\hat{h}_k| = |Dh_k|$.
\end{remark}

\subsection{Concentration compactness} Translating sheets give a useful device to recover a suitable
``compactness statement'' for sequences of  maps
with equi-bounded energy.

\begin{proposition}[Concentration compactness]\label{p:cc}
Let $\Omega\subset \R^m$ be a Lipschitz bounded open set
and $(g_k)_{k\in\N} \subset W^{1,2}(\Omega,\Iq)$ a sequence of functions
with $\sup_k \int_\Omega |Dg_k |^2 < \infty$.
Then, there exist a subsequence (not relabeled) and a sequence of translating sheets $h_k$
such that $\norm{\cG(g_k,h_k)}{L^2}\to 0$ and the following inequalities
hold for every open $\Omega'\subset \Omega$ and any sequence of measurable sets
$J_k$ with $|J_k|\to 0$:
\begin{gather}
\liminf_{k\to+\infty} \left(\int_{\Omega'\setminus J_k} |Dg_k|^2 -\int_{\Omega'} |Dh_k|^2\right)\geq 0\label{e:cc2}\\
\limsup_{k\to+\infty} \int_\Omega
\left(|Dg_k|-|Dh_k|\right)^2 \leq \limsup_k \int_\Omega \left(|Dg_k|^2 - |D h_k|^2\right)\, .\label{e:cc3}
\end{gather}
\end{proposition}

\begin{proof}
We start proving, by induction on $Q$, the existence of translating sheets $\{h_k\}$ (and a subsequence) with $\|\cG(h_k,g_k)\|_{L^2} \to 0$ and satisfying the following additional property. If $J, Q_j$, $y_k^j$ and $\zeta^j$ 
are as in Definition \ref{d:pacchetti}, then there are $Q_j$ valued functions $w_k^j$  such that, after
setting $f_k=\sum_j\a{w_k^j}$, we have
\begin{equation}\label{e:extra}
\|\cG({f_k},g_k)\|_{L^2} + |\{g_k\neq {f_k}\}| \to 0,\quad 
\|\cG (\tau_{-y_k^j}\circ w^j_k, \zeta^j)\|_{L^2}\to 0
\quad\text{and}\quad
|D { f_k}|\leq |Dg_k|\, .
\end{equation}

If $Q=1$ the claim with $f_k=g_k$ is an easy corollary of 
the Poincar\'e inequality and the compact embedding $W^{1,2}\hookrightarrow L^2$.
Assuming that the claim holds for any $Q^*<Q$, we prove it for $Q$. 
By the generalized Poincar\'e inequality \cite[Proposition 2.12]{DS1}, there
exist points
$\bar g_k\in \Iqs$ and a real number $M$ such that
\begin{equation*}
\int_\Omega \cG (g_k, \bar{g}_k)^2 \leq C \int_\Omega |Dg_k|^2 \leq M < \infty \qquad \forall\; k\in\N\, .
\end{equation*}
Recall the separation $s(T)$ and the diameter $d(T)$ of a point $T=\sum_i\a{P_i}$ introduced
in  \cite[Definition 3.4]{DS1}:
$s (T):= \min \big\{ |P_i - P_j| : P_i\neq P_j\big\}$ and $d(T):= \max \{|P_i-P_j|\}$.
We distinguish between to cases.

\medskip

\noindent\textit{Case 1: $\liminf_k d(\bar{g}_k) < \infty$}.
After passing to a subsequence, we find $y_k\in\R^{n}$ such that the functions
$\tau_{y_k}\circ g_k$ are equi-bounded in the $W^{1,2}$-metric.
By the Sobolev embedding \cite[Proposition 2.11]{DS1},
there exists a $Q$-valued map $\zeta\in
W^{1,2}$ such that $\tau_{y_k}\circ g_k \to \zeta$ in $L^2(\Omega)$.

\medskip

\noindent\textit{Case 2: $\lim_k d(\bar g_k)=+\infty$.}
By \cite[Lemma 3.8]{DS1} there are points $S_k\in \Iq$ such that
\begin{equation*}
\beta\,d(\bar g_k) \leq s(S_k) < +\infty  \quad\text{and}\quad \cG(S_k,\bar g_k)\leq s(S_k)/32,
\end{equation*}
where $\beta$ is a dimensional constant.  
Write $S_k=\sum_{i=1}^{J} \kappa_i\a{P^i_k}$, with $P^i_k\neq P^j_k$ for $i\neq j$.
Both $J$ and $\kappa_i$ may depend on $k$ but they have a finite range: therefore,
after extracting a subsequence, we can assume that they do not depend on $k$.
Set next $r_k=\frac{s(S_k)}{16}$ and let $\vartheta_k$ be the retraction of $\Iqs$ into $\overline{B_{r_k}(S_k)}$
provided by \cite[Lemma 3.7]{DS1}.
Clearly, the functions $\hat f_k=\vartheta_k\circ g_k$ satisfy $|D \hat f_k|\leq |Dg_k|$ and
there are $\kappa_i$-valued functions $z_k^i$ such that
\[
\hat f_k=\sum_{i=1}^J\a{z_k^i}, \quad\text{with}\quad\|\cG (z_k^i, \kappa_i \a{P_k^i})\|_\infty
\leq r_k.
\]
Since $\kappa_i<Q$, we apply the inductive hypothesis to each sequence $(z_k^i)_k$ and, using Remark~\ref{r:separating} reach a subsequence (not relabeled) of $\hat f_k$,
a sequence of translating sheets $h_k$ and corresponding functions $f_k$
which satisfy \eqref{e:extra} with $\hat f_k$ replacing $g_k$. 

We next claim that \eqref{e:extra} holds even for $g_k$, i.e. that
$\lim_k \left(\|\cG (f_k, g_k)\|_{L^2} + |\{f_k\neq g_k\}|\right) = 0$.
To this aim, recall first that
\[
\left\{g_k\neq \hat f_k\right\}=\left\{\cG\left(g_k,S_k\right)>r_k\right\}\subseteq \left\{\cG\left(g_k,\bar g_k\right)>r_k/2\right\}.
\]
Thus, 
\begin{equation}\label{e:stima_sopra_livello}
\left|\left\{g_k\neq \hat f_k\right\}\right|\leq |\left\{\cG\left(g_k,\bar g_k\right)>r_k/2\right\}|
\leq \frac{C}{r_k^2}\int_{\left\{\cG\left(g_k,\bar g_k\right)>
\frac{r_k}{2}\right\}}\cG\left(g_k,\bar g_k\right)^2 \leq \frac{CM}{(d (\bar g_k))^2}.
\end{equation}
Since $d(\bar g_k)\to +\infty$ and \eqref{e:extra} holds with $\hat f_k$ replacing $g_k$,
we conclude $|\{f_k\neq g_k\}|\to 0$.
Next, since $\vartheta_k (\bar{g}_k)= \bar{g}_k$ and $\Lip (\vartheta_k) =1$, we have $\cG(\hat f_k,\bar g_k)\leq
\cG(g_k,\bar g_k)$.
Therefore, by the Sobolev embedding and the Poincar\'e inequality, for any $p\in ]2, 2^*[$, we infer 
\begin{align*}
&\int_{\Omega}\cG( \hat f_k,g_k)^2=
\int_{\{g_k\neq  \hat f_k\}}\cG( \hat f_k,g_k)^2
\leq 2\int_{\{ \hat f_k\neq g_k\}}\cG(\hat f_k,\bar g_k)^2+
2\int_{\{\hat f_k \neq g_k\}}\cG(\bar g_k,g_k)^2\notag\\
&\leq 4\int_{\{ \hat f_k\neq g_k\}}\cG(\bar g_k,g_k)^2
\leq C \norm{\cG\left(g_k,\bar g_k\right)}{L^{p}}^2
\left|\left\{\hat f_k\neq g_k\right\}\right|^{1-\frac{2}{p}}
\stackrel{\eqref{e:stima_sopra_livello}}{\leq} \frac{CM^{1-\sfrac{2}{p}}}{d(\bar g_k)^{2 - \sfrac{4}{p}}}
\int_{\Omega} |Dg_k|^{2}.
\end{align*}
Since $d(\bar g_k)$ diverges, this shows $\|\cG(\hat f_k,g_k)\|_{L^2}\to 0$ and by
inductive hypothesis that $\norm{\cG(f_k,g_k)}{L^2}\to 0$.

\medskip

We now show that \eqref{e:cc2} and \eqref{e:cc3} are consequences of \eqref{e:extra}. 
For each $j$ we consider the corresponding embedding $\xii_j: \cA_{Q_j} (\R^n)\to \R^{N(Q_j, n)}$ and,
by a slight abuse of notation, we drop the $j$ subscript.
Then, we conclude that 
$\xii\circ\tau_{-y^j_k}\circ w^j_k\to \xii\circ  \zeta^j$ in $L^2$ and
$\|D(\xii\circ\tau_{-y^j_k}\circ w_k^j)\|_{L^2}$ is a bounded sequence, from which
\begin{equation}\label{e:weaks}
D (\xii\circ \tau_{-y^j_k}\circ w^j_k) \rightharpoonup D 
(\xii \circ \zeta^j) \quad \mbox{in $L^2(\Omega)$}\, .
\end{equation}
If $J_k$ is a sequence of measurable sets with $|J_k|\downarrow 0$, then
${\bf 1}_{\Omega'\setminus J_k}\to {\bf 1}_{\Omega'}$ in $L^2(\Omega)$ and it follows from \eqref{e:weaks} that
\begin{equation*}
D (\xii\circ \tau_{-y^j_k}\circ w^j_k) {\bf 1}_{\Omega'\setminus J_k}
\rightharpoonup D (\xii \circ \zeta^j){\bf 1}_{\Omega'}\quad \mbox{in $L^2(\Omega)$}\,,
\end{equation*}
and, hence,
\begin{equation*}
\D (\zeta^j, \Omega')= \int_{\Omega'} |D (\xii \circ \zeta^j)|^2
\leq \liminf_k \int_{\Omega'\setminus J_k} 
|D (\xii\circ \tau_{-y^j_k}\circ w^j_k)|^2
= \liminf_k \int_{\Omega'\setminus J_k} |Dw^j_k|^2.
\end{equation*}
Summing over $j$, we obtain \eqref{e:cc2}.
As for \eqref{e:cc3}, set $J_k := \{g_k\neq f_k\}$. Thus,
\begin{align}
&\int_{\Omega\setminus J_k} (|Dg_k| - |Dh_k|)^2 \leq \sum_j \int_{\Omega\setminus J_k} (|Dw_k^j| - |D\zeta^j|)^2\nonumber\\
=&
\sum_j \int_{\Omega\setminus J_k} \big(|D(\xii \circ \tau_{-y_k^j} \circ w_k^j)| - |D(\xii \circ \zeta^j)|\big)^2
\leq \sum_j \int_{\Omega\setminus J_k} |D(\xii \circ \tau_{-y_k^j} \circ w_k^j) - D(\xii \circ \zeta^j)|^2\notag\allowdisplaybreaks\\
=& \sum_j \int_{\Omega\setminus J_k} \left(|D(\xii \circ \tau_{-y_k^j} \circ w_k^j)|^2 + |D(\xii \circ \zeta^j)|^2
- 2\,D(\xii \circ \tau_{-y_k^j} \circ w_k^j) \cdot D(\xii \circ \zeta^j)\right).\label{e:cc_ancora_1}
\end{align}
Therefore, by \eqref{e:weaks} (and taking into account that $|J_k|\to 0$) one gets
\begin{align}
&\limsup_{k\to + \infty}\int_{\Omega\setminus J_k} (|Dg_k| - |Dh_k|)^2 \notag\\
\leq &\lim_{k\to + \infty} \sum_j \int_{\Omega\setminus J_k} \Big(|D(\xii \circ \tau_{-y_k^j} \circ w_k^j)|^2 + |D(\xii \circ \zeta^j)|^2
- 2\,D(\xii \circ \tau_{-y_k^j} \circ w_k^j) \cdot D(\xii \circ \zeta^j)\Big)\notag\\
=&\limsup_{k\to + \infty} \int_{\Omega\setminus J_k} \sum_j |D(\xii \circ \tau_{-y_k^j} \circ w_k^j)|^2 - \int_\Omega
\sum_j |D(\xii \circ \zeta^j)|^2\nonumber\\
= &\limsup_{k\to + \infty} \int_{\Omega\setminus J_k} |Dg_k|^2 - \int_\Omega |Dh_k|^2.\label{e:cc_ancora_2}
\end{align}
On the other hand, since $|J_k|\to 0$ we conclude
\[
\limsup_{k\to\infty} \int_{J_k} \left(|Dg_k|-|Dh_k|\right)^2 = \limsup_{k\to\infty} \int_{J_k} |Dg_k|^2\, .
\]
Observe that, after passing to a subsequence, we can actually assume that all limsups are in fact limits.
Summing \eqref{e:cc_ancora_2} and the last equation we then conclude \eqref{e:cc3}.
\end{proof}

\subsection{Dirichlet competitors}
We consider next a standard procedure to construct competitors for the
Dirichlet energy of a sequence of functions with equi-bounded energy.

\begin{proposition}[Construction of a competitor]\label{p:raccordo}
Consider two radii $1\leq r_0<r_1 < 4$ and maps 
$g_k,h_k \in W^{1,2}(B_{r_1}, \Iqs)$ such that
$\{h_k\}_k$ is a sequence of translating sheets,
\[
\sup_k \D(g_k,B_{r_1})< + \infty \quad\text{and}\quad
\|\cG(g_k,h_k)\|_{L^2(B_{r_1}\setminus B_{r_0})} \to 0.
\]
For every $\eta>0$, there exist $r\in ]r_0, r_1[$, a subsequence of $\{g_k\}_k$
(not relabeled)
and functions $H_k\in W^{1,2} (B_{r_1}, \Iqs)$ such that
$H_{k}\vert_{B_{r_1}\setminus B_r} = g_{k}\vert_{B_{r_1}\setminus B_r}$ and
$\D(H_{k}, B_{r_1}) \leq \D(h_{k}, B_{r_1}) + \eta$.
In addition, there is a dimensional constant $C$ and a constant $C^*$ (depending on $\eta$ and the
two sequences, {\em but not on $k$}) such that
\begin{gather}
\Lip(H_{k})\leq C^*\, (\Lip(g_{k}) + 1) ,\label{e:raccordo1}\\
\|\cG(H_{k}, h_{k})\|_{L^2(B_{r})} \leq C \D(g_{k}, B_{r}) + 
C \D(H_{k}, B_{r})\, ,\label{e:raccordo2}\\
\|\etaa\circ H_{k}\|_{L^1(B_{r_1})} \leq C^*\, \|\etaa\circ g_{k}\|_{L^1(B_{r_1})} +
C \|\etaa\circ h_k\|_{L^1 (B_{r_1})}\, .\label{e:raccordo3}
\end{gather}
\end{proposition}

In order to prove the proposition, we need to recall the following two lemmas,
which are slight variants of
\cite[Proposition 4.4]{DS1} and \cite[Lemma~2.15]{DS1}.

\begin{lemma}[Lipschitz approximation]\label{l:approx}
Let $f\in W^{1,2}(B_r,\Iq)$.
Then, for every $\eps>0$, there exists $f_\eps\in \Lip (B_r,\Iq)$ such that
\begin{equation}\label{e:approx interior}
\int_{B_r}\cG(f,f_\eps)^2+\int_{B_r}\big(|Df|-|Df_\eps|\big)^2
+ \int_{B_r}\big(|D(\etaa\circ f)|-|D(\etaa\circ f_\eps)|\big)^2
\leq \eps.
\end{equation}
If $f\vert_{\de B_r}\in W^{1,2}(\de B_r,\Iq)$,
then $f_\eps$ can be chosen to satisfy also
\begin{equation}\label{e:approx bordo}
\int_{\de B_r}\cG(f,f_\eps)^2+\int_{\de B_r}\big(|Df|-|Df_\eps|\big)^2 \leq \eps.
\end{equation}
\end{lemma}

\begin{proof}
By an obvious scaling argument we can assume $r=1$.
We start noticing that \eqref{e:approx interior} 
is a corollary of \cite[Proposition~4.4]{DS1}.
On the other hand, if $f|_{\partial B_1} \in W^{1,2} (\partial B_1)$,
we extend the map to $B_2$ by setting $f(x) = f (\frac{x}{|x|})$ if $|x|\geq 1$.
We then can apply \cite[Proposition~2.5]{DS1}
to find a sequence of Lipschitz maps $f_k$ such that $f_k\to f$ strongly
in $W^{1,2}(B_2)$. Given $\delta>0$, define
the maps $f^\delta (x) = f ((1+\delta) x)$ and $f_k^\delta (x) = f_k ((1+\delta) x)$. Obviously,
$f^\delta_k \to f^\delta$ strongly in $W^{1,2} (B_1)$ and $f^\delta \to f$ strongly in $W^{1,2} (B_1)$ as
$\delta\downarrow 0$. By a standard Fubini argument, for each $j$ we can find a $\delta_j<\frac{1}{j}$ 
and a subsequence $\{f_{k,j}\}_k$ such that $f_{k,j}|_{\partial B_{1+\delta_j}} \to f|_{\partial B_{1+\delta_j}}$
(i.e.~$f^{\delta_j}_{k,j} |_{\partial B_1} \to f^{\delta_j}|_{\partial B_1}
= f|_{\partial B_1}$) 
strongly in $W^{1,2} (\partial B_{1+\delta_j})$ as $k\uparrow \infty$.
By standard diagonal argument we can arrange the subsequences so that $\{f_{k,j}\}\supset \{f_{k, j+1}\}$.
Thus, a suitable diagonal sequence
$\bar{f}_j := f^{\delta_j}_{k(j),j}$ has the property that $\bar{f}_j\to f$ in $W^{1,2} (B_1)$ and $\bar{f}_j|_{\partial B_1} \to
f|_{\partial B_1}$ in $W^{1,2} (\partial B_1)$.
\end{proof}

\begin{lemma}[Interpolation]\label{l:interpolation}
There exists a constant $C_0=C_0(m,n,Q)>0$ with the following property.
Assume $r\in ]1,3[$, $f\in W^{1,2}(B_r,\Iq)$ and $g\in W^{1,2}(\de B_{r},\Iq)$ are
given maps such that $f\vert_{\de B_r} \in W^{1,2}(\de B_r,\Iq)$.
Then, for every $\eps\in ]0,r[$ there exists a function $h\in W^{1,2}(B_r,\Iq)$ such
that $h\vert_{\de B_r}=g$ and
\begin{gather}
\int_{B_r} |Dh|^2 \leq
\int_{B_r} |Df|^2+\eps\int_{\partial B_r} \left(|D_\tau f|^2 + |D_\tau g|^2\right) +
\frac{C_0}{\eps} \int_{\de B_r}\cG(f,g)^2\, ,\label{e:raccordo_Dir}\\
\Lip (h) \leq C_0\left\{\Lip(f)+\Lip(g)+\eps^{-1}\sup_{\de B_r} \cG(f,g)\right\}\, ,\label{e:lip approx}\\
\int_{B_r} |\etaa \circ h|\leq C_0 \int_{\partial B_r} |\etaa\circ g| + C_0 \int_{B_r} |\etaa\circ f|\, ,\label{e:media_per_bu}
\end{gather}
(here $D_\tau$ denotes the tangential derivative).
\end{lemma}

\begin{proof}
The first conclusion is an obvious corollary of \cite[Lemma 2.15]{DS1}. It is then straightforward to see that the map
constructed in the proof of \cite[Lemma 2.15]{DS1} satisfies also \eqref{e:lip approx}. As for the final claim, let $\bar{g} := \sum \a{g_i - \etaa \circ g}$, $\bar{f}:= \sum \a{f_i - \etaa\circ f}$ and consider the interpolation map $\bar{h}$ between $\bar{f}$ and $\bar{g}$ given by \cite[Lemma 2.15]{DS1}. Set $\hat{h} =
\sum_i \llbracket\bar{h}_i - \etaa\circ \bar{h}\rrbracket$ and observe that $\Lip (\hat{h}) \leq \Lip (\bar{h})$ and $\D (\hat{h})
\leq \D (\bar{h})$. 
We apply again \cite[Lemma 2.15]{DS1} in the case $Q=1$
to $\etaa \circ f$ and $\etaa \circ g$, and get the interpolation $u$. 
It is then easy to check that the map $h:= \sum_i \llbracket \hat{h}_i + u \rrbracket$ has all the desired properties.
\end{proof}

\begin{proof}[Proof of Proposition~\ref{p:raccordo}]
Set for simplicity 
$A_k := \|\cG(g_k,h_{k})\|_{L^2(B_{r_1}\setminus B_{r_0})}$ and
$B_k := \|\etaa\circ g_k\|_{L^1(B_{r_1})}$.
If $A_k \equiv 0$, then there is nothing to prove and so we can assume that, for 
a subsequence, not relabeled, $A_k >0$.
Assuming that for yet another subsequence
(not relabeled) $B_k>0$, we consider the function
\begin{equation}\label{e:def_psi}
\psi_k(r) := \int_{\de B_{r}} \left(|Dg_k|^2+ |Dh_{k}|^2\right) 
+ A_k^{-2} \int_{\de B_r} \cG (g_k, h_{k})^2
+ B_k^{-1} \int_{\de B_r}|\etaa\circ g_k|.
\end{equation}
By assumption $\liminf_k \int_{r_0}^{r_1} \psi_k (r)\, dr < \infty$.
So, by Fatou's Lemma, there is $r\in \left]r_0,r_1\right[$ and
a subsequence, not relabeled, such
that $\lim_k \psi_k (r) < \infty$. Thus, for some $M>0$ we have
\begin{gather}
\int_{\de B_r}\cG(g_k,h_{k})^2\to 0, \label{e:raggio1}\\
\D(h_{k},\de B_r)+ \D (g_k, \de B_r)\leq M,
\label{e:raggio2}\\
\int_{\de B_r}|\etaa\circ g_k| \leq M \, \|\etaa\circ g_k\|_{L^1(B_{r_1})}.\label{e:raggio3}
\end{gather}
In case $B_k =0$ for all $k$ large enough, we define $\psi_k$ dropping the last summand in
\eqref{e:def_psi} and reach the same conclusion.

Let $\zeta^j$ be the blocks of the translating sheets $h_k$ as in
Definition~\ref{d:pacchetti}.
We apply Lemma~\ref{l:approx} to each $\zeta^j$
and find Lipschitz 
functions $\zeta^j_{\eta}$ satisfying the conclusion of the lemma with
$\bar\eps_1 = \bar\eps_1(\eta,M)>0$ (which will be chosen later). We also choose
a standard radial convolution kernel $\varphi$ in $\R^m$ and a small parameter $\bar \rho$ (also
to be chosen later).  
Then, set 
\[
h_{k,\eta}:= \sum_{j=1}^J\llbracket \tau_{y_k^j}\circ \zeta^j_\eta\rrbracket
\quad\mbox{and}\quad
\bar h_{k,\eta}:= \sum_{i=1}^Q\llbracket (h_{k,\eta})_i - \etaa \circ h_{k,\eta} +(\etaa \circ h_k)*\varphi_{\bar\rho}
\rrbracket,
\]
and choose $\bar \rho$ so small that 
\begin{gather}
Q^2\|\etaa \circ h_k - (\etaa\circ h_k)*\varphi_{\bar\rho}\|_{L^2}^2 \leq \bar\eps_1,\label{e:regolarizzazione1}\\
\int_{B_r} \left(|D(\etaa \circ h_{k})|^2  - |D(\etaa \circ h_{k}*\varphi_{\bar\rho})|^2\right) \leq \bar \eps_1.\label{e:regolarizzazione2}
\end{gather}
Note that this is possible because, from the fact that $h_k$ is a sequence of translating sheets, it follows that
$\eta \circ h_k (x) = F(x) + p_k$ for some $F \in W^{1,2}$ and a sequence of vectors $p_k \in \R^n$.
Therefore $(\eta \circ h_k)*\varphi_{\bar\rho} = F *\varphi_{\bar\rho} + p_k$ and
$D(\eta \circ h_k)*\varphi_{\bar\rho} = DF *\varphi_{\bar\rho}$, and \eqref{e:regolarizzazione1}
and \eqref{e:regolarizzazione2} follows if $\bar \rho$ is sufficiently small by the usual convolution estimates.
In particular by very rough estimates,
\begin{gather}
\|\cG(g_k,\bar h_{k,\eta})\|_{L^2}\stackrel{\eqref{e:regolarizzazione1}}{\leq}
\|\cG(g_k,h_k)\|_{L^2} +2 \|\cG(h_k,h_{k,\eta})\|_{L^2} + \bar\eps_1
\leq o(1) + 3\,\bar\eps_1,\label{e:z1}\\
\D(\bar h_{k,\eta},\de B_r) \leq 2\,M +2\,\bar\eps_1\label{e:z2}
\end{gather}
and
\begin{align}\label{e:meno media}
\D(\bar h_{k,\eta},B_r) & = \sum_i \int_{B_r} \left| D(h_{k,\eta})_i - D(\etaa \circ h_{k,\eta}) + D(\etaa \circ h_{k}*\varphi_{\bar\rho})\right|^2\notag\\
& = \int_{B_r} \left( |D h_{k,\eta}|^2 - Q|D(\etaa \circ h_{k,\eta})|^2 + Q |D(\etaa \circ h_{k}*\varphi_{\bar\rho})|^2 \right)\notag\\
& = \D(h_{k,\eta},B_r) + 
Q\int_{B_r} \left(|D(\etaa \circ h_{k})|^2  - |D(\etaa \circ h_{k,\eta})|^2\right)\notag\\
&\quad+ 
Q\int_{B_r} \left( |D(\etaa \circ h_{k}*\varphi_{\bar\rho})|^2 - 
|D(\etaa \circ h_{k})|^2 \right)\notag
\allowdisplaybreaks\\
&\stackrel{\eqref{e:approx interior}, \eqref{e:regolarizzazione2}}{\leq}  \D(h_{k,\eta},B_r) + 2\,Q\,\bar \eps_1.
\end{align}
We can then apply
Lemma~\ref{l:interpolation} to $\bar h_{k,\eta}$ and $g_k$ with
$\bar\eps_2 = \bar\eps_2(\eta, M)>0$, and get (up to subsequences) maps
$H_k$ satisfying $H_k|_{\de B_r}=g_k|_{\de B_r}$ and
\begin{align}\label{e:guadagno finale}
\D\left(H_k, B_{r}\right)
&\leq \D\left(\bar h_{k,\eta},B_{r}\right)
+\bar\eps_2\,\D\left(\bar h_{k,\eta},\de B_r\right)+
\bar\eps_2\,\D(g_k,\de B_{r})+\frac{C_0}{\bar\eps_2} \int_{\de B_{r}}\cG\left(\bar h_{k,\eta},g_k\right)^2\notag\\
&\leq \D(h_k, B_r)+ Q \bar \eps_1 + 3\,\bar\eps_2\,(M+\bar \eps_1) + 3\,C_0\, \bar\eps_2^{-1}\bar\eps_1
\notag
\end{align}
where in the last line we have used \eqref{e:raggio1}, \eqref{e:raggio2} and \eqref{e:z1} - \eqref{e:meno media}. 
An appropriate choice of the parameters $\eps_1$ and $\eps_2$ gives the desired bound $\D \left(H_k, B_{r}\right)
\leq \D(h_k, B_r) + \eta$.

Observe next that, by construction, $\limsup_k \Lip (\bar h_{k, \eta}) \leq C^*$, for some
constant which depends on $\eta$ and the two sequences, but not on $k$. Moreover,
\[
\|\cG (\bar h_{k, \eta}, g_k)\|_{L^\infty (\partial B_r)} \leq \|\cG (\bar h_{k, \eta}, g_k)\|_{L^2(\partial B_r)}
+ C \Lip (g_k) + C \Lip (\bar h_{k, \eta})\, .
\] 
Thus \eqref{e:raccordo1} follows from \eqref{e:lip approx}.

Finally, \eqref{e:raccordo2} follows from the
Poincar\'e inequality applied to $\cG(H_k,g_k)$ (which vanishes identically on $\partial B_r$), and
\eqref{e:raccordo3} follows from
\eqref{e:media_per_bu}, because of \eqref{e:raggio3} and $\|\etaa\circ \bar h_{k, \eta}\|_{L^1(B_{r})}
= \|(\etaa\circ h_k)*\varphi_{\bar\rho}\|_{L^1(B_{r})} \leq \|\etaa \circ h_k\|_{L^1(B_{r_1})}$ if
$\bar \rho$ is also chosen small enough such that $r+\bar\rho < r_1$.
\end{proof}

\section{Harmonic approximation}\label{s:o(E)}
In what follows we will always apply Proposition~\ref{p:max} with
$\delta_{11} = E^{2\beta}$ and under a certain scaling of $\bA$.

\begin{definition}[$E^\beta$-Lipschitz approximation]\label{d:Lip-approx}
Let $\beta \in \left(0, \frac{1}{2m}\right)$, $T$ be as in Proposition~\ref{p:max} 
such that $32 E^{(1-2\beta)/m} < 1$ and $s\bA \leq E^{{\sfrac{1}{4}}+\delta}$ for some $\delta>0$.
If the coordinates are fixed as in Remark~\ref{r:Psi}, the map $u$ given by Proposition~\ref{p:max}
for $\delta_{11} = E^{2\beta}$ is then called the
\textit{$E^{\beta}$-Lipschitz approximation of $T$ in $\bC_{3s}(x)$} and will be denoted by $f$.
\end{definition}

In this section we prove that, if $T$ is also area minimizing,
the corresponding $E^{\beta}$-Lipschitz approximation is close to a $\D$-minimizing function $w$. This comes with an $o(E)$-improvement of the estimates in Proposition~\ref{p:max}.

\begin{theorem}[First harmonic approximation]\label{t:o(E)}
For every $\eta_1, \delta>0$ and every $\beta\in (0, \frac{1}{2m})$, there exist constants $\eps_{12},C_{12}>0$ with the following property.
Let $T$ be as in Assumption \ref{ipotesi_base} in $\bC_{4s}(x)$ and assume it is area minimizing.
If $E =\bE(T,\bC_{4s}(x))\leq \eps_{12}$ and $s\bA \leq E^{{\sfrac{1}{4}}+\delta}$, then the
$E^{\beta}$-Lipschitz approximation $f$ in $\bC_{3s}(x)$ satisfies
\begin{equation}\label{e:few energy}
\int_{B_{2s} (x)\setminus K}|Df|^2\leq \eta_1 E\,\omega_m\,(4\,s)^m = \eta_1\,\e_T(B_{4s}(x)).
\end{equation}
Moreover, if we consider the coordinates of Remark \ref{r:Psi}, there exists a $\D$-minimizing function $u: B_{2s} (x)\to \Iq (\R^{\bar{n}})$
such that the map $B_{2s}(x) \ni y \mapsto w = (u, \Psi (y,u))$ satisfies
\begin{gather}
s^{-2} \int_{B_{2s}(x)}\cG(f,w)^2+
\int_{B_{2s}(x)}\big(|Df| - |Dw|\big)^2 \leq \eta_1 E \,\omega_m\,(4\,s)^m=
\eta_1\, \e_T(B_{4s}(x))\, ,\label{e:quasiarm}\\
\int_{B_{2s}(x)} |D(\etaa\circ f) - D (\etaa\circ w)|^2 \leq \eta_1 E \,\omega_m\,(4\,s)^m = \eta_1\, \e_T (B_{4s}(x))\, .\label{e:quasiarm_media}
\end{gather}
\end{theorem}

\begin{remark}[Isoperimetric inequality]\label{r:isop} 
If $S\subset \R^{m+n}$ is an integral current of dimension $m-1$ with $\de S=0$, then there
is an $m$-dimensional integral current $R\subset \R^{m+n}$ such that $\partial R = S$ and 
$\mass (R) \leq C \mass (S)^{m/(m-1)}$,
where the constant $C$ is only dimensional (see \cite[Theorem 30.1]{Sim}). It is also well-known that, when
$\supp (S)\subset \Sigma$ and $\Sigma$ is as in Assumption \ref{ipotesi_base} the same
inequality holds for some $\bar{R}$ with $\supp (\bar{R})\subset \Sigma$ and $\partial \bar{R}= S$, with a dimensional constant $C$ which
depends additionally on the constant $c_0$. This can be easily seen as follows: let $\bq: \R^{m+n} \to
\R^{m+\bar{n}}$ be the orthogonal projection and $\Lambda: \R^{m+n} \to \Sigma$ be the map $\Lambda (p) =
(\bq (p), \Psi (\bq (p)))$. $\Lambda$ is a global Lipschitz retraction of $\R^{m+n}$ onto $\Sigma$
which is the identity on $\Sigma$: thus we can simply set $\bar{R} = \Lambda_\sharp R$. 
\end{remark}

\begin{remark}[Taylor expansion of the mass]\label{r:taylor}
There are dimensional constants $c,C>0$ such that the following holds. Let $V\subset\R^m$ be a bounded
measurable set and let $u:V \to \Iq(\R^n)$ be a Lipschitz
function with $\Lip(u) \leq c$.
Denote by $\mathbf{G}_u$ the integer rectifiable current associated to the graph of $u$ as in \cite[Definition~1.10]{DS2}.
Then, the following Taylor expansion
of the mass of $\mathbf{G}_{u}$ holds:
\[
\mass(\bG_u) = Q\,|V| +\int_V\frac{|Du|^2}{2} + \int_V \sum_i R(Du_i) ,
\]
where $R:\R^{n\times m} \to R$ is a $C^1$ function satisfying $|R(D)| = |D|^3\,L(D)$ for some
positive function $L$ such that $L(0)=0$ and $\Lip(L) \leq C$.
This Taylor expansion is proven in \cite[Corollary~3.3]{DS2} (although the
corollary is stated for $V$ open, the proof works obviously when $V$ is merely measurable).
\end{remark}

\begin{remark}\label{r:prime stime}
There exists a dimensional constant $c>0$ such that,
if $E\leq c$, then the $E^\beta$-Lipschitz approximation satisfies the following estimates:
\begin{gather}
\Lip(f) \leq C\, E^\beta,\label{e:banana1}\\
\int_{B_{3s}(x)} |Df|^2 \leq C\, E\,s^m.\label{e:banana2}
\end{gather}
Indeed \eqref{e:banana1} follows from Proposition~\ref{p:max},
Remark~\ref{r:Psi} and
$\|D\Psi\|_0 \leq C (E^{\sfrac{1}{2}} + \bA)\leq C\,E^\beta$ by the choice of $\beta$ and the scaling of $\bA$.
While \eqref{e:banana2} follows from Remark~\ref{r:taylor} since for $E$ sufficiently small
\[
\int_{B_{3s}(x)} \sum_i R(Df_i) \leq C\,E^{2\beta} \int_{B_{3s}(x)}|Df|^2 < \frac14 \int_{B_{3s}(x)}|Df|^2,
\]
and therefore
\begin{align*}
\int_{B_{3s}(x)} |Df|^2 \leq &C\left(\mass(\bG_f\res \bC_{3s} (x)) - Q \, \omega_m \, (3\,s)^m \right)\\
\leq &C\left(\mass(T\res \bC_{3s} (x)) - Q \, \omega_m \, (3\,s)^m \right) +
C\,\mass(\bG_f\res (B_{3s}(x) \setminus K)\times \R^n)\\
\leq &C\, E\,s^m + C\,E^{2\beta}\,|B_{3s}(x) \setminus K| \leq C\,E\,s^m.
\end{align*}
\eqref{e:banana2} is therefore a rather simple corollary of the ``maximal function truncation'' argument employed in Proposition \ref{p:max}. Other approximation schemes give instead worse bounds for the Lipschitz constant of the approximating map, cf. for instance \cite[Theorem 5.1.1]{Sim}.
\end{remark}

\begin{proof}[Proof of Theorem \ref{t:o(E)}]
By rescaling and translating, it is not restrictive to assume that
$x=0$ and $s=1$. Thus, by Remark \ref{r:Psi} we can assume $\Psi (0)=0$,
$\|D\Psi\|_0 \leq C (E^{\sfrac{1}{2}} + \bA)$ and $\|D^2 \Psi\|_0 \leq \bA$.
The proof of \eqref{e:few energy} is by contradiction.
Assume there exist a constant $c_1>0$, a sequence of
currents $(T_k)_{k\in\N}$ satisfying Assumption~\ref{ipotesi_base} and area minimizing,
ambient manifolds $\Sigma_k$ (parametrized by $\Psi_k$, with 
second fundamental forms bounded by $\bA_k$) and
corresponding $E_k^{\beta}$-Lipschitz approximations $(f_k)_{k\in\N}$ 
such that
\begin{equation}\label{e:contradiction}
E_k:=\bE(T_k,\bC_4)\to 0\, ,
\quad \bA_k \leq E_k^{\sfrac{1}{4}+\delta}\quad \text{and}\quad
\int_{B_2\setminus K_k}|Df_k|^2\geq c_1\, E_k,
\end{equation}
where $K_k:= \{x\in B_3 : \bmax\be_{T_k}(x) < E_k^{2\beta}\}$.
Set $\Gamma_k:=\{x\in B_4 : \bmax\be_{T_{k}}(x)\leq 2^{-m}E_k^{2\beta}\}$ and
observe that $\Gamma_k\cap B_3\subset K_k$.
From Proposition \ref{p:max}, it follows that
\begin{equation}\label{e:lip(1)}
\Lip (f_k) \leq C E_k^{\beta},
\end{equation}
\begin{equation}\label{e:lip(2)}
|B_{r} \setminus K_k| \leq C E_k^{-2\beta} \e_T
\bigl(B_{r+r_0 (k)}\setminus \Gamma_k\bigr) \quad \mbox{for every $r\leq 3$}\, ,
\end{equation}
where $r_0 (k)= 16 \, E_k^{(1-2\beta)/m} <\frac{1}{2}$. We also assume 
\begin{equation}\label{e:parametrizzazioni}
\Psi_k (0)=0\qquad\mbox{and}\qquad \|D\Psi_k\|_0 + \|D^2\Psi_k\|_0 \leq C E_k^{\sfrac{1}{4}+\delta}\, .
\end{equation}
Then, \eqref{e:contradiction}, \eqref{e:lip(1)} and \eqref{e:lip(2)} give
\begin{equation*}
c_1\, E_k\leq \int_{B_2\setminus K_k}|Df_k|^2\leq
C\,\e_{T_k}(B_{s}\setminus \Gamma_k)\quad \forall \; s\in\left[\textstyle{\frac{5}{2}},3\right].
\end{equation*}
Setting $c_2:=c_1/(2C)$, we have $2c_2 E_k \leq \be_{T_k} (B_s \setminus \Gamma_k)=
\be_{T_k} (B_s) - \be_{T_k} (B_s \cap \Gamma_k)$, thus leading to
\begin{equation}\label{e:improv}
\e_{T_k}(\Gamma_k\cap B_s)\leq \e_{T_k}(B_s)-2\,c_2\,E_k.
\end{equation}
Next observe that $\omega_m 4^m E_k = \e_{T_k} (B_4) \geq \e_{T_k} (B_s)$. 
Therefore, by the Taylor expansion in Remark~\ref{r:taylor},
\eqref{e:improv}
and $E_k\downarrow 0$, it follows that, for every $s\in\left[5/2,3\right]$,
\begin{align}\label{e:improv2}
\int_{\Gamma_k\cap B_s}\frac{|Df_k|^2}{2} & \leq (1+C\,E_k^{2\beta})\,\e_{T_k}(\Gamma_k\cap B_s)\notag\\
&\leq (1+C\,E_k^{2\beta}) \,\Big(\e_{T_k}(B_s)-2\,c_2\,E_k\Big)
\; \leq \e_{T_k}(B_s)-c_2\,E_k.
\end{align}
Our aim is to show that \eqref{e:improv2} contradicts the minimizing property
of $T_k$.
To construct a competitor we write
$f_k (x) = \sum_i \a{f_k^i (x)} \in \Iq(\R^{\bar n} \times \R^l)$,
and denote by $(f^i_k)' (x)$ the first $\bar{n}$ components of the points $f_k^i (x)$. This induces a map
$f'_k := \sum_i \a{(f^i_k)'}$ taking values into $\Iq (\R^{\bar n})$. Observe that, since $f_k^i (x)$ are indeed
points of the manifold $\Sigma_k$
\begin{equation}\label{e:coordinate_divise}
f_k=\sum_i \a{((f_k^i)' (x), \Psi_k (x, (f_k^i)' (x)))}\, .
\end{equation}
We consider $g_k:={E_k}^{-\sfrac{1}{2}}f'_k$.
Since by Remark~\ref{r:prime stime} $\sup_k \D (g_k, B_3) <\infty$ and $|B_3\setminus \Gamma_k|\to0$,
by Proposition~\ref{p:cc} we can find
a subsequence (not relabelled) of translating sheets $h_k$ satisfying
\eqref{e:cc2} - \eqref{e:cc3} and $\|\cG (g_k, h_k)\|_{L^2 (B_3)}\to 0$.
In particular, we are in the position to apply Proposition~\ref{p:raccordo}
to $g_k$ and $h_k$, with $r_0= \frac52$, $r_1=3$ and $\eta= \frac{c_2}{4}$,
and find
$r\in \left(\frac{5}{2},3\right)$ and competitor functions $H_k$ satisfying
$H_k\vert_{B_3 \setminus B_r} = g_k\vert_{B_3 \setminus B_r}$,
\begin{gather}
\D(H_{k}, B_{r}) \leq  \D(h_{k}, B_{r}) + \frac{c_2}{4},\label{e:guadagno finale}\\
\Lip(H_{k})\leq C^*\, E_k^{\beta-\sfrac12}\label{e:Lipschitz_sottile}\\
\|\cG(H_{k}, g_{k})\|_{L^2(B_{r})} \leq C^*\, \D(g_k, B_r) + C\, \D (H_k, B_r) \leq M <\infty.
\end{gather}
Moreover, Proposition \ref{p:cc} implies that, for $k$ is large enough, 
\begin{equation}\label{e:guadagno_finale2}
\D (h_{k}, B_{r}) \leq \D (g_k, B_{r} \cap \Gamma_k) + \frac{c_2}{4} \stackrel{\eqref{e:improv2}}{\leq} \frac{\be_{T_k} (B_{r})}{E_k} - \frac{3c_2}{4} E_k\, .
\end{equation}
Note that \eqref{e:Lipschitz_sottile} follows from \eqref{e:raccordo1} observing that $E_k^{\beta-\sfrac12}\uparrow \infty$: 
thus $C^*$ depends on $c_2$ and the two chosen sequences, but not on $k$. From now on, although
this and similar constants are not dimensional, we will keep denoting them by $C$, with the understanding that they do not depend on $k$.
Note that, from \eqref{e:lip(1)} and \eqref{e:lip(2)}, one gets
\begin{align}\label{e:diff mass}
\|T_k-\mathbf{G}_{f_k}\|(\bC_{3})&\leq
\| T_k\| ((B_{3}\setminus K_k)\times\R^{n} )
+\|\mathbf{G}_{f_k}\| ((B_{3}\setminus K_k)\times\R^{n})\notag\\
&\leq Q\,|B_{3}\setminus K_k|+ E_k+ Q\, |B_3\setminus K_k|+ 
C\,|B_{3}\setminus K_k|\, \Lip(f_k)\notag\\
&\leq  E_k+ C\,E_k^{1-2\beta}\leq C\,E_k^{1-2\beta}.
\end{align}
Let $(z,y)$ be coordinates on $\R^m\times \R^n$ and
consider the function $\ph (z,y) = |z|$ and the slice $\la T_k-\mathbf{G}_{f_k}, \ph,r\ra$. Observe that, by the coarea
formula and Fatou's Lemma, 
\[
\int_r^3 \liminf _kE_k^{2\beta -1} \mass (\la T_k-\mathbf{G}_{f_k}, \ph,s\ra) \, ds \leq  \liminf_k E_k^{2\beta-1} \|T_k - \mathbf{G}_{f_k}\| (\bC_3)
\leq C\, .
\]
Therefore, for some $\bar{r} \in (r,3)$ and a subsequence, not relabeled, 
$\mass\big(\la T_k-\mathbf{G}_{f_k}, \ph, \bar{r} \ra\big)\leq C\,E_k^{1-2\beta}$.

Let now $v_k:={E_k}^{\sfrac{1}{2}}\, H_k\vert_{B_{\bar r}}$,
$u_k := (v_k, \Psi_k (x,v_k))$ 
and consider the current
$Z_k := \mathbf{G}_ {u_k}\res \bC_{\bar r}$.
Since $u_k|_{\de B_{\bar r}} = f_k|_{\de B_{\bar r}}$, one gets
$\partial Z_k = \la \mathbf{G}_{f_k},\ph,{\bar r}\ra$ and, hence,
$\mass (\partial (T_k\res \bC_{\bar r} - Z_k))\leq C E_k^{1-2\beta}$. 
We define
\begin{equation}\label{e:competitor_current}
S_k = T_k \res (\bC_4\setminus \bC_{\bar r}) + Z_k + R_k\, .
\end{equation}
where (cp.~Remark \ref{r:isop}) $R_k$ is an integral current supported in $\Sigma_k$ such that 
\[
\partial R_k= \partial (T_k\res \bC_{\bar r} - Z_k)\quad\text{and}\quad\mass (R_k) \leq C
E_k^{\frac{(1-2\beta)m}{m-1}}.
\]
$S_k$ is supported in $\Sigma_k$ and 
$\partial S_k = \partial (T_k\res \bC_4)$.
We now show that, since $\beta<\frac{1}{2m}$, for
$k$ large enough, the mass of $S_k$ is smaller than that of $T_k$.
To this aim we write 
\begin{align*}
& \D (u_k, B_{\bar r}) - \D (f_k , B_{\bar r} \cap \Gamma_k) =
\underbrace{\int_{B_{\bar r}} |D v_k |^2- \int_{B_{\bar r}\cap \Gamma_k} |Df'_k|^2}_{I_1}\nonumber\\
&\quad 
+ \underbrace{\int_{B_{\bar r}} |D (\Psi_k (x, v_k))|^2 - \int_{B_{\bar r}} |D (\Psi_k (x, f'_k))|^2}_{I_2}\, +\underbrace{\int_{B_{\bar r}\setminus \Gamma_k} |D(\Psi_k (x, f'_k))|^2}_{I_3}\notag .
\end{align*}
The first term is estimated by \eqref{e:guadagno finale} and \eqref{e:cc2}: recalling that $v_k =  E_k^{\sfrac{1}{2}} H_k$ and $f'_k = E_k^{\sfrac{1}{2}} g_k$ (but also that the two functions coincide on $B_{\bar r}\setminus B_r$) we achieve $I_1 \leq \frac{c_2}{2} E_k$ for $k$ large enough.
For what concerns the second, we proceed as follows. First we write 
\[
I_2 = \sum_i\int_{B_{\bar r}} (D(\Psi_k(x,u_k(x))_i-D(\Psi_k(x,f'_k(x))_i):
(D(\Psi_k(x,u_k(x))_i+D(\Psi_k(x,f'_k(x))_i) .\notag
\]
Next, recalling the chain rule \cite[Proposition~1.12]{DS1}, we get
\begin{align}
\big|D(\Psi_k (x,u_k(x))_i + D(\Psi_k (x,f'_k(x))_i\big|&\leq C \|D_x\Psi_k\|_0 + C \|D_u \Psi_k\|_0 (\Lip (u_k) + \Lip (f'_k))\nonumber\\
&\stackrel{\eqref{e:parametrizzazioni}}{\leq} C E_k^{\sfrac{1}{4}+\delta}\, .\notag
\end{align}
Using the letter inequality, the chain rule and \eqref{e:parametrizzazioni}, once again we achieve
\begin{align}
I_2 \leq & C\, E_k^{\sfrac14+\delta} \int_{B_{\bar r}} \Big(\sum_i |D_x\Psi_k (x, v_k^i(x)) - D_x\Psi_k (x, (f^i_k)'(x))|\notag\\
& \qquad \qquad\qquad + \|D_u \Psi_k\|_0 \left(|Dv_k| + |D f'_k|\right)\Big)\notag\\
\leq{}& C\, E_k^{\sfrac{1}{4} + \delta} \|D^2\Psi_k\|_0 \int_{B_{\bar r}} \cG(v_k,f'_k) + C\,E_k^{\sfrac{1}{2}+2\delta} \int_{B_{\bar r}} \left(|Dv_k| + |Df'_k|\right)\notag\\
\leq{}& C\, E_k^{\sfrac{1}{2}+2\delta} \, E_k^{\sfrac{1}{2}} + C\,E_k^{1+2\delta} \leq C E_k^{1+2\delta}\, .
\end{align}
Finally, $I_3\leq C \|D\Psi_k\|_\infty^2 |B_3\setminus \Gamma_k| \leq C E_k^{1+\beta}$.
Thus, for $k$ large enough we achieve $I_2+I_3 \leq \frac{c_2}{4} E_k$, thereby reaching $\D (u_k,B_{\bar r})
- \D (f_k,B_{\bar r}\cap \Gamma_k) \leq \frac{3c_2}{4} E_k$.
Hence,
\begin{align}
\mass (S_k) - \mass (T_k)
\leq{}&
\mass(Z_k) + C\, \mass(R_k) - \mass(T_k \res \bC_{\bar r})\nonumber\\
\leq {}& Q\,|B_{\bar r}|+\int_{B_{\bar r}}\frac{|Du_k|^2}{2}+ C\,E_k^{1+2\,\beta} +
C\,E_k^{\frac{(1-2\,\beta)m}{m-1}} - Q |B_{\bar r}| - \e_{T_k} (B_{\bar r})\allowdisplaybreaks\nonumber\\
\leq{}& \int_{B_{\bar r}\cap {\Gamma_k}}\frac{|Df_k|^2}{2} +
\frac{3}{4} \, c_2\, E_k + C\,E_k^{1+2\,\beta}+
C\,E_k^{\frac{(1-2\,\beta)m}{m-1}} - \e_{T_k} (B_{\bar r})\allowdisplaybreaks\nonumber\\
\stackrel{\eqref{e:improv2}}{\leq}& 
-\frac{c_2\, E_k}{4}
+ C\,E_k^{1+\beta}+ C\,E_k^{\frac{(1-2\,\beta)m}{m-1}}<0, \label{e:argomento_chiave}
\end{align}
as soon as $E_k$ is small enough.
This gives the desired contradiction and proves \eqref{e:few energy}.

\medskip

For what concerns \eqref{e:quasiarm} and \eqref{e:quasiarm_media}, we argue similarly. Without loss of generality we assume $x=0$ and $s=1$. Hence, we
let $(T_k)_k$, $(\Sigma_k)_k$ and $(\Psi_k)_k$ be sequences with vanishing $E_k:=\bE (T_k, \bC_4)$ and satisfying \eqref{e:parametrizzazioni}, but contradicting \eqref{e:quasiarm} or \eqref{e:quasiarm_media}. So, being $f_k$ the $E_k^\beta$-Lipschitz approximations,
we know that, for any sequence of $\D$-minimizing functions $\bar u_k$ which we might choose, when we set 
$w_k = (\bar u_k, \Psi_k (x, \bar u_k))$ we will have
\begin{equation}\label{e:contra_harm}
\liminf_k \underbrace{E_k^{-1} \int_{B_2} \big(\cG ( f_k, w_k)^2 +  (|D f_k| - |D w_k|)^2 + |D (\etaa\circ f_k - \etaa\circ w_k)|^2\big)}_{=:I(k)} > 0\, .
\end{equation}
As in the previous argument we introduce the maps $f'_k$ satisfying \eqref{e:coordinate_divise}, the normalized functions $g_k = E_k^{-\sfrac{1}{2}} f'_k$ and, after extraction of a subsequence, the translating sheets $h_k$ satisfying
\eqref{e:cc2} - \eqref{e:cc3} and $\|\cG (g_k, h_k)\|_{L^2 (B_3)}\to 0$. We next claim that
\begin{itemize}
\item[(i)] $\lim_k \int_{B_2} |Dg_k|^2 = \int_{B_2} |Dh_{k_0}|^2$, for any $k_0$
(recall that $\int_{B_2} |Dh_{k}|^2$ is constant);
\item[(ii)] $h_k$ is $\D$-minimizing in $B_2$.
\end{itemize}
If (i) is false, then there is a positive constant $c_2$ such that, for any $r\in [5/2,3]$, 
\begin{equation}\label{e:sostituto}
\int_{B_r} \frac{|Dh_k|^2}{2} \leq \int_{B_r} \frac{|Dg_k|^2}{2} - c_2\leq \frac{\e_{T_k} (B_r)}{E_k}-\frac{c_2}{2},
\end{equation}
provided $k$ large enough (where the last inequality is again an effect of the Taylor expansion
of Remark~\ref{r:taylor}).
We next define the competitor currents $S_k$ as in the argument leading to \eqref{e:argomento_chiave}:
this latter inequality is reached thanks to \eqref{e:sostituto}, which substitutes \eqref{e:improv2} and \eqref{e:guadagno_finale2}. On the other hand
\eqref{e:argomento_chiave} contradicts the minimizing property of $T_k$.
If (ii) is false, then $h_k$ is not $\D$-minimizing in $B_{2}$.
This implies that one of the $\zeta^j$ in the translating sheets $h_k$
is not $\D$-minimizing in $B_{2}$.
Indeed, in the opposite case, by \cite[Theorem 3.9]{DS1}, $\|\cG (\zeta^j, Q\a{0})\|_{C^0 (B_{2})} <\infty$ and, since
$h_k = \sum_i \llbracket \tau_{y_k^i} \circ \zeta^i\rrbracket$ and
$|y_k^i - y^j_k|\to \infty$ for $i\neq j$, by the maximum principle of \cite[Proposition 3.5]{DS1}, $h_k$ would be $\D$-minimizing.
Thus, for some $\zeta^j$ we can find a competitor $\hat \zeta^j$
with less energy in the ball $B_2$. 
So the functions $F_k = \sum_j \llbracket\tau_{y^j_k} \circ \hat{\zeta}^j\rrbracket$ satisfy, for any $r\in [5/2,3]$,
\begin{equation}\label{e:sostituto2}
\int_{B_r} \frac{|DF_k|^2}{2} \leq 
\int_{B_r}\frac{ |Dh_k|^2}{2} -c_2 = \lim_k \int_{B_r} \frac{|Dg_k|^2}{2} -\,c_2\leq
\frac{\e_T(B_r)}{E_k}-\frac{c_2}{2}\, 
\end{equation}
provided $k$ is large enough (here $c_2>0$ is some constant independent of $r$ and $k$). On the other hand $F_k = h_k$ on $B_3\setminus B_{5/2}$
and therefore $\|\cG (F_k, g_k)\|_{L^2 (B_3\setminus B_{5/2})} \to 0$.
We then construct the competitor current $S_k$ of \eqref{e:competitor_current}: this time we use, however, the map $F_k$ in place of $h_k$ to construct $H_k$ via Proposition \ref{p:raccordo} and we reach the contradiction \eqref{e:argomento_chiave} using 
\eqref{e:sostituto2} in place of \eqref{e:improv2} and \eqref{e:guadagno_finale2}. 

We next set $\bar u_k := E_k^{\sfrac{1}{2}} h_k$ and we aim at showing that, for $w_k = (\bar u_k, \Psi_k (x, \bar u_k))$, $I(k)\to 0$, 
a contradiction to \eqref{e:contra_harm}.
Observe first that, by $\|\cG (g_k, h_k)\|_{L^2} \to 0$, we have $D (\xii\circ g_k) - D (\xii\circ h_k) \weak 0$ in $L^2$
(recall the definition
of $\xii$ in Section \ref{ss:xii}). On the other hand, recall that
$D (\xii \circ h_k)$ is actually a single function, independent of $k$, because $h_k$ is a sequence of translating sheets. So, (i)
and the identities $|D (\xii \circ g_k)|= |Dg_k|$, $|D (\xii\circ h_k)| = |D  h_k|$ imply that $D (\xii\circ g_k) - D (\xii\circ h_k)$ converge
strongly to $0$ in $L^2$. If we next set $\hat{h}_k = \sum_i \a{h_k^i - \etaa\circ h_k}$ and $\hat{g}_k = \sum_i \a{g_k^i - \etaa\circ g_k}$,
we obviously have $\|\cG (\hat{h}_k, \hat{g}_k)\|_{L^2} + \|\etaa \circ h_k - \etaa \circ g_k\|_{L^2} \to 0$. Recall however that the Dirichlet
energy enjoys the splitting 
\[
\D (g_k) = Q \int |D (\etaa \circ g_k)|^2 + \D (\hat{g}_k)
\qquad \D (h_k) = Q\int |D (\etaa \circ h_k)|^2 + \D (\hat{h}_k)\, .
\]
So (i) implies that the Dirichlet energies of $\etaa\circ g_k$ and $\hat{g}_k$ converge, respectively, to those of 
$\etaa\circ h_k$ and $\hat{h}_k$ (which, we recall again, are independent of $k$ because the $h_k$'s are translating sheets).
We thus infer that $D (\etaa \circ h_k) - D (\etaa \circ g_k)$ converges to $0$ strongly in $L^2$.

Coming back to $w_k$ we observe that
\begin{equation}\label{e:vanishes1}
E_k^{-1} \int_{B_2} \cG (w_k, f_k)^2 \leq  (2+ \Lip (D\Psi)^2) E_k^{-1} \int_{B_2} \cG (\bar u_k, f_k')^2
= C \int_{B_2} \cG (h_k, g_k)^2 \to 0\, .
\end{equation}
So,
\begin{align}
\limsup_k I(k) \leq &
2\limsup_k \int_{B_2} \big( |Dg_k| - |Dh_k|)^2 + |D (\etaa \circ g_k - \etaa \circ h_k)|^2\big)\nonumber\\
& + C(Q) \limsup_k E_k^{-1} \int_{B_2} \cG (D (\Psi(x, f'_k)), D (\Psi (x, \bar u_k)))^2\nonumber\\
\leq &C \limsup_k  E_k^{-1} \int_{B_2} \cG (D (\Psi(x, f'_k)), D (\Psi (x, \bar u_k)))^2 = \limsup_k E_k^{-1} J(k)\, .\label{e:vanishes2}
\end{align}
Recalling the chain rule of \cite[Proposition 1.12]{DS1}, we have
\begin{align*}
& D (\Psi (x, f'_k)) (x) = \sum_i \a{D_x \Psi (x, (f^i_k)' (x)) + D_v \Psi (x, (f^i_k)' (x)) \cdot D(f^i_k)' (x)}\nonumber\\
& D (\Psi (x, \bar u_k)) (x) = \sum_i \a{D_x \Psi (x, \bar u_k^i (x)) + D_v \Psi (x, \bar u_k^i (x)) \cdot D\bar u_k^i (x)}\, .
\end{align*}
So we can estimate 
\begin{equation*}\label{e:vanishing3}
J (k) \leq C \Lip (D_x \Psi)^2 \int_{B_2} \cG (f'_k, \bar u_k)^2 + C \|D\Psi\|_0^2 \int_{B_2} (|Df'_k|^2 + |D\bar u_k|^2)
\stackrel{\eqref{e:parametrizzazioni}}{\leq} C E_k^{\sfrac{3}{2} + 2\delta}\, .
\end{equation*}
We therefore conclude that $E_k^{-1} J (k)\to 0$ and thus $I(k) \to 0$, which contradicts \eqref{e:contra_harm}.
\end{proof}

\section{Gradient $L^p$ estimate}\label{s:Lp-estimate}
In this section we prove Theorem~\ref{t:higher1}.
The result is a consequence of an higher integrability estimate for the gradient
of Dir-minimizing functions,
the $o(E)$-improved estimate for the excess measure given in Proposition~\ref{p:o(E)}
and a very careful ``covering and stopping radius'' argument 
(cf.~\cite{Sp12} for an exposition in a more elementary context).

\subsection{Higher integrability of the gradient of $\D$-minimizers}\label{s:higher}
Most of the energy of a $\D$-minimizer lies where the gradient is relatively small.
We prove indeed the following a priori estimate 
(cf.~\cite{Sp10} for a different proof and some improvements).

\begin{theorem}[Higher integrability of $\D$-minimizers]
\label{t:hig fct}
There exists $p_{10}>2$ such that, for every 
$\Omega'\subset\subset\Omega \subset\R^{m}$ open domains,
there is a constant $C>0$ such that
\begin{equation}\label{e:hig fct}
\norm{Du}{L^{p_{10}} (\Omega')}\leq C\,\norm{Du}{L^2(\Omega)}\quad \text{for every $\D$-minimizing }\, u\in
W^{1,2}(\Omega,\Iqs).
\end{equation}
\end{theorem}

\begin{proof} The statement is a corollary of Proposition~\ref{p:est2p} below and a Gehring type lemma, cf.~\cite[Proposition 5.1]{GiMo}.
\end{proof}

\begin{proposition}\label{p:est2p}
Let $\frac{2\,(m-1)}{m}<p_{11} <2$.
Then, there exists $C=C(m,n,Q,p_{11})$ such that, for every $u:\Om\to\Iq$ $\D$-minimizing, the following holds
\begin{equation*}
\left(\mint_{B_r(x)}|Du|^2\right)^{\sfrac{1}{2}}\leq C\left(\mint_{B_{2r}(x)}|Du|^{p_{11}}\right)^{\sfrac{1}{p_{11}}}
\quad\forall\;x\in\Om,\;\forall\;r<\min\big\{1,\dist(x,\de\Om)/2\big\}.
\end{equation*}
\end{proposition}

\begin{proof} Since the estimate is invariant under translations
and rescalings, it is enough to prove it for $x=0$ and $r=1$.  We
assume, therefore $\Omega= B_2$.
Let $u:\Om\to \Iqs$ be $\D$-minimizing and let $F=\xii\circ u:\Om\to \cQ\subset\R^{N}$. Denote by $\bar F\in \R^{N}$ the average
of $F$ on $B_2$.
By Fubini's theorem and the Poincar\'e inequality, there exists $s\in[1,2]$ such
that
\begin{equation*}
\int_{\de B_s}\left(|F-\bar F|^{p_{11}}+|DF|^{p_{11}}\right)
\leq C \int_{B_2}\left(|F-\bar F|^{p_{11}}+|DF|^{p_{11}}\right)
\leq C \|DF\|^{p_{11}}_{L^{p_{11}} (B_2)}.
\end{equation*}
Consider $F\vert_{\de B_s}$.
Since $\frac{1}{2}>\frac{1}{p_{11}}-\frac{1}{2\,(m-1)}$, 
we can use the embedding
$W^{1,{p_{11}}}(\de B_s)\hookrightarrow H^{1/2}(\de B_s)$
(see, for example, \cite{Ada}).
Hence, we infer that
\begin{equation}\label{e:HW}
\norm{F-\bar F}{H^{1/2}(\de B_s)}\;\leq\;
C\,\norm{DF}{L^{p_{11}} (B_2)}.
\end{equation}
Let $\hat F$ be the harmonic extension of $F\vert_{\de B_s}$ in $B_s$.
It is well known (one could, for example, use the result
in \cite{Ada} on the half-space together with a partition of unity) that
\begin{equation}\label{e:harm ext}
\|D\hat F\|_{L^2 (B_s)} \leq C(m)\,\min_{p\in\mathbb R^N} \|\hat{F}-p\|_{H^{1/2} (\partial B_s)} \stackrel{\eqref{e:HW}}{\leq}
C\norm{D F}{L^{p_{11}} (B_2)}\, .
\end{equation}
Consider the map $\ro$ of \cite[Theorem 2.1]{DS1}. Since $\ro\circ \hat F\vert_{\de B_s}
= u\vert_{\de B_s}$ and $\ro\circ \hat F$ takes values in $\cQ$,
by the minimizing property of $u$
and the Lipschitz continuity of $\xii$, $\xii^{-1}$ and $\ro$, we conclude:
\begin{equation*}
\left(\int_{B_1}|Du|^2\right)^{\sfrac{1}{2}}\leq
C\,\left(\int_{B_s}|D\hat 
F|^2\right)^{\sfrac{1}{2}}\leq
C\,\left(\int_{B_2} |DF|^{p_{11}}\right)^{\sfrac{1}{p_{11}}}
= C \left(\int_{B_2} |Du|^{p_{11}}\right)^{\sfrac{1}{p_{11}}}.\qedhere
\end{equation*}
\end{proof}

\begin{remark} Proposition~\ref{p:est2p} can be proved in several different ways, which are based on more common test function
arguments: cf.~the intrinsic proof (i.e.~which does not use the biLipschitz embedding $\xii$)
in \cite{Sp10} or the usual Caccioppoli's inequality for
quasi minima \cite[Theorem~6.5]{Giusti}.
\end{remark}

\subsection{Improved excess estimate}\label{ss:weak alm}
The higher integrability of the Dir-minimizing functions and the harmonic
approximation lead to the following estimate, which we call ``weak''
since we will improve it in Theorem \ref{t:higher}.

\begin{proposition}[Weak excess estimate]\label{p:o(E)}
For every $\eta_{10}>0$, there exists $\eps_{13} >0$ with the following property.
Let $T$ be area minimizing and assume it satisfies Assumption \ref{ipotesi_base} in $\bC_{4s} (x)$.
If $E =\bE(T,\bC_{4s}(x))\leq \eps_{13}$,
then
\begin{gather}\label{e:o(E)1}
\e_T(A)\leq \eta_{10}\, E\,s^m
+ C \,\bA^{2}\,s^{m + 2},
\end{gather}
for every $A\subset B_{s}(x)$ Borel with $|A|\leq \eps_{13}|B_{s}(x)|$.
\end{proposition}
 
\begin{proof}
Without loss of generality, we can assume $s=1$ and $x=0$.
We distinguish the two regimes: $E \leq \bA^2$ and $\bA^2 \leq E$.
In the former, clearly
$\be_T(A) \leq C\,E \leq C\, \bA^{2}$.
In the latter, we let $f$ be the $E^{\frac{1}{4m}}$-Lipschitz approximation of $T$ in $\bC_{3}$
and, arguing as for the proof of Theorem~\ref{t:o(E)},
we find a radius $r\in (1,2)$ and a current $R$ such that
\[
\partial R = \la T- \mathbf{G}_f
, \ph, r\ra \quad\text{and}\quad \mass(R)\leq C E^{(1-\frac{1}{2m})\frac{m}{m-1}}.
\]
Therefore, by the Taylor expansion in Remark~\ref{r:taylor}, we have:
\begin{align}\label{e:T min}
\|T\| (\bC_r)&\leq \mass(\mathbf{G}_f\res\bC_r+R) 
\leq
\|\mathbf{G}_f\| (\bC_r)+C\,E^{\frac{2m-1}{2m-2}} \leq
Q\,|B_r|+\int_{B_r}\frac{|Df|^2}{2}+C\,E^{1+\gamma},
\end{align}
where $\gamma = \frac{1}{2m}$.
On the other hand, using again the Taylor expansion for the part of the current which coincides with the graph
of $f$, we deduce as well that
\begin{align}\label{e:T below}
\|T\| ((B_r\cap K)\times\R^{n}) \geq Q\,|B_r\cap K|+\int_{B_r\cap K}\frac{|Df|^2}{2}-C\,E^{1+\gamma}.
\end{align}
Subtracting \eqref{e:T below} from \eqref{e:T min}, we deduce
\begin{equation}\label{e:out K1}
\e_{T}(B_r\setminus K)\leq \int_{B_r\setminus K}\frac{|Df|^2}{2}+ C E^{1+\gamma}.
\end{equation}
If $\eps_{13}$ is chosen small enough, we infer from \eqref{e:out K1} and \eqref{e:few energy} in 
Theorem~\ref{t:o(E)} that 
\begin{equation}\label{e:out K}
\e_{T}(B_r\setminus K)\leq \eta\,E  + C E^{1+\gamma},
\end{equation}
for a suitable $\eta>0$ to be chosen.
Let now $A\subset B_1$ be such that $|A|\leq \eps_{13}\,\omega_m$.
Combining \eqref{e:out K} with the Taylor expansion, we have
\begin{equation}\label{e:on A1}
\e_T (A)\leq \e_T (A\setminus K)+\int_A \frac{|Df|^2}{2}+  C\,E^{1+\gamma}
\leq \int_A
\frac{|Df|^2}{2}+ \eta\, E  + C E^{1+\gamma}.
\end{equation}
If $\eps_{13}$ is small enough, we can again apply Theorem \ref{t:o(E)}. Using the coordinates
of Remark~\ref{r:Psi}, there is a $\D$-minimizing $u$ such that 
$|Df|$ is close in $L^2$ (with an error $\eta E$) 
to $|Dw|$ with $w = (u,\Psi (x, u))$ and by Remark~\ref{r:prime stime}
$\D (u) \leq C E$. On the other hand $|Dw (x)| \leq (1+ \|D\Psi\|_0) |Du| + \|D\Psi\|_0$. Since 
$\|D\Psi\|_0 \leq C E^{\sfrac{1}{2}}$, by Theorem \ref{t:hig fct} $\||Dw|\|_{L^{p_{10}} (B_1)} \leq C E^{\sfrac{1}{2}}$.
Therefore,
\begin{align}\label{e:on A2}
\e_T (A)&\stackrel{\mathclap{\eqref{e:quasiarm}}}{\leq}
\int_A|Dw|^2+ 3\,\eta\, E  + C E^{1+\gamma}
\leq C\left(|A|^{1-\sfrac{2}{p_{10}}} + \eta\right) \, E+ C E^{1+\gamma}.
\end{align}
Hence, if $\eps_{13}$ and $\eta$ are suitably chosen, \eqref{e:o(E)1} follows
from \eqref{e:on A2}.
\end{proof}
 
\subsection{Proof of Theorem \ref{t:higher1}}
We assume without loss of generality that $E>0$ and divide the proof into two steps.

\medskip

\textit{Step 1.}
There exist constants $\gamma\geq 2^m$ and $\varrho >0$ such that, for every $c\in [1,(\gamma\,E)^{-1}]$ and $s\in[2,4]$ with $\bar{s} = s + 4 \, c^{-\sfrac{1}{m}}\leq 4$, we have
\begin{equation}\label{e:higher d2}
\int_{\{\gamma\, c\, E\leq \d\leq1\}\cap B_s}\d\leq \gamma^{-\varrho}
\int_{\left\{\frac{c\,E}{\gamma}\leq \d\leq1\right\}\cap B_{\bar{s}}}\d + C\,
c^{-\sfrac{2}{m}}\, \bA^{2}.
\end{equation}

In order to prove it, let $N_B$ be the constant in Besicovich's covering theorem \cite[Section 1.5.2]{EG}
and choose $N\in\N$ so large that $N_B < 2^{N-1}$.
Let $\eps_{13}$ be as in Proposition \ref{p:o(E)} when we choose $\eta_{10} = 2^{-2m-N}$, and set
\[
\gamma=\max\{2^m,\eps_{13}^{-1}\}\quad\text{and}\quad \varrho=\min\left\{-\log_\gamma(N_B/2^{N-1}), \frac{1}{2m}\right\}\, .
\]
Let $c$ and $s$ be any real numbers as above.
For almost every $x\in \{\gamma\, c\, E\leq \d\leq1\}\cap B_s$, there exists
$r_x$ such that
\begin{equation}\label{e:stop time}
E(T,\bC_{4r_x}(x))\leq c\,E\qquad
\textrm{and}\qquad
E(T,\bC_{t}(x))\geq c\,E\quad \forall t\in ]0, 4\,r_x[.
\end{equation}
Indeed, since $\d(x)=\lim_{r\to0}E(T,\bC_r(x))\geq\gamma\,c\,E\geq2^mc\, E$ and
\[
E(T,\bC_{t}(x))=\frac{\e_T(B_t(x))}{\omega_m\,t^m}
\leq \frac{4^m\,E}{t^m}\leq c\,E
\quad\text{for }\;t\geq\frac{4}{\sqrt[m]{c}},
\]
we just choose $4r_x = \min \{t \leq 4/\sqrt[m]{c} :
E (T, \bC_{t} (x)) \leq c E\}$.
Note also that $r_x\leq 1/\sqrt[m]{c}$.
Consider the current $T$ in $\bC_{4r_x}(x)$.
Setting $A=\{\gamma\, c\, E\leq \d\}\cap B_{4 r_x}(x)$, we have that
\[
\bE(T,\bC_{4r_x}(x))\leq c\,E\leq \frac{E}{\gamma\,E}\leq 
\eps_{13}\quad\mbox{and}\quad
|A|\leq \frac{c\, E \,|B_{4r_x} (x)|}{\gamma\,c\,E } 
\leq \eps_{13} |B_{4r_x} (x)|.
\]
Hence, we can apply Proposition~\ref{p:o(E)} to $T\res\bC_{4r_x}(x)$ to get
\begin{align}\label{e:Bx}
&\int_{B_{r_x}(x)\cap\{\gamma\,c\,E\leq\d\leq 1\}}\d \leq\int_A \d
\leq \e_T (A)
\leq 2^{-2\,m-N}\,\e_T(B_{4r_x}(x)) + C\, r_x^{m+2} \bA^{2}\notag\\
\leq\; & 2^{-2\,m-N}\,(4\,r_x)^{m}\,\omega_m\,\bE(T,\bC_{4r_x}(x)) + C\, r_x^{m+2}\bA^{2}
\stackrel{\mathclap{\eqref{e:stop time}}}{\leq} \,
2^{-N}\,\e_T(B_{r_x}(x)) + C\, r_x^{m+2} \bA^{2}.
\end{align}
Thus, 
\begin{align}\label{e:Bx2}
\e_T(B_{r_x}(x))&= \int_{B_{r_x}(x)\cap\{\d>1\}}\d+
\int_{B_{r_x}(x)\cap\left\{\frac{c\,E}{\gamma}\leq \d\leq 1\right\}}\d
+\int_{B_{r_x}(x)\cap\left\{\d<\frac{c\,E}{\gamma}\right\}}\d\notag\\
&\leq \int_A \d
+\int_{B_{r_x}(x)\cap\left\{\frac{c\,E}{\gamma}\leq \d\leq 1\right\}}\d
+ \frac{c\,E}{\gamma}\,\omega_m\,r_x^m\notag\\
&\stackrel{\mathclap{\eqref{e:stop time},\,\eqref{e:Bx}}}{\leq}\quad
\left(2^{-N}+\gamma^{-1}\right)\,\e_T(B_{r_x}(x)) + C\, r_x^{m+2}\bA^{2}
+\int_{B_{r_x}(x)\cap\left\{\frac{c\,E}{\gamma}\leq \d\leq 1\right\}}\d.
\end{align}
Therefore, recalling that $\gamma\geq 2^{m}\geq 4$, from \eqref{e:Bx} and \eqref{e:Bx2} we infer:
\begin{align*}
\int_{B_{r_x}(x)\cap\{\gamma\,c\,E\leq\d\leq 1\}}\hspace{-0.1cm}\d &\leq
\frac{2^{-N}}{1-2^{-N}-\gamma^{-1}}\int_{B_{r_x}(x)\cap\left\{\frac{c\,E}{\gamma}\leq \d\leq 1\right\}}\hspace{-0.1cm}\d + C\, r_x^{m+2}\bA^{2}\\
&\leq 2^{-N+1}\int_{B_{r_x}(x)\cap\left\{\frac{c\,E}{\gamma}\leq \d\leq 1\right\}}\hspace{-0.1cm}\d + C\, r_x^{m+2} \bA^{2}.
\end{align*}
By Besicovich's covering theorem, we choose $N_B$ families of disjoint balls $\overline{B}_{r_x}(x)$ whose union
covers $\{\gamma\,c\,E\leq\d\leq1\}\cap B_s$ and, since as already noticed $r_x\leq 1/\sqrt[m]{c}$ for every
$x$, we conclude:
\begin{equation*}
\int_{\{\gamma\,c\,E\leq\d\leq 1\}\cap B_s}\d \leq
N_B\,2^{-N+1}
\int_{\left\{\frac{c\,E}{\gamma}\leq \d\leq 1\right\}\cap B_{s+\frac{2}{\sqrt[m]{c}}}}\d
+ C\,c^{-\frac{2}{m}}\, \bA^{2},
\end{equation*}
which, for the above defined $\varrho$, implies \eqref{e:higher d2}.

\medskip

\textit{Step 2.} We iterate \eqref{e:higher d2} in order to conclude \eqref{e:higher1}.
Denote by $L$ the largest integer smaller than $2^{-1} \left((\log_\gamma E^{-1})-1\right)$,
$s_L = 2$ and recursively $s_k = s_{k+1} + 2\,\gamma^{-\frac{2k}{m}}$ for
$k\in \{L-1, \ldots, 1\}$. Notice that, since $\gamma\geq 2^m$, $s_k< 4$ for
every $k$. Thus, we can apply \eqref{e:higher d2} with
$c= \gamma^{2k}$, $s=s_k$ and $\bar{s}=s_{k-1}$ to conclude 
\begin{equation*}
\int_{\{\gamma^{2k+1}\, E\leq \d\leq1\}\cap B_{s_k}} \d\leq \gamma^{-\varrho} 
\int_{\left\{\gamma^{2k-1}\,E\leq \d\leq1\right\}\cap B_{s_{k-1}}}\d + C\, \gamma^{-\frac{4\,k}{m}}\bA^{2}\, \quad \forall \, k\in  \{ 2, \ldots, L\}\, .
\end{equation*}
In particular, iterating this estimate we get
\begin{align}\label{e:iteration}
\int_{\{\gamma^{2\,k+1}\, E\leq \d\leq1\}\cap B_2}\d&\leq
{\gamma^{-(k-1)\,\varrho}}\int_{\{\gamma\, E\leq \d\leq1\}\cap B_{s_1}}\d +
C\, \bA^{2}\, \sum_{\ell=0}^{k-2}\gamma^{-\left(\frac{4\,(k-\ell)}{m}+\ell\,\varrho\right)}.
\end{align}
Set $A_0=\{\d< \gamma\,E\}$,
$A_k=\{\gamma^{2k-1}\,E\leq \d< \gamma^{2k+1}\,E\}$ for $k=1,\ldots, L$,  
and $A_{L+1}=\{\gamma^{2L+1}\,E\leq \d\leq1\}$.
Since $\cup A_k = \{ \d\leq 1\}$, for $p_1< 1+\frac{\varrho}{2}\leq 1+\frac{1}{m}$,
we conclude: 
\begin{align*}
&\int_{B_2\cap \{\d\leq 1\}} \d^{p_1} = 
\sum_{k=0}^{L+1} \int_{A_k\cap B_2} \d^{p_1}\leq\sum_k
\gamma^{(2\,k+1)\,(p_1-1)}\,E^{p_1-1}\int_{A_k\cap B_2} \d \;\\
\stackrel{\mathclap{\eqref{e:iteration}}}{\leq}&\;
C\sum_k\gamma^{k\,(2\,(p_1-1)-\varrho)}\,E^{p_1}
+ C\sum_k\sum_{\ell=0}^{k-2} 
\gamma^{k(2\,(p_1-1)-\frac{4}{m}) +\ell\,(\frac{4}{m}-\varrho)}E^{p_1-1}\,\bA^{2}\\
\leq&\; C E^{p_1} + C \sum_k \gamma^{k(2(p_1-1)-\varrho)}\,E^{p_1-1}\, \bA^{2}.
\end{align*}

\section{Almgren's approximation theorem}\label{s:main}
In this section we show how Theorem~\ref{t:higher1} gives a simple proof of
the approximation result in Theorem~\ref{t:main}.
The key point is the following theorem.

\begin{theorem}[Almgren's strong excess estimate]\label{t:higher}
There are constants $\eps_{11},\gamma_{11}, C> 0$ (depending on $m,n,\bar n,Q$)
with the following property.
Assume $T$ satisfies Assumption \ref{ipotesi_base} in $\bC_4$ and is area minimizing.
If $E =\bE(T,\bC_4) < \eps_{11}$, then
\begin{equation}\label{e:higher2}
\e_T (A) \leq C\, \big(E^{\gamma_{11}} + |A|^{\gamma_{11}}\big) \left(E+\bA^2\right)
\quad \text{for every Borel }\; A\subset B_{\frac98}.
\end{equation}
\end{theorem}

This estimate complements \eqref{e:higher1} enabling to control
the excess in the region where $\bd > 1$.
We call it strong Almgren's estimate because a similar formula
can be found in the big regularity paper (cf.~\cite[Sections 3.24-3.26 \& 3.30(8)]{Alm})
and is a strengthened version of Proposition~\ref{p:o(E)}.
To achieve \eqref{e:higher2} we construct a suitable competitor to estimate the size of the set $K$ where the graph 
of the $E^{\beta}$-Lipschitz approximation $f$ differs from $T$.
Following Almgren, we embed $\Iq$ in a large Euclidean space,
via a biLipschitz embedding $\xii$. We then regularize $\xii\circ f$ by convolution and project it back onto
$\cQ= \xii (\Iq)$. To avoid loss of energy we need a rather special ``almost projection'' $\ro^\star_\delta$.

\begin{proposition}\label{p:ro*}
For every $\bar{n}, Q\in \N\setminus \{0\}$ there are geometric constants $\delta_0, C>0$ with
the following property. 
For every $\delta\in ]0, \delta_0[$
there is $\ro^\star_\delta:\R^{N(Q,\bar n)}\to \cQ=\xii(\Iq(\R^{\bar n}))$ such that $|\ro^\star_\delta(P)-P|\leq C\,\delta^{8^{-\bar n Q}}$ for all $P \in \cQ$
and, for every $u\in W^{1,2}(\Omega,\R^{N})$, the following holds:
\begin{equation}\label{e:ro*1}
\int |D(\ro^\star_\delta\circ u)|^2\leq
\left(1+C\,\delta^{8^{-\bar{n}Q-1}}\right)\int_{\left\{\dist(u,\cQ)\leq \delta^{\bar nQ+1}\right\}} |Du|^2+
C\,\int_{\left\{\dist(u,\cQ)> \delta^{\bar nQ+1}\right\}} |Du|^2\, .
\end{equation}
\end{proposition}

The proof of Proposition \ref{p:ro*} is postponed to the next section. Here we show Theorem \ref{t:higher}
and hence conclude the Theorems \ref{t:main} and \ref{t:harmonic_final}. Theorem
\ref{t:higher1} enters crucially in the argument when estimating the second summand of \eqref{e:ro*1} for the regularization of $\xii \circ f$.

\subsection{Regularization by convolution}
Here we construct the competitor.

\begin{proposition}\label{p:conv}
Let $\beta_1\in \left(0, \frac{1}{2m}\right)$ and
$T$ be an area minimizing current satisfying Assumption \ref{ipotesi_base} 
in $\bC_4$. Let
$f$ be its $E^{\beta_1}$-Lipschitz approximation.
Then, there exist constants $\bar \eps_{12},\gamma_{12}, C>0$ and a subset of radii $B\subset [\sfrac{9}{8},2]$ with
$|B|>1/2$ with the following properties.
If $\bE(T,\bC_4) \leq \bar \eps_{12}$, for every $\sigma\in B$, there exists a $Q$-valued function $g\in \Lip(B_\sigma,\Iq)$ such that
\begin{gather*}
g\vert_{\de B_{\sigma}}=f\vert_{\de B_{\sigma}},\quad
\Lip(g)\leq C\,E^{\beta_1},\quad
\supp(g(x)) \subset \Sigma \quad \forall\, x \in B_\sigma,
\end{gather*}
and
\begin{equation}\label{e:conv}
\int_{B_{\sigma}}|Dg|^2\leq \int_{B_{\sigma}\cap K}|Df|^2 + C\, E^{\gamma_{12}}\left(E+\bA^2\right).
\end{equation}
\end{proposition}

\begin{proof} By Remark \ref{r:Psi} we assume that $\Psi (0) =0$, $\|D \Psi\|_0 \leq
C (E^{\sfrac{1}{2}} + \bA)$ and $\|D^2\Psi\|_0 \leq C \bA$. 
Since $|Df|^2\leq C\, \d_T \leq C E^{2\,\beta_1}\leq 1$ on $K$, by Theorem \ref{t:higher1} there is
$q_1 = 2\, p_1 > 2$ such that
\begin{equation}\label{e:uKtris}
\||Df|\|_{L^{q_1} (K\cap B_2)}^2\leq C\,E^{1-\sfrac{1}{p_1}} (E+\bA^2)^{\sfrac{1}{p_1}}
\leq C(E + \bA^2)\, .
\end{equation}
Given two (vector-valued) functions $h_1$ and $h_2$ and two radii $0<\bar r<r$, we denote by $\lin(h_1,h_2)$ the
linear interpolation in $B_{r}\setminus \bar B_{\bar r}$ between $h_1|_{\partial B_r}$ and $h_2|_{\partial B_{\bar r}}$.
More precisely, if $(\theta, t)\in {\mathbb S}^{m-1}\times [0, \infty)$ are spherical coordinates, then
$$
\lin (h_1, h_2) (\theta, t) = \frac{r-t}{r-\bar r}\, h_2 (\theta, t)
+ \frac{t-s}{r-\bar r}\, h_1 (\theta, t)\, .
$$
Next, let $\delta>0$ and $\eps>0$ be two parameters and let  $1<r_1<r_2<r_3< 2$ be three radii, all to be chosen later.
To keep the notation simple, we will write $\ro^\star$ in place of $\ro^\star_\delta$.
Let $\ph\in C^\infty_c(B_1)$ be a standard (nonnegative!) mollifier. We also use the notation
$f (x)=(f_1 (x),f_2 (x)) \in \Iq(\R^{\bar n}\times \R^l)$ meaning that $f (x)= \sum_i \a{(f^i_1 (x), f^i_2 (x))}$
with $(f^i_1 (x), f^i_2 (x))\in \R^{\bar{n}}\times\R^{l}$ 
and the maps $f_1$ and $f_2$ are then given by $f_j (x) = \sum_i \a{f^i_j (x)}$. This does not
create confusion in ``ordering the sheets'': since the points $f^i (x)$ belong to $\Sigma$
we have indeed the relation $f_2^j (x) = \Psi (x, f_1^j (x))$. We moreover set
$f' := \xii \circ f_1$. Recall the map $\ro$ of \cite[Theorem 2.1]{DS1} and
define:
\begin{equation}\label{e:v}
\g:=
\begin{cases}
\sqrt{E}\,\ro \circ \lin\left(\frac{\f}{\sqrt{E}},\ro^\star\left(\frac{\f}{\sqrt{E}}\right)\right)& \text{in }\; B_{r_3}\setminus B_{r_2},\\
\sqrt{E}\,\ro \circ \lin\left(\ro^\star\left(\frac{\f}{\sqrt{E}}\right),\ro^\star
\left(\frac{\f}{\sqrt{E}}*\ph_\eps\right)\right)
& \text{in }\; B_{r_2}\setminus B_{r_1},\\
\sqrt{E}\,\ro^\star\left(\frac{\f}{\sqrt{E}}*\ph_\eps\right)& \text{in }\; B_{r_1}.
\end{cases}
\end{equation}
Finally set $g_1:= \xii^{-1}\circ g'$ and $g:= \sum_i \a{(g_1^i,\Psi(x, g_1^i))}$.
We claim that, for $\sigma:= r_3$ in a suitable set $B\subset[\sfrac{9}{8},2]$ with $|B|>1/2$, we can
choose $r_2 =r_3-s$ and $r_1= r_2-s$ so that $g$ satisfies the conclusion of the proposition.
Some computations will be simplified taking into account that our choice
of the parameters will imply the following inequalities:
\begin{equation}\label{e:scelte}
\delta^{2\cdot8^{-\bar{n}Q}} \leq s\, ,\quad
 \eps \leq s \quad\mbox{and}\quad E^{1-2 \beta_1}\leq \eps^m\, .
\end{equation}
We start noticing that clearly $g\vert_{\de B_{r_3}}=f\vert_{\de B_{r_3}}$. As for the Lipschitz constant,
it suffices to estimate 
the Lipschitz constant of $\g$. This can be easily done observing that:
\begin{align*}
 \begin{cases}
\Lip(\g)\leq C\,\Lip(\f*\ph_{\eps})\leq C\,\Lip(\f)\leq C\,E^{{\beta_1}}&
\textrm{in }\; B_{r_1},\\
\Lip(\g)\leq C \,\Lip(\f)+C\,\frac{\norm{\f-\f*\ph_\eps}{L^\infty}}{s}
\leq C (1+\frac{\eps}{s})\,\Lip(\f)
\leq C\,E^{{\beta_1}}
&\textrm{in } \; B_{r_2}\setminus B_{r_1},\\
\Lip(\g)\leq C \,\Lip(\f)+C\,E^{1/2}\,
\frac{\delta^{8^{-\bar nQ}}}{s}\leq C\,E^{{\beta_1}}
+C\,E^{1/2}
\leq C\,E^{\beta_1}
&\textrm{in } \; B_{r_3}\setminus B_{r_2}\, .
\end{cases}
\end{align*}
In the first inequality of the last line we have used that, since $\cQ$ is a cone, 
$E^{-\sfrac{1}{2}} f'(x)\in \cQ$ for every $x$: therefore
$|\ro^\star (f'/E^{\sfrac{1}{2}}) - f'/ E^{\sfrac{1}{2}}|
\leq C \delta^{8^{-\bar{n}Q}}$.
We pass now to estimate the Dirichlet energy of $g$.

\medskip

\noindent\textit{Step 1. Energy in $B_{r_3}\setminus B_{r_2}$.}
By Section \ref{ss:xii},
the energy of the first component $g_1$ coincides with the
(classical!) Dirichlet energy of $g'$.
By Proposition~\ref{p:ro*}, $|\ro^\star(P)-P|\leq C\,\delta^{8^{-\bar n Q}}$ for all $P\in\cQ$.
Thus, elementary estimates on the linear interpolation give
\begin{align}\label{e:A}
\int_{B_{r_3}\setminus B_{r_2}}|D\g|^2\leq{}&\frac{C\,E}{\left(r_3-r_2\right)^2}\int_{B_{r_3}\setminus B_{r_2}}
\left|\textstyle{\frac{\f}{\sqrt{E}}}-\ro^\star\left(\textstyle{\frac{\f}{\sqrt{E}}}\right)\right|^2+
C\int_{B_{r_3}\setminus B_{r_2}}|D\f|^2\notag\\
& + C \int_{B_{r_3}\setminus B_{r_2}}|D(\ro^\star\circ \f)|^2
\leq C\int_{B_{r_3}\setminus B_{r_2}}|D\f|^2+ C\,E\,s^{-1}\, \delta^{2\cdot 8^{-\bar n Q}}\, .
\end{align}
As for $g_2$, we compute
$Dg_2^i(x) = D_x \Psi(x, g_1^i(x)) + D_u\Psi(x, g_1^i(x)) Dg_1^i(x)$ and so
\begin{equation}\label{e:A riemanniano}
\int_{B_{r_3}\setminus B_{r_2}}|Dg_2|^2 \leq C\,s\, (E+\bA^2)\,,
\end{equation}
where we used the estimate $\|Dg_2\|_0 \leq C\, \|D\Psi\|_0\leq C(E^{\sfrac{1}{2}}+\bA)$.

\noindent\textit{Step 2. Energy in $B_{r_2}\setminus B_{r_1}$.}
Here, using the same interpolation inequality and a standard estimate on convolutions of $W^{1,2}$
functions,
we get
\begin{align}\label{e:B}
&\int_{B_{r_2}\setminus B_{r_1}}|D\g|^2 \leq
C\int_{B_{r_2 + \eps}\setminus B_{r_1 - \eps}}|D\f|^2+
\frac{C}{(r_2-r_1)^2}\int_{B_{r_2}\setminus B_{r_1}}|\f-\ph_\eps*\f|^2\notag\\
\leq& C\int_{B_{r_2+ \eps}\setminus B_{r_1 - \eps}}|D\f|^2+
C\,\eps^2 s^{-2}\int_{B_{3}}|D\f|^2
\leq C\int_{B_{r_2 + \eps}\setminus B_{r_1- \eps}}|D\f|^2+C\,\eps^2\,E\, s^{-2}\, .
\end{align}

Similarly, for the second component we have that 
\begin{equation}\label{e:B riemanniano}
\int_{B_{r_2 + \eps}\setminus B_{r_1 - \eps}}|Dg_2|^2 \leq C\,(\bA^2+E)\, s.
\end{equation}

\noindent\textit{Step 3. Energy in $B_{r_1}$.}
Define
$Z:=\left\{ \dist\left(\frac{\f}{\sqrt{E}} * \ph_\eps,\cQ\right)>\delta^{\bar{n} Q+1}\right\}$ and
use \eqref{e:ro*1} to get
\begin{equation}\label{e:C}
\int_{B_{r_1}}|D\g|^2\leq\Big(1+C\,\delta^{8^{-\bar{n} Q-1}}\Big)\int_{B_{r_1}\setminus Z}\left|D\left(\f*\ph_\eps\right)\right|^2
+C\int_{Z}\left|D\left(\f*\ph_\eps\right)\right|^2=:I_1+I_2.
\end{equation}
We consider $I_1$ and $I_2$ separately. For $I_1$ we first observe the elementary inequality
\begin{align}
\|D (f'* \varphi_\eps)\|^2_{L^2}\leq & \||Df'|*\varphi_\eps\|_{L^2}^2
\leq \|(|D\f|\,{\bf 1}_K)*\ph_\eps\|_{L^2}^2+
\|(|D\f|\,{\bf 1}_{K^c})*\ph_\eps\|^2_{L^2}\notag\\
&\qquad\qquad\qquad\qquad+2\|(|D\f|\,{\bf 1}_K)*\ph_\eps\|_{L^2} 
\|(|D\f|\,{\bf 1}_{K^c})*\ph_\eps\|_{L^2}\, ,\label{e:I1}
\end{align}
where $K^c$ is the complement of $K$ in $B_3$.
Recalling $r_1+\eps \leq r_1+s = r_2$ we estimate the first summand in \eqref{e:I1} as follows:
\begin{equation}\label{e:I1 1}
\|(|D\f|\,{\bf 1}_{K})*\ph_\eps\|_{L^2 (B_{r_1})}^2 \leq
\int_{B_{r_1+\eps}}\left(|D\f|\,{\bf 1}_{ K}\right)^2
\leq\int_{B_{r_2}\cap K}|D\f|^2\, .
\end{equation}
To treat the other terms recall that $\Lip(\f)\leq C\,E^{{\beta_1}}$ and $|K^c|\leq C\,E^{1-2{\beta_1}}$:
\begin{equation}\label{e:I1 2}
\|(|D\f| {\bf 1}_{K^c})*\ph_\eps\|_{L^2 (B_{r_1})}^2\leq
CE^{2{\beta_1}}\|{\bf 1}_{K^c}*\ph_\eps\|_{L^2}^2
\leq
CE^{2\beta_1}\norm{{\bf 1}_{K^c}}{L^1}^2\norm{\ph_\eps}{L^2}^2
\leq \frac{CE^{2-2\beta_1}}{\eps^{m}}.
\end{equation}
Putting \eqref{e:I1 1} and \eqref{e:I1 2} in \eqref{e:I1} and recalling $E^{1-2\beta_1} \leq \eps^m$
and $\int |Df'|^2 \leq C E$, we get
\begin{align}\label{e:I1 3}
I_1 \leq{}& \int_{B_{r_2}\cap K}|D\f|^2+C\, \delta^{8^{-\bar{n}Q-1}}\,E+
C\, \eps^{-\sfrac{m}{2}}\, E^{\sfrac{3}{2}-{\beta_1}}\, .
\end{align}
For what concerns $I_2$, first we argue as for $I_1$, splitting in $K$ and $K^c$,
to deduce that
\begin{equation}\label{e:I2}
I_2\leq C\int_{Z}\left((|D\f|\,{\bf 1}_{K})*\ph_{\eps}\right)^2+
C\, \eps^{-\sfrac{m}{2}} E^{\sfrac{3}{2}-{\beta_1}}.
\end{equation}
Then, regarding the first summand in \eqref{e:I2},
we note that
\begin{equation}\label{e:Z}
|Z|\,\delta^{2\bar{n}Q+2}\leq\int_{B_{r_1}}\left|\textstyle{\frac{\f}{\sqrt{E}}}*\ph_{\eps}-
\textstyle{\frac{\f}{\sqrt{E}}}\right|^2\leq C\,\eps^2.
\end{equation}
Since $|Df'|\leq |Df|$ (and recalling that $q_1=2p_1>2$), we use \eqref{e:uKtris} to obtain 
\begin{align}\label{e:I2 2}
\int_{Z}\left((|D\f|\,{\bf 1}_{K})*\ph_{\eps}\right)^2&
\leq |Z|^{\frac{p_1-1}{p_1}}
\|(|D\f|\,{\bf 1}_{K})*\ph_{\eps}\|_{L^{q_1}}^2\leq C\, 
\left(\frac{\eps}{\delta^{\bar{n}Q+1}}\right)^{\frac{2\,(p_1-1)}{p_1}} \||Df'|\|_{L^{q_1} (K)}^2\notag\\
&\leq C\,\left(\frac{\eps}{\delta^{\bar{n}Q+1}}\right)^{\frac{2\,(p_1-1)}{p_1}}\,(E + \bA^{2})\,.
\end{align}
Gathering all the estimates together, 
\eqref{e:C}, \eqref{e:I1 3}, \eqref{e:I2} and
\eqref{e:I2 2} give
\begin{equation}\label{e:C2}
\int_{B_{r_1}}|D\g|^2\leq
\int_{B_{r_2}\cap K}|D\f|^2
+C\Big(E \delta^{8^{-\bar{n}Q-1}}+
\frac{E^{\sfrac{3}{2}-{\beta_1}}}{\eps^{\sfrac{m}{2}}}+
(E + \bA^2) \left(\frac{\eps}{\delta^{\bar{n}Q+1}}\right)^{\frac{2\,(p_1-1)}{p_1}}\Big).
\end{equation}
On the other hand, for what concerns $g_2$ we can estimate as follows
\begin{align}
&\int_{B_{r_1}} |D g_2|^2  =  \int_{B_{r_1}} |D f_2|^2 + \sum_{i}\int_{B_{r_1}} (D g_2^i-Df_2^i)\cdot
(D g_2^i+Df_2^i)\notag\allowdisplaybreaks\\
\leq & \int_{B_{r_1}\cap K} |D f_2|^2 + \int_{B_{r_1}\setminus K} |D f_2|^2
+ C\, \big(\bA+E^\frac12\big)\sum_{i}\int_{B_{r_1}} |D g_2^i -Df_2^i|\label{e:Riemanniano_intermedio}
\end{align}
We already observed that $|Df_2| \leq C (\bA + E^{\sfrac{1}{2}})$, leading to the estimate
$\int_{K^c} |Df_2|^2 \leq C (\bA^2 +E) |K^c| \leq C (\bA^2 +E) E^{1-2\beta_1}$. As for the latter
summand we compute
\begin{align*}
|D g_2^i -Df_2^i| \leq & |D_x\Psi(x, g_1^i) - D_x\Psi(x, f_1^i)|\notag\\ 
&+ |D_u\Psi(x,g_1^i(x))Dg_1^i| +
 |D_u\Psi(x,f_1^i(x))Df_1^i|\notag\\
\leq & C \bA\cG(g_1, f_1)
+ C\,\big(\bA+E^{\sfrac{1}{2}}\big)\,E^{\beta_1}\, .
\end{align*}
We next estimate $\|\cG (g_1, f_1)\|_\infty \leq C \g'- f'\|_\infty$ and
\begin{align*}
\|\g-\f\|_{\infty} &\leq
C\,\sqrt{E} \left(\left\|\ro^*\left(\textstyle{\frac{\f}{\sqrt{E}}} * \varphi_\eps\right) - \ro^*\left(\textstyle{\frac{\f}{\sqrt{E}}}\right) \right\|_{\infty} + 
\left\|\ro^*\left(\textstyle{\frac{\f}{\sqrt{E}}}\right)-\textstyle{\frac{\f}{\sqrt{E}}}\right\|_{\infty}\right)\notag\\
& \leq C\,\Lip(\ro^*) \, \| \f * \varphi_\eps - \f\|_{L^\infty} + C\,E^{\sfrac12} \delta^{8^{-\bar n Q}}
\leq C\, E^{\beta_1} \, \eps + C\,E^{\sfrac12} \delta^{8^{-\bar{n} Q}}\leq C E^{\beta_1}\, .
\end{align*}
We therefore conclude
\begin{equation}\label{e: C riemanniano}
\int_{B_{r_1}} |D g_2|^2 \leq \int_{B_{r_1}\cap K} |D f_2|^2 + C (\bA^2 +E ) E^{\beta_1}\, .
\end{equation}

\medskip

{\it Final estimate.} Since $|Dg|^2 = |Dg'|^2 + |Dg_2|^2$, summing \eqref{e:A}, \eqref{e:A riemanniano}, \eqref{e:B}, \eqref{e:B riemanniano}, \eqref{e:C2}  and \eqref{e: C riemanniano} (and recalling $\eps<s$), we conclude
\begin{align*}
\int_{B_{r_3}}|Dg|^2 \leq{}&
\int_{B_{r_1}\cap K}|Df|^2+
C\int_{B_{r_1+3s}\setminus B_{r_1-s}}|D\f|^2
+ C (\bA +E^2) (s+ E^{\beta_1})
\\
&+C\,E\, \bigg(
\delta^{8^{-\bar n Q -1}}+
\frac{\eps^2}{s^2}+
\frac{\delta^{2\cdot 8^{-\bar{n}Q}}}{s}+
\frac{E^{\sfrac{1}{2}-{\beta_1}}}{\eps^{m/2}}+
\left(1
+ \bA^2\,E^{-1}\right)\left(\frac{\eps}{\delta^{\bar{n}Q+1}}\right)^{\frac{2\,(p_1-1)}{p_1}} \bigg)\, .
\end{align*}
We set $\eps=E^a$, $\delta=E^b$ and $s=E^c$, where 
\begin{equation*}
a=\frac{1-2\,{\beta_1}}{2\,m},\quad b=\frac{1-2\,{\beta_1}}{4\,m\,(\bar{n}\,Q+1)}
\quad\textrm{and}\quad c=\frac{1-2\,{\beta_1}}{8^{\bar{n}Q}\, 4m\,(\bar{n}\,Q+1)}.
\end{equation*}
This choice respects \eqref{e:scelte}. Assume $E$ is small enough so that $s\leq \frac{1}{16}$.
Now, if $C>0$ is a sufficiently large constant,
there is a set $B'\subset[\sfrac{9}{8},\frac{29}{16}]$ with $|B'|> 1/2$ such that,
\begin{equation*}
\int_{B_{r_1+3s}\setminus B_{r_1-s}}|D\f|^2\leq C\,s
\int_{B_2}|D\f|^2\leq C\,E^{1+c}
\quad \mbox{for every $r_1\in B'$.}
\end{equation*} 
Indeed by integrating in polar coordinates and by Fubini's Theorem we have that
\begin{align*}
\int_{\frac98}^{\frac{29}{16}} dr \int_{B_{r+3s}\setminus B_{r-s}} |D\f|^2 &=
\int_{\frac98}^{\frac{29}{16}} dr \int^{r+3s}_{r-s}dt \int_{\de B_t} |D\f|^2 d\cH^{n-1} \\
&\leq 4\, s \int_{\frac98-s}^{2} dt \int_{\de B_t} |D\f|^2 d\cH^{n-1} \leq 4 \, s\int_{B_2}|D\f|^2,
\end{align*}
from which the conclusion follows for $C$ big enough:
\begin{align*}
&\left\vert\left\{ r \in \left[\frac98,\frac{29}{16}\right] : \int_{B_{r+3s}\setminus B_{r-s}} |D\f|^2 \geq C\,s
\int_{B_2}|D\f|^2\right\} \right\vert\\
&\quad\leq \frac{1}{C\,s\int_{B_2}|D\f|^2} \int_{\frac98}^{\frac{29}{16}} dr \int_{B_{r+3s}\setminus B_{r-s}} |D\f|^2
\leq \frac{4}{C} < \frac18.
\end{align*}
For $\sigma = r_3\in B = 2s + B'$ we then conclude, for some $\gamma (\beta_1, \bar{n}, n,m, Q)>0$,
\begin{align*}
\int_{B_\sigma}|Dg|^2 \leq{}&
\int_{B_\sigma\cap K}|Df|^2 + C E^\gamma (E + \bA^2)\, .\qquad\qquad\qquad\qedhere
\end{align*}
\end{proof}

\subsection{Proof of Theorem \ref{t:higher}}
Choose $\beta_1= \frac{1}{4m}$ and consider the set $B\subset[\sfrac98,2]$ given in Proposition \ref{p:conv}.
Using the coarea formula and the isoperimetric inequality
(the argument and the map $\ph$ are the same in the proof of Theorem~\ref{t:o(E)} and that of
Proposition~\ref{p:o(E)}),
we find $s\in B$ and an integer rectifiable current $R$ such that
\[
\de R=\la T-\mathbf{G}_f,\ph, s\ra \quad\text{and}\quad
\mass(R)\leq CE^{\frac{2m-1}{2m-2}}.
\]
Since $g\vert_{\de B_s}=f\vert_{\de B_s}$ and $g$ takes values in $\Sigma$, 
we can use $g$ in place of $f$ in the estimates and, arguing as before (see e.g. the proof of Theorem \ref{p:o(E)}), we get, for a suitable $\gamma>0$:
\begin{equation}\label{e:massa1}
\|T\| ( \bC_s)\leq
Q\,|B_s|+\int_{B_s}\frac{|Dg|^2}{2}+CE^{1+\gamma}
\;\stackrel{\mathclap{\eqref{e:conv}}}
{\leq}\;Q\,|B_s|+
\int_{B_s\cap K}\frac{|Df|^2}{2}+CE^\gamma (E
+ \bA^2)\, .
\end{equation}
On the other hand, by Taylor's expansion in Remark~\ref{r:taylor},
\begin{align}\label{e:massa2}
\| T\| (\bC_s)&=\|T\| ((B_s\setminus K)\times \R^{n})+
\|\mathbf{G}_f\| ((B_s\cap K)\times \R^n)\notag\\
&\geq \|T\| ((B_s\setminus K)\times \R^{n})+Q\, |K\cap B_s|+
\int_{K\cap B_s}\frac{|Df|^2}{2}-C\, E^{1+\gamma}.
\end{align}
Hence, from \eqref{e:massa1} and \eqref{e:massa2}, we get $\e_T(B_s\setminus K)\leq
C\, E^{\gamma}\,(E+\bA^2)$.

This is enough to conclude the proof. Indeed, let $A\subset B_{\sfrac98}$ be a Borel set.
Using the higher integrability of $|Df|$ in $K$ (see \eqref{e:uKtris}) and possibly selecting a smaller $\gamma>0$, we get
\begin{align*}
\e_T(A)&\leq\e_T(A\cap K)+\e_T( A\setminus K)\leq
\int_{A\cap K}\frac{|Df|^2}{2}+C\,E^{\gamma}\left(E + \bA^2\right)\notag\\
&\leq C\,|A\cap K|^{\frac{p_1-1}{p_1}}\left(\int_{A\cap K}|Df|^{q_1}\right)^{\sfrac{2}{q_1}}
+C\,E^{\gamma}\left(E + \bA^2\right)\\
&\leq 
C\, |A|^{\frac{p_1-1}{p_1}}\, \left(E + \bA^2\right)
+ C\, E^{\gamma}\left(E + \bA^2\right).
\end{align*}

\subsection{Proofs of Theorems \ref{t:main} and \ref{t:harmonic_final}} 
As usual we assume, w.l.o.g., $r=1$ and $x=0$. 
Choose ${\beta_{11}}<\min \{\frac{1}{2m},\frac{\gamma_{11}}{2(1+\gamma_{11}
)}\}$,
where $\gamma_{11}$ is the constant in Theorem~\ref{t:higher}.
Let $f$ be the $E^{\beta_{11}}$-Lipschitz approximation of $T$.
Clearly \eqref{e:main(i)} follows directly from Proposition~\ref{p:max} if $\gamma_1<{\beta_{11}}$.
Set next $A:= \left\{\bmax\be_T> 2^{-m}E^{2\,{\beta_{11}}}\right\}\cap B_{\sfrac98}$.
By Proposition \ref{p:max}, $|A|\leq C E^{1-2{\beta_{11}}}$.
If $\eps_{1}$ is sufficiently small, apply \eqref{e:max1} and estimate \eqref{e:higher2} to $A$
to conclude:
\begin{equation*}
|B_1\setminus K|\leq C\, E^{-2\,{\beta_{11}}}\,\e_T\left(A\right)
\leq C\, E^{\gamma_{11} - 2\beta_{11}(1+\gamma_{11})} (E+\bA^2).
\end{equation*}
By our choice of $\gamma_{11}$ and ${\beta_{11}}$, this gives \eqref{e:main(ii)}
for some positive $\gamma_1$. Finally, set $S=\mathbf{G}_f$.
Recalling the strong Almgren estimate \eqref{e:higher2} and the 
Taylor expansion in Remark~\ref{r:taylor}, we conclude: for every $0<\sigma\leq1$
\begin{gather*}
\left| \|T\| (\bC_{\sigma}) - Q \,\sigma^m\,\omega_m -\int_{B_{\sigma}} \frac{|Df|^2}{2}\right|
\leq \e_T(B_{\sigma}\setminus K)+\e_S (B_{\sigma}\setminus K)+\left|\e_S (B_{\sigma})-\int_{B_{\sigma}}\frac{|Df|^2}{2}\right|\notag\\
\leq  C\, E^{\gamma_{11}}(E+\bA^2)+C\, |B_{\sigma}\setminus K|+C\,\Lip(f)^2\int_{B_{\sigma}}|Df|^2
\leq C\,E^{\gamma_{1}} (E+\bA^2).
\end{gather*}
The $L^\infty$ bound follows from Proposition \ref{p:max} recalling that, by Remark \ref{r:Psi}, 
we can assume $\|D \Psi\|_0 \leq C (E^{\sfrac12}+ \bA)$. Finally,
Theorem~\ref{t:harmonic_final} is a special case of Theorem~\ref{t:o(E)},
since the map $f$ in Theorem~\ref{t:main} is the $E^{\gamma_1}$-Lipschitz approximation of $T$.

\section{The ``almost'' projections $\ro^\star_\delta$}\label{s:ro*}
In this section we show the existence of the maps $\ro^\star_\delta$
in Proposition~\ref{p:ro*}.
Compared to the original ones introduced by Almgren, our $\ro^\star_\delta$'s
have the advantage of depending on a single parameter. Our proof
is different from Almgren's and gives more explicit estimates, relying heavily on the following
simple corollary of
Kirszbraun's Theorem. 

\begin{lemma}\label{l:kirsz}
Let $f: \Omega \subset \R^{N_1} \to C\subset \R^{N_2}$ be a Lipschitz function and
assume that $C$ is closed and convex.
Then, there is an extension $\hat{f}$ of $f$ to the whole $\R^{N_1}$
which preserves the Lipschitz constant and takes values in $C$.
\end{lemma}

To prove Lemma \ref{l:kirsz} it suffices to take the map $\tilde{f}$ of the classical statement of Kirszbraun's theorem (see
\cite[Theorem 2.10.43]{Fed}) which takes values in $\R^{N_2}$ and compose it with the orthogonal projection $\pi_C$ onto 
the convex closed set $C$,
which is a $1$-Lipschitz map in $\R^{N_2}$.    

\begin{proof}[Proof of Proposition~\ref{p:ro*}]
The proof consists of four parts: the first one is a detailed description of the
set $\cQ$, whereas the remaining three give a rather explicit construction in this
order:
\begin{enumerate}
\item first we specify $\ro^\star_\delta$ on $\cQ$: the resulting map will be called $\ro^\flat$;
\item then we extend it to a map $\ro^\sharp$ on $\cQ_{\delta^{nQ+1}}$, the $\delta^{nQ+1}$-neighborhood
of $\cQ$; $\ro^\sharp$ will satisfy $\Lip (\ro^\sharp) \leq 1+ C \delta^{8^{-\bar{n}Q-1}}$ and
$|\ro^\sharp (P) - P|\leq C \delta^{8^{-\bar{n}Q}}$ for every $p\in \cQ$;
\item  we then extend it to all $\R^{N}$
keeping its Lipschitz constant bounded.
\end{enumerate}
(3) follows easily from (2): we consider $\xii^{-1}\circ\ro^\sharp:\cQ_{\delta^{\bar{n}Q+1}}\to\Iq$ and
a Lipschitz extension $h:\R^{N}\to \Iq$ of it with $\Lip(h)\leq C$, using \cite[Theorem 1.7]{DS1}.
Our map is then $\ro^\star_\delta :=\xii\circ h$.
Then \eqref{e:ro*1} is an easy consequence of (2), (3) and the chain rule.

The description of $\cQ$ and the proofs of (1) and (2) are given in the next subsections.
\end{proof}

From now on we use $n$ instead of $\bar{n}$
to simplify the notation.

\subsection{Conical simplicial structure of $\cQ$}
We first prove that $\cQ$ is the union of families $\{\cF_i\}_{i=0}^{nQ}$ of sets, 
the ``$i$-dimensional faces'' of $\cQ$, with the following properties:
\begin{itemize}
\item[(p$1$)] $\cQ=\cup_i
\cup_{F\in\cF_i} F$;
\item[(p$2$)] $\cF:=\cup_i \cF_i$ is a collection of finitely many disjoint sets;
\item[(p$3$)] each face $F\in\cF_i$ is a convex {\em open} $i$-dimensional cone,
where open means that for every $x\in F$ there exists an $i$-dimensional disk $D$ with
$x\in D\subset F$;
\item[(p$4$)] for each $F\in \cF_i$, $\bar F\setminus F$ is the union of some
elements of $\cup_{j<i} \cF_j$.
\item[(p$5$)] for each $i<k \leq nQ$ and for each $F\in \cF_i$, there exists
$G \in \cF_k$ such that $ F\subset \bar G$.
\end{itemize}

\begin{remark}\label{r:piccolo_abuso}
With a small abuse of notation $\partial F$ will denote $\overline{F}\setminus F$ for any $F\in \mathcal{F}$.
\end{remark}

So, $\cF_0=\{0\}$; $\cF_1$ consists of finitely many half-lines meeting at $0$, i.e. of sets of type
$l_v=\{\lambda \,v\,:\:\lambda\in]0,+\infty[\}$ for $v\in \s^{N-1}$;
$\cF_2$ consists of finitely many $2$-dimensional ``infinite triangles'' delimited by
pairs of half lines $l_{v_1}$, $l_{v_2}\in\cF_1$ and by $\{0\}$;
and so on.
To prove this statement, first of all we recall the construction of
$\xii$ (see \cite[2.1.2]{DS1}).
After selecting a suitable finite collection of non zero vectors
$\{e_k\}_{k=1}^{h}$ (in general $h > n$),
we define the linear map $L:\R^{n\,Q}\to \R^{N}$ with $N := h\, Q>n\,Q$ given by
\[
L(P_1,\ldots, P_Q):=\big(\underbrace{P_1\cdot e_1,\ldots,P_Q\cdot e_1}_{w^1},
\underbrace{P_1\cdot e_2,\ldots,P_Q\cdot e_2}_{w^2},
\ldots,\underbrace{P_1\cdot e_h,\ldots,P_Q\cdot e_h}_{w^h}\big)\,.
\]
Then, we consider the map $O:\R^{N}\to\R^{N}$ which maps $(w^1\ldots,w^h)$
into the vector $(v^1,\ldots,v^h)$
where each $v^i$ is obtained from $w^i$ ordering its components in
increasing order.
Note that the composition $O\circ L:\left(\R^{n}\right)^Q\to\R^{N}$ is now
invariant under the action of the symmetric group $\Pe_Q$.
Therefore, $\xii$ is simply the induced map on $\Iq=(\R^{n})^Q/\Pe_Q$
and $\cQ=\xii(\Iq)=O(V)$ where $V:=L(\R^{n\,Q})$. Moreover, since the
vectors $e_i$'s span $\R^n$ (cf. \cite[2.1.2]{DS1}), the map $L$ is injective 
and thus $V$ is an $nQ$-dimensional subspace.

Consider the following equivalence relation $\sim$ on $V$:
\begin{equation}\label{e:equi}
(w^1,\ldots,w^h)\sim (z^1,\ldots,z^h)\qquad \textrm{if}\qquad
\begin{cases}
w^i_j=w^i_k &\Leftrightarrow z^i_j=z^i_k\\
w^i_j>w^i_k &\Leftrightarrow z^i_j>z^i_k
\end{cases}
\quad\forall\;i,j,k\,,
\end{equation}
where $w^i=(w^i_1,\ldots,w^i_Q)$ and $z^i=(z^i_1,\ldots,z^i_Q)$: if $w\sim z$, then
$O$ rearranges their components with the same permutation.
We let $\cE$ denote the set of corresponding equivalence classes in $V$
and $\cC:=\{L^{-1}(E)\,:\,E\in\cE\}$.
The following fact is an obvious consequence of definition \eqref{e:equi}:
\begin{equation*}
L(P)\sim L(S)\quad\text{if and only if}\quad
L(P_{\pi(1)},\ldots, P_{\pi(Q)})\sim L(S_{\pi(1)},\ldots, S_{\pi(Q)})\quad\forall\;\pi\in\Pe_Q\,.
\end{equation*}
Thus, $\pi (C)\in \mathcal{C}$
for every $C\in\cC$ and every $\pi\in\Pe_Q$. 
Since $\xii$ is injective and is induced by $O\circ L$, it follows that, for every pair $E_1,\,E_2\in\cE$,
either $O(E_1)=O(E_2)$ or $O(E_1)\cap O(E_2)=\emptyset$. Therefore,
the family $\cF:=\{O(E)\,:\,E\in\cE\}$ is a partition of $\cQ$.

Clearly, each $E\in\cE$ is a convex cone.
Let $i$ be its dimension and $D$ any $i$-dimensional disk
$D\subset E$. Denote by $x$ the center of $D$ and let $y$ be any other point of $E$.
Then, by \eqref{e:equi}, the point $z= y - \eps (x-y) = (1+\eps)\,y-\eps \, x$ belongs as well to $E$ for any $\eps>0$ sufficiently small. 
The convex envelope of $D\cup\{z\}$, which is contained in $E$, contains in turn an $i$-dimensional disk
centered in $y$: therefore $E$ is an open convex cone.
Since $O\vert_E$ is a linear injective map, $F=O(E)$ is an open convex cone of dimension $i$.
Therefore, $\cF$ satisfies (p$1$)-(p$3$).

Next notice that, having fixed $w\in E$, a point $z$ belongs to $\bar E\setminus E$ if and only if
\begin{enumerate}
\item $w^i_j\geq w^i_k$ implies $z^i_j\geq z^i_k$ for every $i,\,j$ and $k$; 
\item there exists $r,\,s$ and $t$ such that $w^r_s> w^r_{t}$
and $z^r_s= z^r_t$.
\end{enumerate}
Thus, if $d$ is the dimension of $E$, $\partial E := \overline{E}\setminus E$ (cf. Remark \ref{r:piccolo_abuso})
is the union of some elements of $\cup_{j<d} \cE_j$, where with $\cE_j$ we denote the $j$-dimensional elements of $\cE$. Observe that, since $O$ is continuous, we must have $\overline{F} \supset O (\overline{E})$.
On the other hand, if $x\in \overline{F}$ and $x_k\to x$ is a sequence contained in $F$, then
there is a sequence $\{y_k\}\subset E$ with $O (y_k) = x_k$. By the definition of $O$ the sequence $\{y_k\}$
is bounded and hence, up to subsequence, we can assume that it converges to $y\in \overline{F}$: thus $O(y) =x$
and $O (\overline{E}) = \overline{F}$. On the other hand, for equivalence classes $E_1, E_2$ of different dimension
we necessarily have $O (E_1)\cap O (E_2) = \emptyset$. Thus $O (\partial E) \cap O (E) =\emptyset$, i.e.
$\partial F= O (\partial E)$, which shows (p$_4$).

For what concerns (p$_5$) we show first that if $L(P)=z \in E \in \mathcal{E}$ is such that
$z^i_j \neq z^i_k$ for all $i$ and for all $j\neq k$, then $O(E) \in \cF_{nQ}$.
Indeed, if $t < \sfrac14 \min_{i,j\neq k} |z^i_j - z^i_k|$, then 
$L(P+v) \in E$ for every $v \in B_t(0) \subset \R^{nQ}$, i.e.~$E$ is an $(nQ)$-dimensional
convex cone.
Therefore it follows that for every $F \in \cF_i$ with $i < nQ$
there exists $G \in \cF_{nQ}$ such that $F \subset \bar G$.
To show this claim it is enough to prove that, if $F = O(E)$ and $L(P)=z \in E$, then
$z$ is the limit of points $w \in V$ such that $w^i_j \neq w^i_k$
for all $i,j,k$, which can be easily proved by a simple perturbation argument.
Next, we argue inductively on $k$: knowing that $F\in \mathcal{F}_i$ is contained in $\overline{G}$ for some
$G\in \mathcal{F}_k$ with $k>i+1$, we show that there is $H\in \mathcal{F}_{k-1}$ such that
$F\subset \overline{H}$. Observe indeed that $F\subset \partial G = \overline{G}\setminus G$ and that,
for dimensional reasons, $\overline{G}\setminus G$ must be contained in the closure of those $H\in \mathcal{F}_{k-1}$ such that $H\subset \overline{G}$. Let $H\in \mathcal{F}_{k-1}$ be such that $F\cap \overline{H}\neq \emptyset$. Consider $E, K\in \mathcal{E}$ such that $F = O(E)$ and $H = O (K)$. Let $x\in E$ such that $O(x)\in F\cap \overline{H}$ and $z\in K$. We then must have that $x^i_k \geq x^i_j$ whenever $z^i_k> z^i_j$ and that $x^i_k = x^i_j$ whenever $z^i_k = z^i_j$. By the very definition of $\sim$, the same property holds even if we replace $x$ with another element $\xi\in E$. Therefore the open segment $]\xi, z[$ must be contained in $K$, which in turn implies that $\xi\in \overline{K}$. Thus we conclude $F\subset \overline{H}$.

\subsection{Construction of $\ro^\flat$}
The main building block in the construction of $\ro^\flat$ 
is given by the following lemma.

\begin{lemma}\label{l:radial}
For $\tau\in ]0,\frac{1}{4}[$ and any $D\in \N\setminus \{0\}$ consider the map $\Phi_\tau : \R^D\to \R^D$ defined by:
\[
\Phi_\tau (x) = \left\{
\begin{array}{ll}
0 &\mbox{if $|x|\leq \tau$}\\
\sqrt{\tau}\frac{|x|- \tau}{\sqrt{\tau}-\tau} \frac{x}{|x|}\quad &\mbox{if $\tau \leq |x|\leq \sqrt{\tau}$}\\
x &\mbox{if $|x|\geq \sqrt{\tau}$}. 
\end{array}\right. 
\]
Then $|\Phi_\tau (x)-x|\leq \tau$ and $\Lip (\Phi_\tau)\leq 1+ 2 \sqrt{\tau}$.
\end{lemma}
\begin{proof}
The proofs of the two claims are straightforward computations. First $\Phi_\tau (x) =x$ if $|x|\geq \sqrt{\tau}$
and $|\Phi_\tau (x) -x|=|x|\leq \tau$ if $|x|\leq \tau$. For $\tau\leq |x|\leq \sqrt{\tau}$ we compute
\begin{align*}
|\Phi_\tau (x)-x| &=
 \left|\textstyle{\frac{\sqrt{\tau}\,(|x|-\tau)}{\sqrt{\tau} - \tau}} -|x|\right| = \tau \textstyle{\frac{\sqrt{\tau} - |x|}{\sqrt{\tau} - \tau}}\leq \tau.
\end{align*}
Next we show that $|D\Phi_\tau (x) \cdot v| \leq (1+ 2 \sqrt{\tau}) |v|$ at
any point of differentiability. This inequality
obviously imply the claimed Lipschitz constant estimate
because $\Phi_\tau$ is Lipschitz and its domain of definition is a convex set.
The inequality is, moreover, obvious when $|x|< \tau$ and $|x|> \sqrt{\tau}$.
For $\tau < |x|< \sqrt{\tau}$, we can compute
\[
D\Phi_\tau (x) = \frac{1-\frac{\tau}{|x|}}{1-\sqrt{\tau}} {\rm Id} + \frac{\frac{\tau}{|x|}}{1-\sqrt{\tau}} \frac{x}{|x|}\otimes
\frac{x}{|x|}\, .
\]
The matrix is symmetric with positive eigenvalues (because $|x|> \tau$) and the maximal eigenvalue is $(1-\sqrt{\tau})^{-1} \leq 1 + 2\sqrt{\tau}$,
thereby proving our claim.
\end{proof}

\subsubsection{Special coordinates, conical sections and separation} 
Let $S_k$ be the $k$-dimensional skeleton of $\cQ$, i.e. the union of
$F\in \mathcal{F}_k$ and denote by $(S_k)_\sigma$ its $\sigma$-neighborhood $\{x: \dist (x, S_k) < \sigma\}$.
Incidentally, $(S_k)_\sigma$ contains $(S_i)_\sigma$ for every $i<k$. 

\begin{definition}[Coordinates and conical sections]\label{d:coordinates}
Fix any face $F\in \mathcal{F}_k$ and introduce Cartesian coordinates 
$(y,z) \in \R^k\times \R^{N-k}$ in such a way that $F \subset \R^k\times \{0\}$.
For a positive constant $\tilde{c}$ consider the cone $\mathscr{C} (F) := \{(y,z)\in \cQ: (y,0)\in F\, , |z|\leq \tilde{c}\, \dist ((y,0), S_{k-1})\}$.
For any $p= (y, 0)\in F$ we set $V_p:= (\{y\} \times \R^{N-k})\cap \mathscr{C} (F)$.
\end{definition}

Note that, if $\tilde{c}$ is sufficiently small, we will have the following property
\begin{equation*}
\mathscr{C} (F) \cap \mathscr{C} (G) \neq \emptyset \quad\Longrightarrow\quad \mbox{either $F\subset \overline{G}$ or $G\subset \overline{F}$.}
\end{equation*}
For every constants $a,b>0$, $k=1\ldots,nQ-1$ and $F\in\cF_k$,
we fix coordinates as in Definition \ref{d:coordinates} and denote by $F_{a, b}$ the sets
\[
F_{a, b}:=\big\{(y,z)\,: |z|\leq a, (y, 0)\in F\setminus (S_{k-1})_b\}.
\]
For the faces $F\in\cF_{nQ}$ of maximal dimension and for every $a>0$, $F_{\star,a}$ denotes the set
$F_{\star,a}:=F \setminus (S_{nQ-1})_a$.
The following lemma is an obvious corollary of the linear simplicial and conical structures of $\cQ$.

\begin{lemma}\label{l:separation}
There is a constant $\bar{c}>0$ (independent of $a,b$ below)
with the following property. Assume $F$ and $G$ are two distinct $k$ dimensional faces.
\begin{itemize}
\item If $k=nQ$, $a>0$, $x\in F_{\star,a}$ and $x'\in G_{\star,a}$, then $|x-x'|\geq \bar{c} a$;
\item If $k<nQ$, $b/a > \bar{c}^{-1}$, $x\in F_{a,b}$ and $x'\in G_{a,b}$, then $|x-x'|\geq \bar{c} b$.  
\end{itemize}
Moreover, if $F\in \mathcal{F}_k$, $H\in \mathcal{F}_i$ with $i>k$ and $F\not\subset \partial H$ (cf. Remark
\ref{r:piccolo_abuso}), then $|x-x'|\geq \bar{c} a$ for every $x\in H$ and $x'\in F\setminus (S_{k-1})_a$.
\end{lemma}

\subsubsection{The domains ${\rm Dom} (f_k)$}
Next we choose constants $c_k := \delta^{8^{-nQ+k}}$. If $\delta$ is small enough, each family 
$\{F_{2\sqrt{c_k},c_{k-1}^2}\}_{F\in\cF_k}$ with $k<nQ$ is made by pairwise disjoint sets, which are at least
$\bar{c} c_{k-1}^2$ far apart, where $\bar{c}$ is the constant of Lemma \ref{l:separation},
and it holds $F_{2\sqrt{c_k},c_{k-1}^2} \subset \mathscr{C}(F) \subset \cQ$.
We are ready to define the map $\ro^\flat := \ro^\star|_{\cQ}$ inductively
``from the top to the bottom''. More precisely we will define a family of maps $\{f_k\}_{k\in \{0, \ldots, nQ\}}$ on domains
${\rm Dom} (f_k) \subset \cQ$ starting from $f_{nQ}$ and ending with $f_0 = \ro^\flat$. We first explicitly define
${\rm Dom} (f_k) :=  \cQ \setminus (S_{k-1})_{c_{k-1}}$ for $k>0$ and ${\rm Dom} (f_0) = \cQ$, and in order to simplify our notation
we then agree that $c_{-1} = \delta^{8^{-nQ-1}}$ and $S_{-1} = (S_{-1})_{c_{-1}} = \emptyset$.
Note that ${\rm Dom} (f_{k+1}) \not\subset {\rm Dom} (f_k)$.
It is obvious that
\begin{equation}\label{e:dominio_induttivo}
{\rm Dom} (f_k) = \Big( {\rm Dom} (f_{k+1}) \cup \bigcup_{F\in \mathcal{F}_k} F_{2\sqrt{c_k}, c_{k-1}^2}\Big) \setminus (S_{k-1})_{c_{k-1}}\, .
\end{equation}
Indeed, if $x\in {\rm Dom} (f_k)\setminus {\rm Dom} (f_{k+1})$ we then must have $\dist (x, S_k) < c_k$ and $\dist (x, S_{k-1})\geq c_{k-1}$. Let $q\in S_k$ be such that $|x-q|< c_k$. Since $\dist (x, S_{k-1}) \geq c_{k-1} > c_k$,
the point $q$ must necessarily belong to a $k$-dimensional face $F$.
Fix coordinates as in Definition \ref{d:coordinates}.
If $x=(y,z)$, we then obviously have $|z|<c_k \leq 2 \sqrt{c_k}$. On the other hand $\dist ((y,0), S_{k-1}) \geq \dist (x, S_{k-1})
- |z| \geq c_{k-1} - c_k > c_{k-1}^2$. This shows that $x\in F_{2\sqrt{c_k}, c_{k-1}^2}$. 

\subsubsection{The maps $f_k$} On ${\rm Dom} (f_{nQ})$ we define $f_{nQ}=\Id$ and specify next the
procedure to define $f_k$ knowing $f_{k+1}$. Along the procedure we claim inductively
the following.

\begin{ipotesi}[Inductive step]\label{ipotesi_induttiva_mappe_bastarde}
The map $f_{k+1}$ has the following three properties.
\begin{itemize}
\item[(a$_{k+1}$)] $\Lip(f_{k+1})\leq 1+C c_{k+1}^{\sfrac{1}{2}}$ and $|f_{k+1}(x)-x|\leq C\, c_{k+1}$.
\item[(b$_{k+1}$)] Consider  $i\leq k+1$, an $i$-dimensional face $F$, the cone $\mathscr{C}(F)$ in Definition~\ref{d:coordinates} and the corresponding coordinates. 
Then, $f_{k+1}$ factorizes on ${\rm Dom} (f_{k+1}) \cap \mathscr{C} (F)$ as
\begin{equation}\label{e:ipotesi_bast_b}
f_{k+1}(y,z)=(y,h^F_{k+1}(y, z))\in\R^{i}\times\R^{N-i}\, .
\end{equation}
\item[(c$_{k+1}$)] For every $G\in \mathcal{F}_i$ with $i\geq k+1$, $f_{k+1}$ maps ${\rm Dom} (f_{k+1})\cap \{x: \dist (x,G)<\delta\}$
into $\overline{G}$. Moreover the restriction of $f_{k+1}$ to $G_{c_i, c_k}$ is the orthogonal projection onto $G$. 
\end{itemize}
\end{ipotesi}

The constants involved depend on $k$ but not on the parameter
$\delta$ and since the process is iterated finitely many times, we will not keep track of such dependence.
Note that $f_{nQ}$ satisfies (a$_{nQ}$), (b$_{nQ}$) and (c$_{nQ}$) trivially, 
because it is the identity map. Given $f_{k+1}$ we next show
how to construct $f_k$.
For every $p\in G\in \cF_{k}$ with $p\notin (S_{k-1})_{c_{k-1}^2}$, 
set coordinates as in Definition \ref{d:coordinates} and consider the cone $W_p := \{(y,z)\in V_p: |z|\leq 2\sqrt{c_k}\}$.
Let now $\Phi_\tau$ be the map of Lemma \ref{l:radial} with $\tau = 2 c_k$.
The function $f_k$ is defined in $W_p$  by
\begin{equation}\label{e:fk}
f_{k}(x)= f_{k} (y,z) := (y, h^F_k (y, z)) :=
\begin{cases}
\left(y,0\right)
& \textrm{for }\; |z|\leq \tau/2=c_k,\\
\left(y, \Phi_\tau (h^F_{k+1} (y, z))\right) &\textrm{otherwise.}
\end{cases}
\end{equation}
If $q\in {\rm Dom} (f_k)$ does not belong to any $W_p$ as above, then
we set $f_{k+1} (q) = f_k (q)$.

Observe that the definition above gives values to $f_k$ on a set which is larger than ${\rm Dom} (f_k)$:
this will be useful to carry on some of the estimates, but we insist that Assumption \ref{ipotesi_induttiva_mappe_bastarde} will only be checked on ${\rm Dom} (f_k)$.

\subsubsection{Well-definition and continuity} Consider a point $q\in {\rm Dom} (f_{k})$. If $q$ is not contained in $F_{2\sqrt{c_k}, c_{k-1}^2}$
for some $k$-dimensional face, then by \eqref{e:dominio_induttivo} it is contained in the domain of $f_{k+1}$ and thus $f_k (q)$ is
defined. If $q$ is contained in $F_{2\sqrt{c_k}, c_{k-1}^2}$ for some $k$-dimensional face, then $q$ belongs to some $W_p$ as above.
Let $q=(y,z)$. If $|z|\leq c_k$, then $f_k (q)$ is defined;
otherwise, since $\dist (q, S_k) \geq c_k$, we infer that $q\in
{\rm Dom} (f_{k+1})$ and $f_k(q)$ is also defined.

As for the continuity, fix $(y,z)\in W_p\cap {\rm Dom} (f_k)$ with $p=(y,0)\in F\in \mathcal{F}_k$.
If $|z|= c_k$, then by (a$_{k+1}$) we have $|h^F_{k+1} (y,z)|\leq |z| + C c_{k+1}
\leq \tau/2 + C \tau^8$. For $\delta$ sufficiently small this obviously implies $|h^F_{k+1} (y,z)|\leq \tau$ and thus, by
the definition of $\Phi_\tau$, $\Phi_\tau (h^F_{k+1} (y,z))=0$. On the other hand, if $|z|= 2\sqrt{c_k}$, then 
$|h^F_{k+1} (y,z)|\geq |z|- C c_{k+1} = 2\sqrt{c_k} - C c_k^8 \geq \sqrt{2c_k}$ and thus $\Phi_\tau (h^F_{k+1} (y,z)) = h^F_{k+1} (y,z)$.
Therefore under this assumption we have $f_{k+1} (q)= f_k (q)$.

We next check that $f_k$ maps ${\rm Dom} (f_k)$ into $\cQ$. This is true by induction where $f_k$ coincides
with $f_{k+1}$. Fix therefore a point $q$ in some $W_p\cap {\rm Dom} (f_k)$ with $p\in F\in \mathcal{F}_k$
and let $G$ be the $i$-dimensional face containing $q$ with $i>k$.
Then, $f_{k+1} (q)$ belongs to a face $\overline{G}$, by Assumption \ref{ipotesi_induttiva_mappe_bastarde}.
By the estimate in (a$_{k+1}$) and the assumption (b$_{k+1}$), the face $G$ must 
intersect $\mathscr{C} (F)$ and thus $F\subset \bar G$. Observe that, by the properties of $\Phi_\tau$ and
by the inductive assumption (b$_{k+1})$, $f_k (q)$ is mapped in the segment joining $f_{k+1} (q)$ and $q$ and
thus must belong to $\overline{G}$. 

    
\subsubsection{The inductive conclusions (c$_k$) and (b$_k$)} The first claim of (c$_k$) is simple to prove: as noticed,
if a point $q\in {\rm Dom} (f_{k})$ belongs also to ${\rm Dom} (f_{k+1})$, then $f_{k}$ maps it into the closure of the face containing $q$.
If the point is not contained in ${\rm Dom} (f_{k+1})$, then it must be contained in the $c_k$-neighborhood of some $k$-dimensional
face $F$ and hence it is mapped into $F$: when this happens $F$ is a portion of the boundary of the face containing $q$.
Next, fix a face $G\in \mathcal{F}_i$.
If $i= k$, by the very definition of $f_k$, we have that the restriction of $f_k$ to ${\rm Dom} (f_{k}) \cap G_{c_k, c_{k-1}}$
is the orthogonal projection onto $G$.
If $i>k$, we actually have that $f_k = f_{k+1}$ on ${\rm Dom} (f_k) \setminus (S_k)_{2\sqrt{c_k}}\supset \cQ \setminus (S_{i-1})_{c_{k-1}}$.

Fix now an $i$-dimensional face $L$ with $i\leq k$, consider coordinates $\R^i\times \R^{n-i}$ as in Definition \ref{d:coordinates}
and the corresponding $\mathscr{C} (L)$. If $q = (y, 0)\in L$, the condition (b$_k$) is equivalent to saying that $V_q 
\cap {\rm Dom} (f_k)$ gets mapped into $\{(y,0)\}\times \R^{N-i}$. Fix a point $\tilde{q}\in V_q$. If $f_{k+1} (\tilde{q})
= f_k (\tilde{q})$ there is nothing to prove. Otherwise it turns out that there is a $k$-dimensional face $F$ such that $\tilde{q}\in \mathscr{C}
(F)$. But then we necessarily have $L\subset \bar F$. So, set coordinates $\R^i\times \R^{k-i}\times \R^{n-k}$ so that 
at the same time $L\subset \R^i\times \{0\}\times\{0\}$ and $F\subset \R^{i}\times \R^{k-i} \times \{0\}$. Thus, $(y, 0,0)$ is the coordinate
of $q$ and $(y,z,w)$ that of $\tilde{q}$. According to our definition of $f_k$, $f_k (\tilde q) = (y,z,w')$ for some $w'$, which indeed
implies the desired claim. 

\subsubsection{$C^0$ estimate} Observe that, for every $x$ where $f_k$ coincides with $f_{k+1}$, we have
$|f_k (x) - x|\leq C c_{k+1} \leq C c_{k}^{8}$. Instead, for any point $x$
where $f_{k}$ is newly defined, we distinguish the following two cases: either $x = (y,z)$ with 
$|z| \leq c_k$, in which case $|f_k (x)-x| \leq c_k$; or $x = (y,z)$ with 
$|z| > c_k$, and then by the estimates of Lemma \ref{l:radial} and the triangle inequality we have
\begin{equation*}
|f_k (x)-x|\leq |f_{k+1} (x) -f_k (x)| + |f_{k+1} (x)-x| \leq C c_{k+1} + \tau \leq C c_{k+1} + 2c_k\, .
\end{equation*} 

\subsubsection{Lipschitz estimate} We fix $x,x'\in {\rm Dom} (f_k)$ and, apart from the trivial one $f_k (x)=f_{k+1} (x)$
and $f_k (x')= f_{k+1} (x')$, we  distinguish three cases.

\noindent {\bf Case 1: $x,x'\in G_{2 \sqrt{c_k}, c_{k-1}^2}$ for some $k$-dimensional face $G$.} 
Choosing coordinates as in Definition \ref{d:coordinates}, we set $x=(y,z)$ and $x'=(y',z')$. If both $|z|, |z'|\leq \frac{\tau}{2}$,
then $|f_k (x)- f_k (x')| = |y-y'|\leq |x-x'|$. If $|z|\geq \frac{\tau}{2}$ and $|z'|\geq \frac{\tau}{2}$, then
\begin{align*}
|f_k (x) - f_k (x')|^2 \leq &|y-y'|^2 + (1+ 2 \tau^{\sfrac{1}{2}})^2 |h^F_{k+1} (y,z) - h^F_{k+1} (y',z')|^2\nonumber\\
\leq& (1+  2 \tau^{\sfrac{1}{2}})^2 \left(|y-y'|^2 + |h^F_{k+1} (y,z) - h^F_{k+1} (y',z')|^2\right)\nonumber\\
=& (1+ \sqrt{2c_k})^2 |f_{k+1} (x) - f_{k+1} (x')|^2\leq (1+ \sqrt{2c_k})^2 (1+ C \sqrt{c_{k+1}})^2 |x-x'|^2\, . 
\end{align*}
If $|z|\leq \frac{\tau}{2}$ and $|z'| > \frac{\tau}{2}$, let $\tilde{z}$ be the point
with $|\tilde z| = \frac{\tau}{2}$ on the segment joining $z$ and $z'$, and 
$\tilde{x}= (y, \tilde{z})$. Observe that $f_k (\tilde{x}) = f_k (x) = (y, 0)$
and that $|\tilde{x}-x'|^2 = |y - y'|^2 + |z'-\tilde{z}|^2\leq |y-y'|^2 +|z-z'|^2 \leq |x-x'|^2$. On the other hand we have just shown $|f_k (x') - f_k (\tilde{x})|\leq (1+C c_k^{\sfrac{1}{2}}) |x'-\tilde{x}|$. 

\noindent {\bf Case 2: $x\in F_{2\sqrt{c_k}, c_{k-1}^2}$, $x'\in G_{2\sqrt{c_k}, c_{k-1}^2}$ for
distinct $F,G\in \mathcal{F}_k$.} 
By Lemma \ref{l:separation}, $|x-x'|\geq \bar{c} \,c^2_{k-1}\geq \bar{c} c_k^{\sfrac{1}{4}}$. On the other hand,
we also have, by the $C^0$ estimate,
\[
|f_k (x)-f_k (x')|\leq |x-x'| + C c_k \leq (1+ C c_k^{\sfrac{3}{4}}) |x-x'|\, .
\]

\noindent {\bf Case 3: $x\in G_{2 \sqrt{c_k}, c^2_{k-1}}$ for some $k$-dimensional face $G$ and $f_k (x') = f_{k+1} (x')$.} Without loss of generality we assume
\begin{itemize}
\item $G\in \mathcal{F}_k$;
\item $x'\not\in G_{2\sqrt{c_k}, c_{k-1}^2}$;
\item $x'\in H$ for some face $H$ (of dimension $i>k$).
\end{itemize}
We have two possibilities.

\noindent{\it Case 3a: $G\not\subset \overline{H}$.} Consider the closed
set $\tilde{G}:= G\setminus (S_{k-1})_{c_{k-1}^2}$. By Lemma \ref{l:separation}
$\dist (x', \tilde{G}) \geq \bar{c} c_{k-1}^2$ and thus $|x-x'|\geq \bar{c} c_{k-1}^2 - 2\sqrt{c_k}
\geq \frac{\bar{c}}{2}c_{k-1}^2$. We can therefore argue as in Case 2.

\noindent{\it Case 3b: $G\subset \overline{H}$.}
We then have two possibilities. The first is that $x\in {\rm Dom} (f_{k+1})$.
Since $f_k (x')= f_{k+1} (x')$, we have $|f_k (x') -x'|\leq C c_{k+1} =  C c_k^8$.
We use the coordinates of Definition \ref{d:coordinates} and (a$_{k+1}$) to conclude
$f_k (x') = f_{k+1} (y',z') = (y'', z'')$ with
$|z''|\geq |z'| - C\,c_k^8 \geq 2\sqrt{c_k} - C\,c_k^8 \geq \sqrt{2 c_k}$.
We can therefore write $f_k (x') = (y'', \Phi_\tau (z''))$
(because $\Phi_\tau (z'')=z''$)
and, hence, recalling $f_k (x) = (y, \Phi_\tau (h^F_{k+1} (y,z)))$ and $f_{k+1} (x) = (y, h^F_{k+1} (y,z))$,
\begin{align*}
&|f_k (x') - f_k (x)|^2 \leq |y - y''|^2 + (1+2\sqrt{\tau})^2 |h^F_{k+1} (y,z) - z''|^2\nonumber\\
\leq & (1+ 2 \sqrt{\tau})^2 |f_{k+1} (x)-f_{k+1} (x')|^2\, .
\end{align*}
We therefore conclude $|f_{k} (x')-f_{k} (x)|\leq (1+ C \tau^{\sfrac{1}{2}}) |x'-x| \leq (1+ C c_k^{\sfrac{1}{2}}) |x'-x|$.

The second possibility is that $x$ is not in the domain of definition of $f_{k+1}$. In that case $x$ is at distance $c_k$ from
$G$ and thus $|x-x'|\geq \sqrt{c_k}$. We then conclude that
$|f_{k} (x) - f_{k} (x')| \leq |x-x'| + C\,c_k \leq (1+ C\sqrt{c_k}) |x-x'|$.

\subsubsection{Summary} After $nQ$ steps, we get a function $f_0=\ro^\flat:\cQ\to\cQ$ which satisfies
\begin{align}
&\Lip(\ro^\flat)\leq 1+C\,\delta^{8^{-nQ-1}}
\quad\text{and}\quad
|\ro^\flat(x)-x|\leq C\,\delta^{8^{-nQ}},\label{e:bemolle_1}\\
&\ro^\flat (\{x:\dist (x,F)\leq \delta\})\subset \overline{F}\quad\mbox{for every $F\in \mathcal{F}_k$},\\
&\ro^\flat: F_{\delta, c_0^{\sfrac{1}{8}}}\to F \mbox{ is the orthogonal projection on $F$ for every $F\in \mathcal{F}_k$.}\label{e:bemolle_3}
\end{align}

\subsection{The extension $\ro^\sharp$ of $\ro^\flat$ to $\cQ_{\delta^{nQ+1}}$}\label{sss:ro*2}
Next we extend the map $\ro^\flat:\cQ\to\cQ$ to the $\delta^{nQ+1}$-neighborhood
of $\cQ$, keeping the estimate \eqref{e:bemolle_1}.
We first observe that, since the number of all the faces is finite,
when $\delta$ is small enough, there exists a constant $C=C(N)$ 
with the following property.
Consider two distinct faces $F$ and $H$ in $\cF_i$. If $x,y$ are two points contained, 
respectively, in $F_{\delta^{i+1}}\setminus \cup_{j<i}\cup_{G\in\cF_j}G_{\delta^{j+1}}$
and $H_{\delta^{i+1}}\setminus \cup_{j<i}\cup_{G\in\cF_j}G_{\delta^{j+1}}$, then
\begin{equation}\label{e:conflit}
\dist (x,y) \;\geq\; C\,\delta^{i}.
\end{equation}
Similarly if $F \in \cF_l$ and $H \in \cF_i$ with $l<i$ and $F \not\subset \bar H$, then for every
$x \in F_{\delta^{l+1}}$
and $y\in H_{\delta^{i+1}}\setminus \cup_{j<i}\cup_{G\in\cF_j}G_{\delta^{j+1}}$ it holds
\begin{equation}\label{e:conflit2}
\dist (x,y) \;\geq\; C\,\delta^{i}.
\end{equation}
The extension $\ro^\sharp$ is defined 
inductively, but this time ``from the bottom to the top''.
The first extension $g_0$ is identically $0$ on $B_\delta (0)$ (note that this is feasible because 
$\ro^\flat \equiv 0$ in
$B_\delta (0)\cap \cQ$).
Now we come to the inductive step. Suppose we have an extension 
$g_\ell$ of $\ro^\flat$,
defined on the union of the $\delta^{\ell+1}$-neighborhoods of the 
$\ell$-skeletons $S_\ell$, for $\ell\in \{0, \ldots, k\}$, i.e.
\[
\mathbf{L}_k \;:=\;
\cQ\cup B_{\delta}(0)\cup\bigcup_{\ell=1}^k\bigcup_{F\in\cF_\ell} F_{\delta^{\ell+1}}\,.
\]
Assume inductively that $\Lip(g_k)\leq 1+C\,\delta^{8^{-nQ-1}}$ and assume that $g_k$ maps any $\delta^{j+1}$-neighborhood of any
$j$-dimensional face into its closure, when $j\leq k$.
Then, we define the extension of $g_k$ to
$\mathbf{L}_{k+1}$ in the following way. For every face
$F\in\cF_{k+1}$, we set
\begin{equation}\label{e:casi}
g_{k+1}:=
\begin{cases}
\ro^\flat & \textrm{on} \quad  \cQ\, ,\\
g_k &\textrm{on} \quad (S_k)_{\delta^{k+1}}\cap F_{\delta^{k+2}},\\
\p_F & \textrm{on} \quad \{x\in\R^{N}\,:\, \p_F (x)\in F_{\delta, 1}\}\cap F_{\delta^{k+2}}\,,
\end{cases}
\end{equation}
where $\p_F$ stands for the orthogonal projection on $F$ (recall that by \eqref{e:bemolle_3} $\ro^\flat = \p_F$ on $F\cap F_{\delta^{k+2}, 1}$).
Consider now a face $F$ as above and $U (F)$ the union of all the $\delta^{j+1}$-neighborhoods of the $j$-dimensional
faces which belong to $\overline{F}$.
As defined above, $g_{k+1}$
maps a portion of $U (F)$ into $\overline{F}$. 
We can use Lemma \ref{l:kirsz} to extend $g_{k+1}$
to $U(F)$ keeping
the same Lipschitz constant, which we now compute. This constant is obviously smaller than $1+ C \delta^{8^{-nQ-1}}$ on the domain
$((S_k)_{\delta^{k+1}}\cap F_{\delta^{k+2}})\cup F$ by inductive hypothesis. The same constant is $1$ on $\{x\in\R^{N}\,:\, \p_F (x)\in F_{\delta, 1}\}\cap F_{\delta^{k+2}}$.
Consider now a point $x\in \{x\in\R^{N}\,:\, \p_F (x)\in F_{\delta, 1}\}\cap F_{\delta^{k+2}}$ and a point
$y\in F \cup ((S_k)_{\delta^{k+1}}\cap F_{\delta^{k+2}})$. If $y\not\in (S_{k})_{c_0^{\sfrac{1}{8}}}$, then necessarily $y\in F$ and we then
have
\[
|g_{k+1} (x) - g_{k+1} (y)| = |\p_F (x) - y| = |\p_F (x) - \p_F (y)| \leq |x-y|.
\]
Otherwise we have
$|x-y|\geq 1- c_0^{\sfrac{1}{8}} = 1- \delta^{8^{-nQ-1}}$ and we can write
\begin{align*}
|g_{k+1} (x) - g_{k+1} (y)|
\leq & |g_{k+1} (x) - y| + C c_0 \leq |x-y| + \delta^{k+2} + C c_0\\
\leq & \left(1+ \frac{\delta^{k+2} + C c_0}{1- C c_0^{\sfrac{1}{8}}}\right) |x-y|\leq (1+ C \delta^{8^{-nQ-1}}) |x-y|.
\end{align*}

Note that, if $x\in U(F_1)\cap U(F_2)$ for two distinct $F_1, F_2 \in \cF_{k+1}$, 
then $x\in \mathbf{L}_k$.
Thus, the map $g_{k+1}$ is continuous. We next bound
the global Lipschitz constant of $g_{k+1}$.
Indeed consider points
$x\in U(F_1) \setminus U(F_2)$ and $y\in U(F_2) \setminus U(F_1)$
for two distinct $F_i\in \mathcal{F}_{k+1}$.
Since by \eqref{e:conflit} and \eqref{e:conflit2} $|x-y|\geq C\,\delta^{k+1}$,
we easily see that
\begin{align*}
|g_{k+1}(x)-g_{k+1}(y)|&\leq |g_{k+1}(x) - g_{k+1}(\p_{F_1}(x))|+ |g_{k+1}(\p_{F_1}(x))-g_{k+1}(\p_{F_2}(y))|\\
&\quad + |g_{k+1}(\p_{F_2}(y))- g_{k+1}(y)|\\
& \leq 2(1+C \delta^{8^{-nQ-1}})\,\delta^{k+2}+|\ro^\flat (\p_{F_1}(x))-\ro^\flat (\p_{F_2}(y))|\notag\\
&\leq 2(1+C \delta^{8^{-nQ-1}})\,\delta^{k+2}+(1+C\,\delta^{8^{-nQ-1}})|\p_{F_1}(x)-\p_{F_2}(y)|\notag\\
&\leq 2(1+C \delta^{8^{-nQ-1}})\,\delta^{k+2}+(1+C\,\delta^{8^{-nQ-1}})
\Big(|x-y|+2\,\delta^{k+2}\Big)\\
&\leq (1+C\,\delta^{8^{-nQ-1}})\,|x-y|.
\end{align*}
Next, consider the case $x\in \cQ\setminus U(F), y\in U (F)$. If $|x-y|\geq \delta^{k+1}$, we can then argue as above and 
(considering that $g_{k+1} (x) = \ro^\flat (x)$) we bound
\begin{align*}
& |g_{k+1}(x)-g_{k+1}(y)|\leq (1+C \delta^{8^{-nQ-1}})\,\delta^{k+2}+|\ro^\flat (x)-\ro^\flat (\p_{F}(y))|\notag\\
&\leq (1+C \delta^{8^{-nQ-1}})\,\left(\delta^{k+2}+|x-\p_{F}(y)|\right)
\leq (1+C \delta^{8^{-nQ-1}})\,\left(\delta^{k+2}+ |x-y|+\delta^{k+2}\right)\notag\\
&\leq (1+C\,\delta^{8^{-nQ-1}})\,|x-y|.
\end{align*}
We therefore assume $|x-y|\leq \delta^{k+1}$. Observe also that, if $y\not \in \{x\in\R^{N}\,:\, \p_F (x)\in F_{\delta, 1}\}\cap F_{\delta^{k+2}}$, then $g_{k+1} (y) = g_k (y)$ and since
$g_{k+1} (x) = \ro^\flat (x) = g_k (x)$, we know the Lipschitz bound by inductive assumption. We therefore conclude that $x\in F_{\delta^{k+2}+ \delta^{k+1}, 1 - \delta^{k+1}}$. Assuming $\delta_0$ small enough, $\delta^{k+2} + \delta^{k+1} \leq \delta$ and 
$1 - \delta^{k+1} \geq \delta^{8^{-nQ -1}} = c_0^{\sfrac{1}{8}}$, therefore $x\in F_{\delta, c_0^{\sfrac{1}{8}}}$. By \eqref{e:bemolle_3} we then have $|g_{k+1} (x) - g_{k+1} (y)| = |\p_F (x) - \p_F (y)| \leq |x-y|$.

Since $\cQ$ and the union of the $U (F_i)$ is the domain of definition of $g_{k+1}$, this shows 
$\Lip(g_{k+1})\leq 1+C\,\delta^{8^{-nQ-1}}$.
Note that by construction we also have that $U(F)$ is mapped into $\overline{F}$, which is the other inductive hypothesis. 

After making the step above $nQ$ times we arrive to a map
$g_{nQ}$ which extends $\ro^\flat$ and is defined in a 
$\delta^{nQ+1}$-neighborhood of $\cQ$. This is the map
$\ro^\sharp$.

\section{Persistence of $Q$-points: Proof of Theorem \ref{t:persistence}}\label{s:corol}

\begin{proof}[Proof of Theorem \ref{t:persistence}]
As usual, by scaling and translating we assume $x=0$ and $r=1$. According to \cite[Theorem 3.9]{DS1}, there are constants $\bar{C} (m,n,Q),\kappa (m,n,Q)>0$ such that
\begin{equation}\label{e:Hold_est}
\sup_{x \neq y\in B_{1/2}} \frac{\cG (w(x), w(y))}{|y-x|^\kappa} \leq \bar{C} (\D (w))^{\frac{1}{2}}\quad
\mbox{for any $\D$-minimizer $w: B_1 \to \Iqs$.} 
\end{equation}
The final choice of $\bar{s}$ will be specified at the very end, but for the moment we impose $\bar{s}<\textstyle{\frac{1}{4}}$. 
Fix now $s<\bar{s}$ and $C^\star$ as in the statement 
and assume by contradiction that, no matter how small we choose $\hat{\eps}>0$, there are a current $T$ 
and a submanifold $\Sigma$ as in Theorem \ref{t:main} and a point $(p,q)\in \bC_{1/2}$ satisfying:
\begin{itemize}
\item[(a)] $E := \bE (T, \bC_4)< \hat{\eps}$ and $\bA^2 \leq C^\star E$; 
\item[(b)] $\Theta (T, (p,q)) = Q$;
\item[(c)] the $E^{\gamma_1}$-approximation $f$ (which is the map of Theorem \ref{t:main}) violates \eqref{e:persistence}, that is
\begin{equation}\label{e:no_persistence}
\int_{B_s (p)} \cG (f, Q \a{\etaa\circ f})^2 > \hat{\delta} s^m E\, .
\end{equation}
\end{itemize}
Set $\bar{\delta} = \frac{1}{4}$ and fix $\bar{\eta}>0$ (whose choice will be specified later). By (a), for
a suitably small $\hat{\eps}$ we can apply Theorem \ref{t:harmonic_final} in the
coordinates of Remark \ref{r:Psi}:  we let $u$ be the corresponding $\D$-minimizer and $w = (u, \Psi (x, u))$. If $\bar{\eta}$ and $\hat{\eps}$ are suitably small, we have
\[
\int_{B_s (p)} \cG (w, Q\a{\etaa\circ w})^2 \geq \textstyle{\frac{3\hat\delta}{4}} s^m E\, ,
\]
and $\sup\big\{\D (f), \D (w)\} \leq  C E$ (here we use Remark~\ref{r:prime stime}). Thus there is 
$\bar{p}\in B_s (p)$ with $\cG (w (\bar p), Q \a{\etaa \circ w (\bar p)})^2 \geq \frac{3\hat\delta}{4\omega_m}\,E$ and, by \eqref{e:Hold_est}, we conclude
\begin{equation}\label{e:imponi}
g (x) := \cG (w (x), Q\a{\etaa\circ w (x)}) \geq \left(\textstyle{\frac{3\hat\delta}{4\omega_m}} E\right)^{\sfrac{1}{2}} - 
2\,(C E)^{\sfrac{1}{2}} \bar{C} \bar{s}^{\kappa} \geq \left(\textstyle{\frac{\hat\delta}{2}} E\right)^{\sfrac{1}{2}}\,,
\end{equation}
where we assume that $\bar s$ is chosen small enough in order to satisfy the last inequality.
Setting $h (x):= \cG (f (x), Q\a{\etaa\circ f (x)})$, we recall that we have
\[
\int_{B_s (p)} |h-g|^2 \leq C\, \bar{\eta} E\, .
\]
Consider therefore the set $A:= \big\{h > \big(\frac{\hat \delta}{4} E\big)^{\sfrac{1}{2}}\big \}$. If $\bar \eta$ is sufficiently small, we can assume that 
$|B_s(p)\setminus A| < \frac{1}{8} |B_s|$. Further, define $\bar{A}:= A\cap K$, where $K$ is the set of Theorem \ref{t:main}. Assuming $\hat{\eps}$ is sufficiently small we ensure $|B_s(p)\setminus \bar A| < \frac{1}{4} |B_s|$. Let $N$ be the smallest integer such that $N \frac{\hat \delta E}{64 Q s} \geq \frac{s}{2}$. Set $\sigma_i := s- i \frac{\hat \delta E}{64Q s}$ for $i\in \{0, 1\ldots, N\}$ and consider, for $i\leq N-1$,  the annuli $\cC_i:= B_{\sigma_i} (p) \setminus B_{\sigma_{i+1}} (p)$. If $\hat{\eps}$ is sufficiently small, we can assume that $N\geq 2$ and
$\sigma_N \geq \frac{s}{4}$.
For at least one of these annuli we must have $|\bar A\cap \cC_i|\geq \frac{1}{2} |\cC_i|$. We then let $\sigma:= \sigma_i$ be the corresponding outer radius and we denote by $\cC$ the corresponding annulus.

Consider now a point $x\in \cC \cap \bar{A}$ and let $T_x$ be the slice $\langle T, \p, x \rangle$. Since $\bar{A}\subset K$, for a.e.~$x\in \bar{A}$ we have $T_x = \sum_{i=1}^Q \a{(x,f_i (x))}$. Moreover, there exist $i$ and $j$ 
such that $|f_i (x)-f_j (x)|^2\geq \frac{1}{Q} \cG (f (x), \a{\etaa \circ f (x)})^2 \geq \frac{\hat \delta}{4Q} E$ (recall
that $x\in \bar{A}\subset A$). When $x\in \cC$ and the points $(x,y)$ and $(x,z)$ belong both to $\bB_\sigma ((p,q))$, we must have 
\[
|y-z|^2 \leq 4 \left(\sigma^2 - \Big(\sigma - \textstyle{\frac{\hat{\delta} E}{64 Q s}}\right)^2\Big) \leq 
\textstyle{\frac{\sigma \hat\delta E}{8Q s}} \leq \textstyle{\frac{\hat\delta E}{8Q}}\, .
\]
Thus, for $x\in \bar{A}\cap \cC$ at least one of the points $(x, f_i (x))$ is not contained in $\B_\sigma ((p,q))$. We conclude therefore 
\begin{align}
\|T\| (\bC_{\sigma} (p) \setminus \bB_\sigma ((p,q))) &\geq |\cC\cap \bar{A}| \geq \frac{1}{2} |\cC| = 
\frac{\omega_m}{2} \left( \sigma^m - \left(\sigma -  \textstyle{\frac{\hat{\delta} E}{64 Qs}}\right)^m \right)\nonumber\\
&\geq \frac{\omega_m}{2} \sigma^m \left( 1- \left(1-  \textstyle{\frac{\hat{\delta} E}{64 Qs \sigma}}\right)^m \right)\, .
\end{align}
Recall that, for $\tau$ sufficiently small, $(1-\tau)^m \leq 1 - \frac{m\tau}{2}$.
Since $\sigma \geq \frac{s}{4}$, if $\hat{\eps}$ is chosen sufficiently small we can therefore conclude
\begin{equation}\label{e:punti_persi}
\|T\| (\bC_{\sigma} (p) \setminus \bB_\sigma (p)) \geq \frac{\omega_m \sigma^m \hat{\delta} E}{256 Q s \sigma}
\geq \frac{\omega_m}{1024 Q} \hat{\delta} E  \sigma^{m-2}= c_0 \hat{\delta} E \sigma^{m-2}\, .
\end{equation}
Next, by Theorem \ref{t:main} and Theorem \ref{t:harmonic_final},
\begin{equation}\label{e:stima_nel_cilindro}
\|T\| (\bC_{\sigma} (p)) \leq Q \omega_m \sigma^m + C E^{1+\gamma_1} +\bar{\eta} E + \int_{B_\sigma (p)} \frac{|Dw|^2}{2}\, .
\end{equation}
Moreover, as shown in \cite[Section 3.3]{DS1} (cf.~\cite[Proposition 3.10]{DS1}), we have
\begin{equation}\label{e:decay}
\int_{B_\sigma (p)} |Dw|^2 \leq \|D\Psi\|^2 \sigma^m + C \int_{B_\sigma (p)}  |Du|^2 \leq C (1 + C^\star) E \sigma^m + C \D (u) \sigma^{m-2+2\kappa},
\end{equation}
(for some constants $\kappa$ and $C$ depending only on $m$, $n$ and $Q$; in fact the exponent
$\kappa$ is the one of \eqref{e:Hold_est}). Combining \eqref{e:punti_persi},
\eqref{e:stima_nel_cilindro} and \eqref{e:decay}, we conclude
\begin{equation}\label{e:da_sopra}
\|T\| (\bB_\sigma ((p,q))) \leq Q \omega_m \sigma^m + (\bar{\eta} + C(1+ C^\star) \sigma^m) E + C E^{1+\gamma_1} + C E \sigma^{m-2+2\kappa} - c_0 \sigma^{m-2} \hat\delta E\, .
\end{equation}
Next, by the monotonicity formula, $\rho\mapsto \exp (C \bA^2 \rho^2) \rho^{-m} \|T\| (\bB_\rho ((p,q)))$ is a monotone function 
(indeed, the usual monotonicity formula of the theory of varifolds
with bounded mean curvature gives the monotonicity of 
$\rho \mapsto \exp (C\bA \rho) \rho^{-m} \|T\| (\bB_\rho ((p,q)))$, cf.~\cite[Theorem 17.6]{Sim}); the slight improvement needed in this proof follows from minor modifications of the usual argument but, since we have not been able to find a reference, we provide a proof in Lemma~ \ref{l:monot} in the appendix). Using $\bA^2 \leq C^\star E$, $\Theta(T,(p,q)) =Q$ and the Taylor expansion of the exponential, we conclude
\begin{equation}\label{e:da_sotto}
\|T\| (\bB_\sigma ((p,q))) \geq Q \omega_m \sigma^m - C C^\star E \sigma^{m+2}\, .
\end{equation}
Combining \eqref{e:da_sopra} and \eqref{e:da_sotto} we conclude 
\begin{equation}\label{e:basso_alto_10}
C (1+ C^\star) \sigma^2 + (\bar{\eta} + C E^\gamma_1) \sigma^{2-m} + C \sigma^{2\kappa} \geq c_0 \hat{\delta}\, .
\end{equation}
Recalling that $\sigma \leq s < \bar{s}$, we can, finally, specify $\bar{s}$: it is chosen so that
$C (1+ C^\star) \bar{s}^2 + C \bar{s}^{2\kappa}$ is smaller than $\frac{c_0}{2} \hat{\delta}$. Combined with
\eqref{e:imponi} this choice of $\bar{s}$ depends, therefore, only upon $\hat{\delta}$. \eqref{e:basso_alto_10}
becomes then
\begin{equation}\label{e:basso_alto_11}
(\bar{\eta} + C E^{\gamma_1}) \sigma^{2-m}  \geq \textstyle{\frac{c_0}{2}} \hat{\delta}\, .
\end{equation}
Next, recall that $\sigma\geq \frac{s}{4}$. We then choose $\hat{\eps}$ so that
$(\bar{\eta} + C \hat{\eps}^{\gamma_1}) (\frac{s}{4})^{2-m} \leq \frac{c_0}{4} \hat{\delta}$.
This choice is incompatible with \eqref{e:basso_alto_11}, thereby reaching a contradiction:
for this choice of the parameter $\hat{\eps}$ (which in fact depends only upon $\hat{\delta}$
and $s$) the conclusion of the Theorem, i.e. \eqref{e:persistence}, must then be valid.
\end{proof}

\appendix

\section{Monotonicity formula}

\begin{lemma}\label{l:monot}
There is a constant $C$ depending only on $m$, $n$ and $\bar{n}$ with the following property. If $\Sigma\subset \R^{m+n}$ is a $C^2$ $(m+\bar{n})$-dimensional submanifold with $\|A_\Sigma\|_\infty \leq \bA$ and $T$ an $m$-dimensional integer-rectifiable current supported in $\Sigma$ which is stationary in $\Sigma$, then for every $\xi\in \Sigma$ the function $\rho\mapsto \exp (C \bA^2 \rho^2) \rho^{-m} \|T\| (\bB_\rho (\xi))$ is monotone on the interval $]0, \bar{\rho}[$, where $\bar{\rho} := \min \{\dist (x, \supp (\partial T)), (C \bA)^{-1}\}$.
\end{lemma}

\begin{proof}
The argument is a minor variant of the classical proof of the monotonicity formula for varifolds with bounded mean curvature due to Allard (cf.~\cite{All}).
Here the stronger hypothesis that $T$ is stationary in a $C^2$-submanifold allows a better estimate of the relevant
error term.  Without loss of generality assume $\xi=0$, let $s\in ]0, \bar{\rho}[$ and $\varphi\in C^1_c (]-1,1[)$ with $\varphi \equiv 1$ in a neighborhood of $0$. For each $x\in \Sigma$ let $\p_x: \R^{m+n}\to T_x \Sigma$ be the orthogonal projection onto the tangent space to $\Sigma$ in $x$ and consider the vector field
$X_s (x) := \varphi (\frac{|x|}{s}) \p_x (x)$. Note that $X_s$ is tangent to $\Sigma$ and thus $\delta T (X_s) =0$. In order to compute $\delta T (X_s)$, consider at $\|T\|$-a.e. $x\in \supp (T)$ an orthonormal frame $e_1, \ldots, e_m$ with $e_1\wedge \ldots \wedge e_m = \vec{T}$. It turns out that 
\[
\delta T (X_s) = \int {\rm div}_{\vec{T}} X_s\, d\|T\| = \int \sum_i \langle D_{e_i} X_s, e_i\rangle\, d\|T\|\, .
\]
Next, at any $x\in \Sigma$ let $\nu_1, \ldots, \nu_l$ ($l=n-\bar{n}$) be an orthonormal frame orthogonal to $\Sigma$. Since $\p_x (x) = x - \sum_j \langle x, \nu_j\rangle \nu_j$ and $\langle e_i, \nu_j \rangle = 0$,
we compute:
\begin{align}
{\rm div}_{\vec{T}} X_s (x) & = \underbrace{\sum_i \left[D_{e_i} \left( \varphi \left(\textstyle{\frac{|x|}{s}}\right)\right) \langle x, e_i \rangle
+  \varphi \left(\textstyle{\frac{|x|}{s}}\right) \langle D_{e_i} x, e_i \rangle\right]}_{I} - 
\underbrace{\varphi \left(\textstyle{\frac{|x|}{s}}\right) \sum_{i,j} \langle x, \nu_j \rangle \langle D_{e_i} \nu_j, e_i\rangle}_{II}\, .\nonumber
\end{align}
$I$ is the usual expression appearing in the proof of the standard monotonicity formula for
stationary varifolds. If we use the notation $r$ for the function $x\mapsto |x|$ and $\nabla^\perp r$
for the orthogonal projection on the orthogonal complement of ${\rm{Span}}\{e_1, \ldots, e_m\}$,
we find $I=m \varphi (\frac{r}{s}) + \frac{r}{s} \varphi' (\frac{r}{s}) (1- |\nabla^\perp r|^2)$ (see for instance \cite[(2.2)]{DL-All}).  In order to bound $II$, we first observe that $\langle D_{e_i} \nu_j, e_i \rangle = - \langle A(e_i,e_i), \nu_j\rangle$. Next, since $r\leq (C \bA)^{-1}$, if $C$ is chosen sufficiently large we can assume that the geodesic segment of $\Sigma$ connecting $0$ and $x$ has length $\ell < 2r$. Denote by $\gamma : [0,  \ell] \to \Sigma$ a parametrization by arc-length of such a segment. Then,
\begin{equation}\label{e:stimella}
\langle x, \nu_j (x)\rangle = \int_0^\ell \langle\dot{\gamma} (\sigma), \nu_j (\gamma (\ell))\rangle\, d\sigma = \int_0^\ell \underbrace{\langle \dot{\gamma} (\sigma) , \left[\nu_j (\gamma (\ell)) - \nu_j (\gamma (\sigma))\right]\rangle}_{g (\sigma)}\, d\sigma\, ,
\end{equation}
and observe that
\[
|g' (\sigma)|\leq 2 \left|\frac{d}{d\sigma} \dot{\gamma} (\sigma)\right|  + \left|\langle \dot{\gamma} (\sigma), 
\frac{d}{d\sigma} \nu_j (\gamma (\sigma)) \rangle \right| \leq 3 |A (\dot{\gamma} (\sigma), \dot{\gamma} (\sigma))|.
\]
Since $g (\ell )=0$, integrating the latter inequality we conclude $|g(\sigma)|\leq 3 \ell \bA \leq 6 r \bA$, 
which in turn, together with \eqref{e:stimella}, gives  $|x\cdot \nu_j (x)|\leq 12 r^2 \bA$. 

Putting all estimates together, we achieve the inequality $|II| \leq C \varphi (\frac{r}{s}) r^2 \bA^2$. From here
on we can follow the usual strategy leading to the monotonicity formula (cf.~\cite{Sim} or 
\cite[Proof of Theorem 2.1]{DL-All}): letting the test function $\varphi$ converge from below to the indicator 
function of $]-1,1[$, after few manipulations we achieve the inequality
\[
\frac{d}{ds} \frac{\|T\| (\bB_s)}{s^m} \geq - C \bA^2 s \frac{\|T\| (\bB_s)}{s^m}\, ,
\]
which leads to the desired claim.
\end{proof}

\begin{remark}
The proof can be easily extended to varifolds which are stationary in $\Sigma$. In fact the argument
above can be considerably shortened using directly the Monotonicity Formula of Section 5 in \cite{All}.
\end{remark}

\bibliographystyle{plain}
\bibliography{Lp-references}

\end{document}